\documentclass[reqno]{amsart}

\usepackage{etoolbox}
\numberwithin{equation}{section}
\usepackage[foot]{amsaddr}
\patchcmd{\abstract}{\null\vfil}{}{}{}

\usepackage[dvipsnames]{xcolor}
\usepackage{smartdiagram}
\usepackage{tikz}
\usepackage{amsmath,bm,amsthm, amssymb}
\usepackage{fullpage}
\usepackage{color}
\usepackage{animate}
\usepackage{etoolbox}
\usepackage{comment}
\usepackage{pgf}
\usetikzlibrary{calc}
\usetikzlibrary{patterns}
\usetikzlibrary{arrows}
\usetikzlibrary{decorations.pathreplacing}
\usepackage[utf8]{inputenc}
\usepackage{pgfplots}

\usepackage[most]{tcolorbox}
\usepackage{adjustbox}
\usepackage[final]{pdfpages}
\usepackage{wrapfig}
\usepackage{ marvosym }
\usepackage{comment}
\usepackage{mathrsfs}
\usepackage{changepage}
\usepackage{booktabs}
\usepackage{tabularx}

\usepackage{subcaption}

\usepackage{enumitem}

\theoremstyle{plain}
\newtheorem{theorem}{Theorem}
\newtheorem{proposition}[theorem]{Proposition}
\newtheorem{corollary}[theorem]{Corollary}
\newtheorem{lemma}[theorem]{Lemma}

\theoremstyle{definition}
\newtheorem{definition}[theorem]{Definition}

\newtheorem{example}[theorem]{Example}

\theoremstyle{remark}
\newtheorem{remark}[theorem]{Remark}

\usepackage{xcolor}
\usepackage[most]{tcolorbox}

\newtcolorbox{manishbox}{
  colback=red!5!white,       
  colframe=red!75!black,     
  title=\textbf{Manish:},    
  fonttitle=\bfseries,       
  left=2mm, right=2mm, top=1mm, bottom=1mm,
  boxrule=0.8pt,
  sharp corners,
  enhanced,
}

\definecolor{wiasblue}   {cmyk}{1.0, 0.60, 0, 0}
\definecolor{mlugreen}{RGB}{172,6,52}
\definecolor{darmstadt}{RGB}{135,206,250}

\def\Q{\mathbb Q}
\def\R{\mathbb R}
\def\T{\mathbb T}

\def\mc{\mathcal}
\def\ms{\mathsf}

\def\su{\subseteq}

\def\g{\gamma}

\def\de{\delta}

\def\one{\mathbf{1}}

\def\G{\Gamma}

\def\YY{\mc Y}

\def\PP{\mc P}

\def\CC{\mc C}

\def\bep{\begin{proof}}
\def\enp{\end{proof}}
\def\bepr{\begin{proposition}}
\def\enpr{\end{proposition}}
\def\bec{\begin{corollary}}
\def\enc{\end{corollary}}
\def\bea{\begin{align}}
\newcommand\eea{\end{align}}
\def\beas{\begin{align*}}
\def\eeas{\end{align*}}
\def\bet{\begin{theorem}}
\def\ent{\end{theorem}}
\def\bee{\begin{example}}
\def\ene{\end{example}}

\def\bede{\begin{definition}}
\def\ende{\end{definition}}
\def\ber{\begin{remark}}
\def\enr{\end{remark}}
\def\beca{\begin{cases}}
\def\enca{\end{cases}}
\def\bel{\begin{lemma}}
\def\enl{\end{lemma}}
\def\been{\begin{enumerate}}
\def\enen{\end{enumerate}}

\def\beit{\begin{itemize}}
\def\enit{\end{itemize}}
\def\befr{\begin{frame}}
\def\enfr{\end{frame}}

\def\diam{\ms{diam}}

\renewcommand\le{\leqslant}
\renewcommand\ge{\geqslant}

\def\dist{\ms{dist}}

\def\becbb{\begin{center}\begin{tcolorbox}[{colback=Dandelion!20}]}
\def\encbb{\end{tcolorbox}\end{center}}
\def\beccb{\begin{center}\begin{tcolorbox}[{colback=Dandelion!20}]}
\def\enccb{\end{tcolorbox}\end{center}}
\def\becb{\begin{center}\begin{tcbox}[{colback=Dandelion!20}]}
\def\encb{\end{tcbox}\end{center}}

\def\bef{\begin{figure}[!h]}
\def\enf{\end{figure}}
\def\bels{\begin{lstlisting}}
\def\enls{\end{lstlisting}}

\def\betp{\begin{tikzpicture}}
\def\entp{\end{tikzpicture}}

\def\endo{\end{document}}

\usepackage{placeins}

\usetikzlibrary{arrows.meta,calc,angles,quotes,decorations.pathreplacing}
\pgfplotsset{compat=1.18}

\begin{document}

\title{Large deviations for sparse systems of moving particles}

\author{Cairui Duan}

\author{Paula De Dios Andres}

\author{Manish Pandey}

\author{Miguel A. Ramos Docampo}

\author{Brigitte St\"adler}
\address[Cairui Duan, Manish Pandey, Christian Hirsch]{Department of Mathematics, Aarhus University,
Ny Munkegade 118, DK-8000 Aarhus C, Denmark}
\address[Paula de Dios Andres, Miguel A. Ramos Docampo, Brigitte St\"adler]{Interdisciplinary Nanoscience Center, Aarhus University,
Gustav Wieds Vej 14, DK-8000 Aarhus C, Denmark}

\author{Christian Hirsch}

\begin{abstract}
We study large deviations for rare clusters in sparse systems of moving
particles. In the regime \(nr_n^d\to0\) and
\(\rho_{k,n}=n^kr_n^{d(k-1)}\to\infty\), we prove a large deviation principle
for the empirical measure of isolated \(k\)-particle trajectory clusters. The
speed is \(\rho_{k,n}\), and the rate function is the relative entropy
\(h(\,\cdot\,\mid\tau_k)\), where the finite reference measure \(\tau_k\)
explicitly incorporates the underlying path law. As consequences, we derive
free-energy variational formulas for bounded interactions and a hard-core
constraint and identify the corresponding minimizing cluster law. The
normalized optimizer provides the basis for Metropolis--Hastings sampling of
interacting trajectory clusters. As an application motivated by chain
formation and swarming in active particle systems, we calibrate the resulting
stochastic model to experimental trajectories of magnetic micromotors and find
that the fitted velocity scale varies systematically with particle size and
magnetic forcing.
\end{abstract}
\maketitle
\thispagestyle{empty}

\frenchspacing

\section{Introduction}
\label{sec:intro}

Large deviation analysis of continuum many-body systems is a fundamental topic in statistical physics. For interacting particles placed in space, large deviation theory quantifies the entropy-energy balance that determines the most likely realization of an atypical event; see, for example, \cite{sidorova,jansen2,jkm15,gtf}. While these works cover a broad range of regimes, they concern static particle systems. In many physical settings, however, particles move according to an underlying stochastic or deterministic dynamics. The aim of this paper is to study the large deviation behavior of rare clusters of such moving particles in a dilute regime. Our state variable is the empirical measure of rare spatial configurations of several trajectories.

The key challenge is that rare collective structures in dilute systems of
moving particles are governed by entire trajectories rather than by a single
spatial configuration. Retaining the full trajectories changes both the
limiting cluster law and the localization argument needed to obtain it. Our
main probabilistic result identifies the resulting path-valued large
deviation principle. Its variational consequences then provide the link to
the computational and experimental parts of the paper.

We note that the study of rare clusters of trajectories is also of interest in the context of continuum Gibbs point processes. More precisely, R\oe lly and Zass~\cite{roelly-zass} established existence results for marked Gibbs point processes with unbounded marks and nonuniform
interaction ranges. Zass~\cite{zass} subsequently developed existence,
cluster expansion, and uniqueness results for Gibbs point processes whose
marks are diffusion trajectories. More recently, Jahnel, K\"oppl, Steenbeck,
and Zass~\cite{jahnel-koppl-steenbeck-zass} proved an infinite-volume Gibbs
variational principle for a marked continuum model with unbounded interaction
range. These works study the construction and equilibrium characterization
of infinite-volume Gibbs states. By contrast, we work in a sparse asymptotic
regime and identify the exponential cost, favored shape distribution, and
computational representation of rare finite clusters of trajectories.

The application is motivated by active matter and the design of micro- and
nanoscale particle systems. Micro- and nanomotors convert chemical or
externally supplied energy into directed or persistent motion that can
dominate Brownian diffusion over relevant length and time scales
\cite{bishop:biswal:bharti:2023,ju:etal:2025}. Beyond the motion of
individual particles, interactions between motors and coupling to external
fields can generate collective states such as chains, clusters, dynamically
reconfigurable assemblies, and coordinated swarms
\cite{ramosdocampo:nieto:dediosandres:qian:stadler:2023,
urso:ussia:peng:oral:pumera:2023,yan:etal:2024,
heuthe:panizon:gu:bechinger:2024,xu:ge:xu:2025}. 
More specifically, motors endowed with magnetic features can respond to
externally applied magnetic fields
\cite{ramosdocampo:2024,shen:etal:2023,liu:etal:2023}, giving rise to a variety of spatial
configurations while moving.
Experiments show that these states depend
sensitively on particle geometry and magnetic anisotropy, particle
concentration, interparticle interactions, and the strength and temporal
protocol of the applied fields \cite{yan:etal:2024,cheng:etal:2025,
haque:maestas:zhu:hanson:wu:wu:2025,
camacho:devicente:2025}.

Consequently, micromotor experiments can produce large and heterogeneous
trajectory data sets in which population-averaged quantities may obscure
rare clusters, distinct subpopulations, and correlations between particle
motion and collective structure. Systematically exploring the relevant
particle and forcing parameters experimentally is therefore demanding.
Stochastic models that retain complete trajectories can help identify
parameter regimes for subsequent laboratory investigation. This idea is
consistent with statistical-physics approaches to materials design
\cite{miskin:etal:2016}, including inverse-design methods based on large
deviation theory for nonequilibrium colloidal assembly
\cite{das:limmer:2021}. The framework developed here provides a complementary
route in which a large deviation principle for rare trajectory clusters
yields an explicit stochastic cluster law that can be sampled
computationally and calibrated against experimental micromotor trajectories.
In the present application, this enables comparison of simulated and
measured velocity distributions across particle sizes and magnetic forcing
conditions, providing a first step toward simulation-assisted exploration
of experimentally relevant parameter regimes.

We consider a Poisson number of particles on the \(d\)-dimensional unit torus.
Their initial positions are independent and uniform, and each particle carries
an independent displacement path \(\Gamma_i\) with law \(\mathbb Q\). Its
trajectory is \(\Gamma_i'(t)=X_i+r_n\Gamma_i(t)\), where \(r_n\) is the
interaction range. We work in the sparse regime \(nr_n^d\to0\), so a typical
particle has no nearby neighbors. Nevertheless, for fixed \(k\), the expected
number of local clusters of \(k\) particles may diverge at the scale
\(\rho_{k,n}=n^kr_n^{d(k-1)}\). This is the scale at which rare local
structures involving \(k\) particles become macroscopically visible.

A collection of \(k\) trajectories is called connected when its full diameter
is of order \(r_n\) at one common time. Throughout the paper, this term refers
to the stated diameter condition and not to connectivity in the geometric
graph obtained from pairwise proximity. The collection is isolated when it
remains separated from every other trajectory at the same scale. After
rescaling and centering, each cluster determines an element of a Polish space
\(E_k\) of unordered trajectory configurations. The resulting random finite
measure \(\xi_{k,n}\) satisfies an LDP with speed \(\rho_{k,n}\) and rate
function \(h(\,\cdot\,\mid\tau_k)\).

The finite reference measure \(\tau_k\) is the central object in the theory. It
is obtained from \(k\) independent paths with law \(\mathbb Q\), integrated
over their relative initial positions and restricted by the connection
condition. Consequently, \(\tau_k\) records both the geometry of a cluster and
the underlying motion law. This is the principal mathematical effect of retaining
the full trajectory configuration.

For a bounded measurable energy \(W:E_k\to\mathbb R\), exponential tilting
gives
\[
\lim_{n\to\infty}\frac1{\rho_{k,n}}
\log\mathbb E\exp\big\{-\beta\rho_{k,n}H_W(\PP_n)\big\}
=
-\inf_{\mu\in M_f(E_k)}
\Big\{\beta\int W\,d\mu+h(\mu\mid\tau_k)\Big\}.
\]
The minimizer has finite intensity measure
\(\mu_\beta^{\mathrm{gc}}=e^{-\beta W}\tau_k\), and its normalized
cluster-shape law is proportional to \(e^{-\beta W}\tau_k\). Thus, the LDP
does more than assign an exponential cost to rare clusters. It identifies the
cluster shapes favored by interaction and provides a direct target for
Metropolis--Hastings sampling. This finite-cluster variational principle is
closely related in spirit to the entropy-energy balance in Gibbs point process
theory \cite{zass,jahnel-koppl-steenbeck-zass}, but its object and asymptotic
regime are different. We also treat a hard-core constraint, for which the
lower bound requires direct control of the event that no forbidden cluster
occurs.

The computational and experimental parts complete this connection. We first
solve the finite-measure variational problem and separate its total cluster
intensity from the normalized law of one cluster shape. The latter yields an
implementable MCMC sampler for interacting OU-Brownian trajectories. Synthetic
examples show how anisotropic interactions and external fields favor chains,
aggregation, and coordinated motion. We then calibrate the stochastic cluster
model to experimental trajectories of magnetically driven micromotors and
compare simulated and observed velocity distributions across particle sizes
and magnetic forcing conditions. Beyond assessing whether the stochastic model
can reproduce the principal experimental velocity statistics, this comparison
provides a pilot study for simulation-assisted parameter exploration. In
particular, the fitted velocity parameter \(v^\star\) varies systematically with
particle size. Thus, the data analysis is a first test of whether
the cluster law selected by the large deviation variational problem can serve
as a computational tool for exploring experimentally relevant parameter
regimes.

To summarize, the main contributions of this paper are as follows.
\begin{enumerate}
\item We establish a large deviation principle for the empirical measure of
isolated \(k\)-particle trajectory clusters in the sparse regime, with an
explicit entropy rate function determined by the path-valued reference
measure \(\tau_k\).

\item We identify the explicit minimizing cluster law and use it as the target
of a Metropolis--Hastings sampler for interacting OU-Brownian trajectory
clusters.

\item We calibrate the resulting stochastic cluster model to experimental
micromotor trajectories and use the fitted model to investigate how effective
parameters vary with particle size and magnetic forcing.
\end{enumerate}

The main technical difficulty is created by particle motion. In the static
model of \cite{gtf}, a spatial partition almost localizes the clusters
immediately. Here trajectories that start in different cubes may later
approach one another, affecting both connectivity and isolation. Under the
assumptions on the path tails, trajectories with large range may occur, but
we show that the collection of blocks they influence is superexponentially
negligible at speed \(\rho_{k,n}\). Capping the number of blocks influenced
by one trajectory at the total number of blocks permits a single Poisson
exponential estimate for the tail exponents covered by the theorem.
After truncating the paths, the remaining discrepancies are confined to
thin boundary layers. This yields
exponential equivalence between \(\xi_{k,n}\), a localized empirical measure
\(\eta_{k,n}\), and the normalized Poisson reference measure \(\zeta_{k,n}\)
with mean measure \(\tau_k\). In particular, no additional scale condition is needed
for the ordinary LDP when the exponent in the tail bound for the path range
is at least the spatial dimension.

The hard-core lower bound requires exact absence of forbidden clusters,
which the ordinary LDP lower bound and exponential equivalence at a fixed
positive distance do not control. We first exclude trajectories beyond the
same range cutoff, at cost \(\exp\{-o(\rho_{k,n})\}\), and then tilt the
independent short blocks by the hard-core Gibbs weight. The localization
errors have probabilities tending to zero under this tilt and a dependency
graph of uniformly bounded degree. The quantitative Lov\'asz local lemma
\cite[Theorem~1.1]{moser-tardos2009} then shows that all errors can be
excluded simultaneously at subexponential
cost. On this event, the global and blockwise cluster measures agree
exactly, giving the lower bound for the original forbidden set.

Section~\ref{sec:model} defines the model and states the main results.
Section~\ref{sec:proof-outline} introduces the central propositions and proves
the main LDP conditional on them. Sections~\ref{sec:prelim},
\ref{sec:1proof}, and \ref{sec:2proof} establish the estimates for path ranges,
Poisson approximation, and spatial localization, respectively.
Section~\ref{sec:interact} proves the results for bounded and hard-core
interactions. Section~\ref{sec:ex} derives the explicit optimizer and develops
the MCMC method, and Section~\ref{sec:real-data} calibrates the resulting model
to experimental micromotor trajectories and investigates its potential for
simulation-assisted exploration of experimentally relevant parameter regimes.

%
%
\section{Model definition and main results}
\label{sec:mod}
\label{sec:model}

%
%
%
\label{ss:fine}

Fix an integer spatial dimension \(d\ge1\), an integer cluster size
\(k\ge2\), and an
interaction parameter \(L>0\). We consider moving particles on the unit torus
\(\Lambda=[0,1]^d\), with addition modulo \(\mathbb Z^d\).  The number of
particles, denoted by \(N\), has the Poisson distribution with mean \(n\).
Conditional on \(N\), their initial positions \(X_1,\ldots,X_N\) are
independent and uniformly distributed on \(\Lambda\).  The interaction range
is \(r_n>0\), and we work in the sparse regime \(nr_n^d\to0\).  Thus, under
the reference measure, the expected number of particles within distance of
order \(r_n\) from a given particle tends to zero.
We first specify the particle dynamics.  Let
\[
\CC:=C([0,1],\R^d),\qquad
\CC_0:=\{\gamma\in\CC:\gamma(0)=0\},\qquad
\CC_{\T}:=C([0,1],\T^d),
\]
where all three path spaces carry the supremum metric and their Borel
\(\sigma\)-fields.  A path \(\gamma\in\CC_0\) describes displacement relative
to the initial position, and \(\Q\) is a probability measure on \(\CC_0\).
Let \(\Gamma_1,\Gamma_2,\ldots\) be independent random elements of \(\CC_0\)
with law \(\Q\), independent of the initial positions.  Particle \(i\) follows
the path
\[
\Gamma_i'(t):=X_i+r_n\Gamma_i(t),\qquad t\in[0,1],
\]
on the torus.  We write \(\PP_n=\{\Gamma_i':i\le N\}\) for the resulting point
process on \(\CC_{\T}\).  Equivalently, \(\PP_n\) is the image of the marked
Poisson process \(\{(X_i,\Gamma_i)\}_{i\le N}\) on
\(\Lambda\times\CC_0\), with intensity measure
\(n\,\ms{Leb}\otimes\Q\), under the measurable map
\[
\Phi_n:\Lambda\times\CC_0\to\CC_{\T},\qquad
\Phi_n(x,\gamma)(t)=x+r_n\gamma(t).
\]
Write \(\lambda_\Lambda\) for Lebesgue measure on \(\Lambda\).
The mapping theorem shows that \(\PP_n\) is a Poisson point process on
\(\CC_{\T}\) with intensity
\(\Theta_n:=n(\lambda_\Lambda\otimes\Q)\circ\Phi_n^{-1}\).  For
\(A\subseteq\Lambda\), write
\(\CC_{\T,A}:=\{\gamma'\in\CC_{\T}:\gamma'(0)\in A\}\).  Distances between
paths on the torus are computed with the flat torus metric.  
The numerical model in Section~\ref{sec:ex} uses Ornstein-–Uhlenbeck velocity dynamics with an additional Brownian positional component.

We study isolated connected clusters of \(k\) particles at range \(r_n\), as
in \cite{gtf}.  The first particle may lie anywhere on the torus, whereas the
remaining \(k-1\) particles must start within distance of order \(r_n\).  The
natural scale for the number of such clusters is therefore
\[
\rho_{k,n}:=n^kr_n^{d(k-1)}.
\]
We assume that \(\rho_{k,n}\to\infty\).  To avoid inessential rounding in the
block construction, we also assume that \(\rho_{k,n}^{1/d}\) is an integer for
all sufficiently large \(n\).  The torus is then partitioned into
\(\rho_{k,n}\) congruent cubes of side length \(\rho_{k,n}^{-1/d}\).  The
argument extends to arbitrary sequences by the usual rounding procedure.
We next define the space of cluster shapes.  Let
\[
\widetilde E_k
:=
\Big\{(\gamma_1,\ldots,\gamma_k)\in\CC^k:
\sum_{i=1}^k\gamma_i(0)=0\Big\},
\]
and let \(E_k:=\widetilde E_k/\mathfrak S_k\), where \(\mathfrak S_k\) acts by
permuting the coordinates.  For \([\gamma],[\eta]\in E_k\), set
\[
d_{E_k}([\gamma],[\eta])
:=
\min_{\sigma\in\mathfrak S_k}
\max_{1\le i\le k}\|\gamma_i-\eta_{\sigma(i)}\|_\infty.
\]
The space \(\widetilde E_k\) is closed in the Polish space \(\CC^k\), and the
finite group acts by isometries.  Hence, \(E_k\) is Polish, and we equip it
with its Borel \(\sigma\)-field.    Write \(M_f(E_k)\) for the space
of finite Borel measures on \(E_k\), equipped with the evaluation
\(\sigma\)-field, and let \(\mathcal M_p(E_k)\subseteq M_f(E_k)\) be the
measurable subspace of finite point measures.

Now, let
\(\Delta_{\T}(x,y)\in[-1/2,1/2)^d\) be the coordinatewise principal
representative of \(y-x\) on the torus, and fix a Borel order on
\(\CC_{\T}\).  Given
\(\YY=\{\gamma'_1,\ldots,\gamma'_k\}\subseteq\PP_n\), enumerate its elements
in this order and put \(x_i=\gamma'_i(0)\).  Let
\(D_n\gamma'_i\in\CC_0\) be the rescaled displacement obtained from the unique
lift of \(\gamma'_i\) that starts at the representative
\(x_i\in[0,1)^d\).  Define
\[
\widehat\gamma_i
:=r_n^{-1}\Delta_{\T}(x_1,x_i)+D_n\gamma'_i,
\qquad
\bar\gamma_i
:=\widehat\gamma_i-\frac1k\sum_{j=1}^k\widehat\gamma_j(0).
\]
The cluster shape is
\(\operatorname{sh}_n(\YY):=[\bar\gamma_1,\ldots,\bar\gamma_k]\in E_k\).
This is a measurable function of the unordered cluster, and we retain the
notation \(r_n^{-1}\bar\YY:=\operatorname{sh}_n(\YY)\). 
For a finite trajectory set \(\mathcal Y\), write
\(\mathcal Y(0):=\{\gamma'(0):\gamma'\in\mathcal Y\}\).

Our main object is the empirical measure of isolated connected
\(k\)-particle clusters at the interaction scale:
\begin{align}
\label{eq:xi}
\xi_{k,n}
:=\frac1{\rho_{k,n}}
\sum_{\substack{\YY\su\PP_n\\|\YY|=k}}
 s_{\ms{conn},n}(\YY)\,
 s_{\ms{iso},n}(\YY,\PP_n)\,
 \de_{r_n^{-1}\bar\YY},
\end{align}
where
\begin{align*}
s_{\ms{conn},n}(\{\g_1',\dots,\g_k'\})
&:=\one\{\diam(\g_1',\dots,\g_k')\le r_nL\}\\
&:=\one\Big\{\inf_{t\in[0,1]}\max_{i,j\le k}
|\g_i'(t)-\g_j'(t)|\le r_nL\Big\},\\
s_{\ms{iso},n}(\{\g_1',\dots,\g_k'\},\PP_n)
&:=\one\{\dist_0(\{\g_1',\dots,\g_k'\},
\PP_n\setminus\{\g_1',\dots,\g_k'\})>Lr_n\}\\
&:=\one\Big\{\inf_{t\in[0,1]}\min_{i\le k}
\min_{\G'\in\PP_n\setminus\{\g_1',\dots,\g_k'\}}
|\g_i'(t)-\G'(t)|>Lr_n\Big\}.
\end{align*}
Thus, a set of trajectories forms a cluster when its full diameter is at most
\(Lr_n\) at one common time and it remains isolated from every other
trajectory at all times.  
The same connection and isolation indicators will be used for restrictions
of the Poisson process, with the minimum over an empty set interpreted as
\(+\infty\). The measure \(\xi_{k,n}\) is a measurable random element
of \(M_f(E_k)\).

For an ordered tuple \((\gamma_1,\ldots,\gamma_k)\in\CC^k\), let
\(\operatorname{cen}(\gamma_1,\ldots,\gamma_k)\in E_k\) be its equivalence
class after subtracting \(k^{-1}\sum_i\gamma_i(0)\) from every coordinate.
For Euclidean trajectories, \(\operatorname{diam}\) denotes the same
infimum of the simultaneous diameter, with Euclidean distances in place
of torus distances. Let \(\lambda_{k-1}\) denote Lebesgue measure on
\((\mathbb R^d)^{k-1}\). The intensity measure governing a typical cluster is
\begin{align*}
\tau_k(A)
:=\frac1{k!}\,
\mathbb E\Big[\lambda_{k-1}\Big(
\Big\{y\in(\R^d)^{k-1}:\,
&\diam(\G_0,y_1+\G_1,\dots,y_{k-1}+\G_{k-1})\le L,\\
&\operatorname{cen}(\G_0,y_1+\G_1,\dots,y_{k-1}+\G_{k-1})\in A
\Big\}\Big)\Big],
\end{align*}
where \(\G_0,\G_1,\dots,\G_{k-1}\) are independent trajectories with law
\(\Q\).    Symmetry of the integrand and the product
law make the definition independent of the ordering, and the centering map
removes every common lattice translation.
The large deviation principle is stated in the \(\tau\)-topology, the
coarsest topology on \(M_f(E_k)\) for which
\(\mu\mapsto\int f\,d\mu\) is continuous for every bounded measurable
function \(f\); see \cite{dz98}.  For finite measures \(\mu\) and \(\tau_k\),
define
\[
h(\mu\mid\tau_k)
:=
\begin{cases}
\displaystyle\int_{E_k}(f\log f-f+1)\,d\tau_k,
&\mu=f\tau_k,\\[1ex]
+\infty,&\mu\not\ll\tau_k.
\end{cases}
\]
Thus, the static cluster intensity is replaced by a measure on entire
trajectory configurations, while the entropy form of the rate function is
preserved.

Write \(R(\gamma):=\sup_{t\in[0,1]}|\gamma(t)|\).  We assume that there are
constants \(C,c,\alpha>0\) such that
\begin{align}
\label{eq:curve_tail_condition}
\Q(R(\Gamma)>u)\le C\exp\{-cu^\alpha\},\qquad u\ge0.
\end{align}
We also assume
\begin{align}
\label{eq:relaxed_tail_scale_condition}
\alpha\ge d
\quad\text{or}\quad
\rho_{k,n}r_n^\alpha\log n\to0.
\end{align}
The second alternative is needed only when \(\alpha<d\).  The same assumptions
will be used for the ordinary LDP and for the bounded and hard-core interaction
results.  The tail bound also implies
\begin{align}
\label{eq:curve_moment_condition}
\mathbb E_\Q[R(\Gamma)^p]<\infty
\qquad\text{for every }0\le p<\infty.
\end{align}
We note that 
the case \(\alpha=2\) includes Brownian tails.  In dimension two, no scale
condition beyond the sparse regime is required for the ordinary LDP.
Bounded trajectories likewise require no additional scale condition.
On the connection event in the definition of \(\tau_k\), each relative
starting point has norm at most \(L+2\max_{0\le i\le k-1}R(\Gamma_i)\).
Thus, \(\tau_k(E_k)\le
C\mathbb E(1+\max_{0\le i\le k-1}R(\Gamma_i))^{d(k-1)}<\infty\) by
\eqref{eq:curve_moment_condition}. Fix also \(\varepsilon_0>0\)
for the polynomial sparsity assumption in the case \(\alpha<d\).

\bet[LDP for the empirical measure of clusters]
\label{thm:fine}
Let $k \ge 2$. Suppose that $\rho_{k,n}\to\infty$, that
$n r_n^d\to0$, and that
\eqref{eq:curve_tail_condition} and
\eqref{eq:relaxed_tail_scale_condition} hold.
 If $\alpha<d$, suppose in addition that
$n r_n^d=O(n^{-\varepsilon_0})$.
Then $(\xi_{k,n})_{n\ge1}$ satisfies an LDP on
$M_f(E_k)$, equipped with the $\tau$-topology, with
speed $\rho_{k,n}$ and good rate function
$\sigma\mapsto h(\sigma\mid\tau_k)$.
\ent

The proof reduces the theorem to three propositions.  A normalized Poisson
random measure with intensity \(\tau_k\) satisfies the LDP.  The two
central comparisons show that this Poisson measure is exponentially equivalent
first to a cluster measure defined block by block and then to \(\xi_{k,n}\).  These
propositions are stated in Section~\ref{sec:proof-outline}, where the theorem
is proved conditional on them.  Sections~\ref{sec:prelim}--\ref{sec:2proof}
then establish the required estimates.

We next state the consequence for the free energy. Let \(\beta>0\) be
the inverse temperature, and let
\(W:E_k\to\mathbb R\) be bounded and measurable.  For example, suppose that
\(v:(\mathbb R^d)^k\to\mathbb R\) is bounded and measurable, invariant under
translations, and invariant under permutations of its \(k\) arguments.  It
then induces the energy
\[
W([\gamma_1,\ldots,\gamma_k])
:=\int_0^1v(\gamma_1(t),\ldots,\gamma_k(t))\,dt.
\]
Now, define
\[
H_W(\PP_n)
:=\frac1{\rho_{k,n}}
\sum_{\substack{\YY\subseteq\PP_n\\|\YY|=k}}
 s_{\ms{conn},n}(\YY)\,
 s_{\ms{iso},n}(\YY,\PP_n)\,
 W(r_n^{-1}\bar\YY)
=\int_{E_k}W\,d\xi_{k,n}.
\]
We now state the limiting free energy result. The derivation is postponed to the proof section.
\bec[Limiting free energy]
\label{cor:free}
For every \(\beta>0\),
\[
\lim_{n\to\infty}
\frac1{\rho_{k,n}}
\log\mathbb E\exp\big\{-\rho_{k,n}\beta H_W(\PP_n)\big\}
=
-\inf_{\mu\in M_f(E_k)}
\Big\{\beta\int W\,d\mu+h(\mu\mid\tau_k)\Big\}.
\]
\enc

We also allow a hard-core constraint. More precisely,  for
\(\varphi=\{\gamma_1,\ldots,\gamma_k\}\in E_k\), define
\[
d_{\min}(\varphi)
:=\min_{t\in[0,1]}\min_{i\ne j}|\gamma_i(t)-\gamma_j(t)|.
\]
The function \(d_{\min}\) is continuous because
\(|d_{\min}(\varphi)-d_{\min}(\psi)|\le2d_{E_k}(\varphi,\psi)\).  For
\(\omega>0\), let
\(B=\mathcal F_\omega:=\{\varphi\in E_k:d_{\min}(\varphi)\le\omega\}\).
Given a bounded measurable function \(w:E_k\to\mathbb R\), set
\[
W_{\mathrm{hc}}(\varphi)
:=
\begin{cases}
+\infty,&\varphi\in B,\\
w(\varphi),&\varphi\notin B.
\end{cases}
\]
Define the finite measure \(\mu_*:=e^{-\beta w}\mathbf{1}_{B^c}\tau_k\).
The Gibbs weight of the hard-core energy is understood as
\(\exp\{-\beta\rho_{k,n}\int w\,d\xi_{k,n}\}
\mathbf{1}\{\xi_{k,n}(B)=0\}\).

\begin{theorem}[Hard-core free energy]
\label{thm:hard-core-free-energy}
Under the assumptions of Theorem~\ref{thm:fine}, for every \(\beta>0\),
\[
\begin{aligned}
\lim_{n\to\infty}\frac1{\rho_{k,n}}
\log\mathbb E e^{-\beta\rho_{k,n}\int W_{\mathrm{hc}}\,d\xi_{k,n}}
&=
-\inf_{\substack{\mu\in M_f(E_k)\\\mu(B)=0}}
\Big\{\beta\int_{B^c}w\,d\mu+h(\mu\mid\tau_k)\Big\}\\
&=
-\tau_k(B)+\int_{B^c}(e^{-\beta w}-1)\,d\tau_k.
\end{aligned}
\]
The unique minimizer satisfies
\(d\mu_*/d\tau_k=e^{-\beta w}\mathbf{1}_{B^c}\).
\end{theorem}

For the lower bound, define
\(
Z_n^{\mathrm{hc}}
:=
\mathbb E\Big[
\exp\Big\{-\beta\rho_{k,n}\int w\,d\xi_{k,n}\Big\}
\mathbf{1}\{\xi_{k,n}(B)=0\}
\Big].
\)
The following proposition obtains the lower bound by excluding far
trajectories and requiring exact agreement with the independent block
measure. It applies directly to the forbidden set \(B\).

\begin{proposition}[Hard-core lower bound]
\label{pr:hard-core-lower}
Under the assumptions of Theorem~\ref{thm:fine}, for every \(\beta>0\),
\[
\liminf_{n\to\infty}
\frac1{\rho_{k,n}}\log Z_n^{\mathrm{hc}}
\ge
-\tau_k(B)
+\int_{B^c}
\big(e^{-\beta w}-1\big)\,d\tau_k.
\]
\end{proposition}

\begin{proof}[Proof of Theorem~\ref{thm:hard-core-free-energy}]
For the upper bound, fix \(M>0\), set \(W_M:=w+M\mathbf{1}_B\), and write
\(F_M(\mu):=\int W_M\,d\mu\).  The hard-core Gibbs weight is at most
\(\exp\{-\beta\rho_{k,n}F_M(\xi_{k,n})\}\).  Corollary~\ref{cor:free} and
pointwise minimization over \(d\mu=f\,d\tau_k\) give
\[
\limsup_{n\to\infty}
\frac1{\rho_{k,n}}\log Z_n^{\mathrm{hc}}
\le
-\inf_{\mu\in M_f(E_k)}
\big\{\beta F_M(\mu)+h(\mu\mid\tau_k)\big\}
=
\int_{E_k}\big(e^{-\beta W_M}-1\big)\,d\tau_k.
\]
Letting \(M\to\infty\) and using dominated convergence yields the required
upper bound. Proposition~\ref{pr:hard-core-lower} gives the matching
lower bound under exactly the same hypotheses. If \(\mu=f\tau_k\) and
\(\mu(B)=0\), then \(f=0\) \(\tau_k\)-almost everywhere on \(B\), and
\[
\beta\int_{B^c}w\,d\mu+h(\mu\mid\tau_k)
=
\tau_k(B)
+
\int_{B^c}\big(\beta wf+f\log f-f+1\big)\,d\tau_k.
\]
The integrand is minimized uniquely at \(f=e^{-\beta w}\), with minimum
\(1-e^{-\beta w}\).  Measures that are not absolutely continuous with
respect to \(\tau_k\) have infinite entropy.  This proves the variational
identity and the formula for the unique minimizer.
\end{proof}

\section{Reduction to Poisson approximation and spatial localization}
\label{sec:out}
\label{sec:proof-outline}
The proof follows the block approximation strategy used for the static
model in \cite{gtf}. Motion allows trajectories from different starting
blocks to interact, so localization requires control of their ranges.
We reduce Theorem~\ref{thm:fine} to the Poisson reference LDP and two
comparison propositions, and prove the theorem from those inputs below.
Partition the torus into the cubes
\(Q_1,\ldots,Q_{\rho_{k,n}}\) from Section~\ref{sec:model}.  For
\(A\subseteq\Lambda\), let \(\PP_A\) denote the trajectories whose starting
positions lie in \(A\), and define
\[
\eta_{k,n}
:=
\frac1{\rho_{k,n}}
\sum_{\ell=1}^{\rho_{k,n}}
\sum_{\substack{\mathcal Y\subseteq\PP_n,\ |\mathcal Y|=k\\
\mathcal Y(0)\subseteq Q_\ell}}
s_{\ms{conn},n}(\mathcal Y)
 s_{\ms{iso},n}(\mathcal Y,\PP_{Q_\ell})
 \de_{r_n^{-1}\bar{\mathcal Y}}.
\]
Thus, \(\eta_{k,n}\) counts only clusters whose starting positions lie in one
cube and tests isolation using only particles from that cube.
Figure~\ref{fig:localized-measure} illustrates this localization.
\begin{figure}[!htbp]
    \centering
    \includegraphics[width=0.36\textwidth]{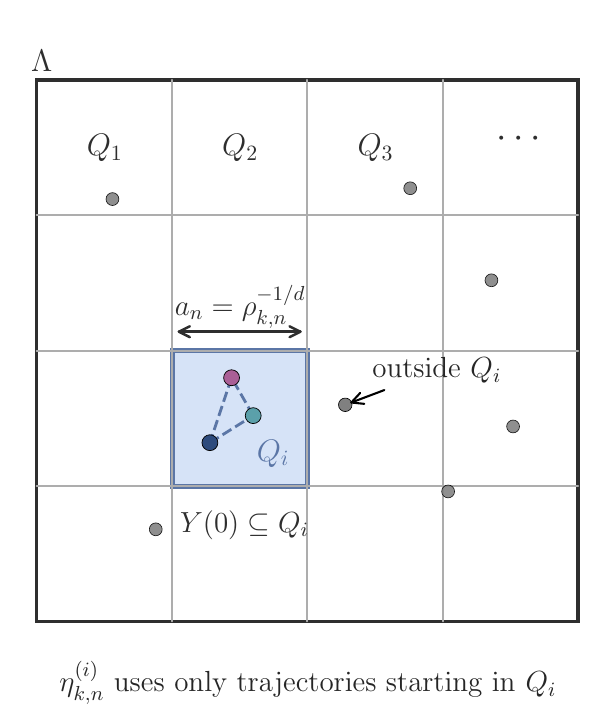}
    \caption{Illustration of the local cluster measure in one cube $Q_i$.}
    \label{fig:localized-measure}
\end{figure}

For \(\mu,\nu\in M_f(E_k)\), write
\(d_{\mathrm{TV}}(\mu,\nu)=\|\mu-\nu\|_{\mathrm{TV}}
:=|\mu-\nu|(E_k)\), where \(|\mu-\nu|\) is the variation measure.
For probability laws \(P,Q\) on \(M_f(E_k)\), write
\(d_{\mathrm{TV}}(P,Q):=\sup_A|P(A)-Q(A)|\), where \(A\) ranges over
the evaluation \(\sigma\)-field. Thus, total variation between laws is
half the variation norm of their difference. We denote the law of a
random measure \(X\) by \(\mathcal L(X)\).

Let \((\zeta_k^{(\ell)})_{\ell\ge1}\) be independent Poisson random measures
on \(E_k\), each with intensity \(\tau_k\), and set
\(\zeta_{k,n}=\rho_{k,n}^{-1}\sum_{\ell=1}^{\rho_{k,n}}
\zeta_k^{(\ell)}\). Three propositions reduce the LDP to this reference
process. The first gives its large deviation behavior, while the other two
control the Poisson approximation within cubes and the localization errors
between cubes.

\begin{proposition}[Poisson reference process]
\label{pr:poisson-ldp}
The sequence \((\zeta_{k,n})_{n\ge1}\) satisfies an LDP in the
\(\tau\)-topology with speed \(\rho_{k,n}\) and good rate function
\(h(\,\cdot\,\mid\tau_k)\).
\end{proposition}

The first comparison couples the independent block contributions with
the Poisson reference process. Its proof combines an approximation in one
cube with independent maximal couplings over the cubes.
\begin{proposition}[Poisson approximation]
\label{pr:55}
The measures \(\eta_{k,n}\) and \(\zeta_{k,n}\) are exponentially equivalent
at speed \(\rho_{k,n}\) with respect to the distance in total variation.
\end{proposition}

The second comparison controls the effects of interactions across the
boundaries of the starting cubes. It treats both clusters that lose
isolation outside their cube and global clusters whose starting positions
lie in several cubes.
\begin{proposition}[Spatial localization]
\label{pr:56}
The measures \(\xi_{k,n}\) and \(\eta_{k,n}\) are exponentially equivalent
at speed \(\rho_{k,n}\) with respect to the distance in total variation.
\end{proposition}

\begin{proof}[Proof of Theorem~\ref{thm:fine}]
We first put the two comparisons on one probability space. Write
\(S:=M_f(E_k)\), equipped with its evaluation \(\sigma\)-field.
Since \(E_k\) is Polish, this measurable space is standard Borel,
so the original pair \((\xi_{k,n},\eta_{k,n})\) admits a regular
conditional law \(\mathcal K_n(\nu,\cdot)\) of \(\xi_{k,n}\)
given \(\eta_{k,n}=\nu\). Let \(\pi_n\) be a coupling of
\(\eta_{k,n}\) and \(\zeta_{k,n}\) supplied by
Proposition~\ref{pr:55}. Define a probability law on \(S^3\) by
\[
\widehat{\mathbb P}_n(d\mu,d\nu,d\omega)
:=\mathcal K_n(\nu,d\mu)\,\pi_n(d\nu,d\omega).
\]
Denote the three coordinates by \(\xi_{k,n},\eta_{k,n},\zeta_{k,n}\).
Integrating in \(\mu\) recovers \(\pi_n\), because
\(\mathcal K_n(\nu,S)=1\). Integrating in \(\omega\) recovers
\(\mathcal K_n(\nu,d\mu)\mathcal L(\eta_{k,n})(d\nu)\), which
is the original joint law of \((\xi_{k,n},\eta_{k,n})\).
Thus, both comparison propositions apply under
\(\widehat{\mathbb P}_n\). For every \(\delta>0\), the triangle
inequality gives
\[
\begin{aligned}
\widehat{\mathbb P}_n\big(
d_{\mathrm{TV}}(\xi_{k,n},\zeta_{k,n})>\delta\big)
&\le\widehat{\mathbb P}_n\big(
d_{\mathrm{TV}}(\xi_{k,n},\eta_{k,n})>\delta/2\big)
+\widehat{\mathbb P}_n\big(
d_{\mathrm{TV}}(\eta_{k,n},\zeta_{k,n})>\delta/2\big),\\
\limsup_{n\to\infty}\frac1{\rho_{k,n}}
\log\widehat{\mathbb P}_n\big(
d_{\mathrm{TV}}(\xi_{k,n},\zeta_{k,n})>\delta\big)
&=-\infty.
\end{aligned}
\]
The second line follows by bounding the sum in the first line by twice
its larger summand and using Propositions~\ref{pr:56} and \ref{pr:55}.
The additional term \(\rho_{k,n}^{-1}\log2\) tends to zero.

To verify the topological approximation hypothesis, let
\(\mathcal F=(f_1,\ldots,f_m)\) be any finite nonempty family of
bounded measurable functions on \(E_k\), and put
\(c_{\mathcal F}:=\max_j\|f_j\|_\infty\).
The pseudometric \(d_{\mathcal F}\) defined below controls these
test-function evaluations. For every \(\varepsilon>0\), we have
\[
\begin{aligned}
d_{\mathcal F}(\mu,\nu)
&:=\max_{j\le m}\Big|\int f_j\,d\mu-\int f_j\,d\nu\Big|
\le c_{\mathcal F}d_{\mathrm{TV}}(\mu,\nu),\\
\widehat{\mathbb P}_n\big(
d_{\mathcal F}(\xi_{k,n},\zeta_{k,n})>\varepsilon\big)
&\le\widehat{\mathbb P}_n\big(
d_{\mathrm{TV}}(\xi_{k,n},\zeta_{k,n})>
\varepsilon/c_{\mathcal F}\big)
\qquad(c_{\mathcal F}>0).
\end{aligned}
\]
If \(c_{\mathcal F}=0\), then \(d_{\mathcal F}=0\), and the
probability on the left is zero for every \(\varepsilon>0\).
Otherwise, the preceding total variation estimate shows that its
normalized logarithm tends to \(-\infty\).
These pseudometrics generate the \(\tau\)-topology, separate finite
measures, and are closed under finite maxima by concatenating the
families of test functions. Their balls are evaluation-measurable,
and the same holds for their discrepancy events on \(S^2\), since
integration against a bounded measurable function is
evaluation-measurable. The evaluation maps embed \(S\) into a product
of copies of \(\mathbb R\); hence, this topology is Hausdorff and
completely regular.

The exponential-equivalence case of
\cite[Theorem~1.13]{eichelsbacher-schmock} transfers a full LDP with
a good rate function between laws on such a space when a coupling has
superexponentially small discrepancy for every generating pseudometric
and every positive threshold. Its measurable-space hypothesis only
requires the \(\sigma\)-field to contain the pseudometric balls,
as verified above. In the cited theorem, take both spaces to be \(S\)
and the reference laws to be those of \(\zeta_{k,n}\).
Use a constant sequence of identity maps as the approximating maps,
with the identity as their limit.
The maps are measurable and continuous, and their approximation error
is identically zero on every rate-function level set.
Proposition~\ref{pr:poisson-ldp} supplies the full reference LDP at
speed \(\rho_{k,n}\) with good rate function
\(h(\,\cdot\,\mid\tau_k)\), while the coupling above verifies the
required exponential approximation for every \(d_{\mathcal F}\).
The rate function produced by the identity map is
\(\inf_{\nu:\,\nu=\mu}h(\nu\mid\tau_k)=h(\mu\mid\tau_k)\).
Thus, \(\xi_{k,n}\) has the asserted full LDP with the same speed
and good rate function.
\end{proof}

\begin{proof}[Proof of Proposition~\ref{pr:poisson-ldp}]
The measure \(\rho_{k,n}\zeta_{k,n}\) is a Poisson random measure with
intensity \(\rho_{k,n}\tau_k\). The Poisson version of Sanov's theorem
in the \(\tau\)-topology therefore gives the stated LDP; see
\cite[Proposition~3.6]{hirsch-jahnel-patterson}.
\end{proof}

The remainder of the argument proves the two comparison propositions.
Section~\ref{sec:prelim} controls trajectories with large range,
Section~\ref{sec:1proof} proves Proposition~\ref{pr:55}, and
Section~\ref{sec:2proof} proves Proposition~\ref{pr:56}.  Finally,
Section~\ref{sec:interact} proves the results for the free energy.

\section{Control of trajectories with large range}
\label{sec:prelim}

This section controls the blocks affected by trajectories with large range.
We use one cutoff throughout the ordinary LDP and the hard-core lower bound.
The number of affected blocks is superexponentially negligible, while the
event that no trajectory exceeds the cutoff has the subexponential cost
needed for the hard-core constraint. The estimate below uses the fact that
one trajectory can affect at most the total number of blocks.

Put \(s_n:=nr_n^d\), \(m_n:=n/\rho_{k,n}=s_n^{-(k-1)}\),
\(a_n:=\rho_{k,n}^{-1/d}\), and
\(b_n:=r_n/a_n=s_n^{k/d}\). Thus, \(s_n\to0\), \(m_n\to\infty\),
and \(b_n\to0\). Let \(q:=\min\{\alpha,d\}\), and choose \(c_0>0\)
such that \(M_0:=\mathbb E_\Q e^{c_0R(\Gamma)^q}<\infty\), as permitted
by \eqref{eq:curve_tail_condition}. For all sufficiently large \(n\), set
\(v_n:=(8c_0^{-1}\log m_n)^{1/q}\) and
\(q_n:=\Q(R(\Gamma)>v_n)\). The identities for \(m_n\) and \(b_n\)
give \(b_nv_n=C s_n^{k/d}(\log(1/s_n))^{1/q}\to0\), while Markov's
inequality gives \(q_n\le M_0m_n^{-8}\). In particular,
\begin{equation}
v_n\longrightarrow\infty,\qquad
b_nv_n\longrightarrow0,\qquad
m_nq_n\le M_0m_n^{-7}\longrightarrow0.
\label{eq:range-cutoff}
\end{equation}

Let \(\widetilde{\mathcal P}_n^{\mathrm s}\) and
\(\widetilde{\mathcal P}_n^{\mathrm f}\) be the restrictions of the
marked Poisson process to \(R(\Gamma)\le v_n\) and \(R(\Gamma)>v_n\),
respectively, and denote their path-space images by \(\PP_n^{\mathrm s}\)
and \(\PP_n^{\mathrm f}\). Poisson thinning makes these two processes
independent. For \(A\subseteq\Lambda\), write \(\PP_A^{\mathrm s}\)
for the short trajectories whose starting positions belong to \(A\).
Define \(N_n^{\mathrm f}:=\#\widetilde{\mathcal P}_n^{\mathrm f}\)
and \(F_n:=\{N_n^{\mathrm f}=0\}\). Since \(N_n^{\mathrm f}\) is
Poisson with mean \(nq_n=o(\rho_{k,n})\),
\begin{equation}
\mathbb P(F_n)=e^{-nq_n}=\exp\{-o(\rho_{k,n})\}.
\label{eq:cutoff-void-probability}
\end{equation}
This identity will allow us to work with the short process in the hard-core
lower bound. The ordinary LDP instead uses the stronger estimate for the
number of affected blocks proved below.

Set \(h_n^{\mathrm s}:=2r_nv_n+Lr_n\) and
\(Q_i^{\mathrm s,+}:=\{x\in\Lambda:d_{\mathbb T}(x,Q_i)\le kh_n^{\mathrm s}\}\).
By \eqref{eq:range-cutoff}, \(h_n^{\mathrm s}=o(a_n)\). For a marked
trajectory \((z,\gamma)\), define
\(\mathcal D_n(z,\gamma):=\{i:d_{\mathbb T}(z,Q_i)\le c_*r_nR(\gamma)\}\),
where \(c_*:=2L+4\), and let \(\mathcal D_n\) be the union of these
sets over \(\widetilde{\mathcal P}_n^{\mathrm f}\). The following
estimate controls both the number of far trajectories and their spatial
influence.

\begin{lemma}[Blocks influenced by trajectories with large range]
\label{lem:far-influence}
Under the assumptions of Theorem~\ref{thm:fine}, for every \(\delta>0\),
\[
\lim_{n\to\infty}\frac1{\rho_{k,n}}
\log\mathbb P(N_n^{\mathrm f}\ge\delta\rho_{k,n})
=
\lim_{n\to\infty}\frac1{\rho_{k,n}}
\log\mathbb P(|\mathcal D_n|\ge\delta\rho_{k,n})=-\infty.
\]
For all sufficiently large \(n\), if a far trajectory approaches within
distance \(Lr_n\) of a connected short cluster whose first trajectory in
the fixed Borel order starts in \(Q_i\), then \(i\in\mathcal D_n\).
\end{lemma}

\begin{proof}
We first verify that every block containing the first starting point of
an affected short cluster belongs to \(\mathcal D_n\). We then bound the
number of blocks influenced by one far trajectory and use this bound
to control both random quantities in the statement.

Let
\(\mathcal Y=\{\Phi_n(x_\ell,\gamma_\ell):1\le\ell\le k\}\)
be a connected short cluster, enumerated in the fixed Borel order.
Thus, \(R(\gamma_\ell)\le v_n\) for every \(\ell\), and the first
starting point is \(x_1=x\in Q_i\). Write
\(p_\ell(t):=\Phi_n(x_\ell,\gamma_\ell)(t)
=x_\ell+r_n\gamma_\ell(t)\), with addition on the torus.
For every \(t\in[0,1]\) and every \(\ell\),
\[
d_{\mathbb T}(x_\ell,p_\ell(t))
\le r_n|\gamma_\ell(t)|
\le r_nR(\gamma_\ell)
\le r_nv_n.
\]

By the definition of the connection indicator, the full diameter of
\(\mathcal Y\) is at most \(Lr_n\) at one common time.
More precisely, continuity of the paths and compactness of \([0,1]\)
show that the infimum in the definition of the connection indicator
is attained. Hence, there is a time \(t_0\in[0,1]\) such that
\(d_{\mathbb T}(p_\ell(t_0),p_j(t_0))\le Lr_n\)
for all \(\ell,j\). For any member of the cluster, with starting point
\(y=x_j\), the triangle inequality therefore gives
\[
\begin{aligned}
d_{\mathbb T}(x,y)
&\le
d_{\mathbb T}(x,p_1(t_0))
+d_{\mathbb T}(p_1(t_0),p_j(t_0))
+d_{\mathbb T}(p_j(t_0),y)\\
&\le r_nv_n+Lr_n+r_nv_n
=h_n^{\mathrm s}.
\end{aligned}
\]
Thus, this bound follows from the common-time diameter condition in
the definition of a connected cluster.

Now suppose that a far trajectory \((z,\gamma)\) approaches the
member starting at \(y=x_j\). Write
\(p_{\mathrm f}(t):=\Phi_n(z,\gamma)(t)\).
There is a time \(t_1\in[0,1]\) such that
\(d_{\mathbb T}(p_{\mathrm f}(t_1),p_j(t_1))\le Lr_n\).
The time \(t_1\) need not equal \(t_0\); the preceding estimate has
already bounded the distance between the starting points \(x\)
and \(y\). Since \(x\in Q_i\), another application of the triangle
inequality yields
\[
\begin{aligned}
d_{\mathbb T}(z,Q_i)
&\le d_{\mathbb T}(z,x)
\le
d_{\mathbb T}(z,p_{\mathrm f}(t_1))
+d_{\mathbb T}(p_{\mathrm f}(t_1),p_j(t_1))
+d_{\mathbb T}(p_j(t_1),y)
+d_{\mathbb T}(y,x)\\
&\le r_nR(\gamma)+Lr_n+r_nv_n+h_n^{\mathrm s}
=r_n\big(R(\gamma)+3v_n+2L\big).
\end{aligned}
\]
For all sufficiently large \(n\), we have \(v_n\ge1\).
Since the trajectory is far, \(R(\gamma)>v_n\), and consequently
\(3v_n\le3R(\gamma)\) and \(2L\le2LR(\gamma)\).
It follows that
\[
d_{\mathbb T}(z,Q_i)
\le (4+2L)r_nR(\gamma)
=c_*r_nR(\gamma).
\]
Thus, \(i\in\mathcal D_n(z,\gamma)\), and therefore
\(i\in\mathcal D_n\). The choice of sufficiently large \(n\)
depends only on the cutoff sequence, so the assertion holds
simultaneously for all short clusters and far trajectories.

We next give a uniform bound on the number of blocks within a
prescribed distance of a point. For \(z\in\Lambda\) and \(u\ge0\),
let
\(\mathcal J_n(z,u):=\{i:d_{\mathbb T}(z,Q_i)\le u\}\),
and write \(\overline B_{\mathbb T}(z,u)\) for the closed torus
ball of radius \(u\). Each block has volume \(a_n^d\) and diameter
at most \(\sqrt d\,a_n\). If \(i\in\mathcal J_n(z,u)\), compactness
of \(\overline{Q_i}\) gives a point \(w_i\in\overline{Q_i}\) with
\(d_{\mathbb T}(z,w_i)\le u\). Hence, every \(w\in Q_i\) satisfies
\(d_{\mathbb T}(z,w)\le u+\sqrt d\,a_n\), and
\[
\bigcup_{i\in\mathcal J_n(z,u)}Q_i
\subseteq
\overline B_{\mathbb T}(z,u+\sqrt d\,a_n).
\]

Let \(\omega_d\) denote the volume of the Euclidean unit ball.
A torus ball of radius \(r\) has volume at most \(\omega_dr^d\):
it is the image of a Euclidean ball under the quotient map, and
this projection cannot increase volume. This estimate remains
valid when the ball wraps around the torus. Since the blocks
form a partition, we obtain
\[
\begin{aligned}
|\mathcal J_n(z,u)|a_n^d
&=
\lambda_\Lambda\Big(
\bigcup_{i\in\mathcal J_n(z,u)}Q_i
\Big)
\le
\lambda_\Lambda\big(
\overline B_{\mathbb T}(z,u+\sqrt d\,a_n)
\big)
\le \omega_d(u+\sqrt d\,a_n)^d.
\end{aligned}
\]
Using
\((s+t)^d\le2^{d-1}(s^d+t^d)\) for \(s,t\ge0\),
we may therefore choose a constant \(C_d\ge1\), depending only
on \(d\), such that
\(
|\mathcal J_n(z,u)|
\le C_d\big(1+(u/a_n)^d\big).
\)
There are only \(\rho_{k,n}\) blocks in the partition, so we
also have \(|\mathcal J_n(z,u)|\le\rho_{k,n}\).
Both estimates hold uniformly in \(z\), \(u\), and all sufficiently
large \(n\).

By definition,
\(\mathcal D_n(z,\gamma)
=\mathcal J_n(z,c_*r_nR(\gamma))\).
Since \(b_n=r_n/a_n\), the preceding estimates give
\[
|\mathcal D_n(z,\gamma)|
\le
\min\big\{
\rho_{k,n},
C_d\big(1+c_*^d(b_nR(\gamma))^d\big)
\big\}.
\]
Choose
\(C_*\ge\max\{1,C_d,C_dc_*^d\}\), independently of \(n\),
and define
\[
g_n(u):=
\min\big\{\rho_{k,n},C_*\big(1+(b_nu)^d\big)\big\},
\qquad u\ge0.
\]
Then
\(|\mathcal D_n(z,\gamma)|\le g_n(R(\gamma))\).
Moreover, \(\rho_{k,n}\to\infty\) and \(C_*\ge1\), so
\(1\le g_n(u)\le\rho_{k,n}\) for every \(u\ge0\) and all
sufficiently large \(n\).

The next estimate compares the block count with the power of
the range appearing in the exponential moment assumption.
Recall that \(q=\min\{\alpha,d\}>0\), and put
\(\kappa_n:=v_n^{-q}+\rho_{k,n}r_n^q\).
The cutoff satisfies
\(v_n^q=8c_0^{-1}\log m_n\), where
\(m_n=s_n^{-(k-1)}\to\infty\). Consequently,
\(
v_n^{-q}=\frac{c_0}{8\log m_n}\longrightarrow0.
\)
If \(q=d\), then
\[
\rho_{k,n}r_n^q
=n^kr_n^{d(k-1)}r_n^d
=(nr_n^d)^k
=s_n^k
\longrightarrow0.
\]
If \(q=\alpha<d\), assumption
\eqref{eq:relaxed_tail_scale_condition} gives
\(\rho_{k,n}r_n^\alpha\log n\to0\).
Since \(\log n\ge1\) eventually,
\(0\le\rho_{k,n}r_n^\alpha
\le\rho_{k,n}r_n^\alpha\log n\to0\).
Thus, \(\kappa_n\to0\) in both cases.

We establish the bound on \(g_n(u)\) separately below and above
the range \(r_n^{-1}\). These two ranges cover every \(u>v_n\)
for all sufficiently large \(n\): indeed,
\(r_nv_n=a_nb_nv_n\to0\), by \(a_n\to0\) and
\eqref{eq:range-cutoff}, so \(v_n<r_n^{-1}\) eventually.
Suppose first that \(v_n<u\le r_n^{-1}\).
Then \(1\le v_n^{-q}u^q\).
Also, \(b_n^d=\rho_{k,n}r_n^d\), and \(q\le d\), so
\[
\begin{aligned}
(b_nu)^d
=\rho_{k,n}r_n^du^d
=\rho_{k,n}r_n^qu^q(r_nu)^{d-q}
\le\rho_{k,n}r_n^qu^q.
\end{aligned}
\]
The last inequality uses \(r_nu\le1\); when \(q=d\), the
factor \((r_nu)^{d-q}\) equals one. Therefore,
\[
\begin{aligned}
g_n(u)
&\le C_*\big(1+(b_nu)^d\big)
&\le C_*\big(v_n^{-q}+\rho_{k,n}r_n^q\big)u^q
&=C_*\kappa_nu^q.
\end{aligned}
\]

Suppose next that \(u>r_n^{-1}\).
Here we use the bound by the total number of blocks.
Since \((r_nu)^q>1\),
\[
g_n(u)
\le\rho_{k,n}
\le\rho_{k,n}r_n^qu^q
\le\kappa_nu^q
\le C_*\kappa_nu^q.
\]
Combining the two ranges gives the uniform estimate
\(
g_n(u)\le C_*\kappa_nu^q\) when
\( u>v_n.
\)
The bound \(g_n(u)\le\rho_{k,n}\) is essential in the second
range: it prevents the estimate from continuing to grow as
\(u^d\) when \(q<d\).
Define the nonnegative random sum
\[
S_n:=
\sum_{(z,\gamma)\in\widetilde{\mathcal P}_n^{\mathrm f}}
g_n(R(\gamma)).
\]
The far process contains finitely many points almost surely,
and \(g_n\le\rho_{k,n}\), so \(S_n<\infty\) almost surely.
For all sufficiently large \(n\), the lower bound \(g_n\ge1\)
gives
\(
N_n^{\mathrm f}
\le S_n.
\)
The definition of \(\mathcal D_n\) as a union gives, without
requiring the individual influence sets to be disjoint,
\[
\begin{aligned}
|\mathcal D_n|
=
\Big|
\bigcup_{(z,\gamma)\in\widetilde{\mathcal P}_n^{\mathrm f}}
\mathcal D_n(z,\gamma)
\Big|
\le
\sum_{(z,\gamma)\in\widetilde{\mathcal P}_n^{\mathrm f}}
|\mathcal D_n(z,\gamma)|
\le
\sum_{(z,\gamma)\in\widetilde{\mathcal P}_n^{\mathrm f}}
g_n(R(\gamma))
=S_n.
\end{aligned}
\]
These inequalities also hold when there are no far trajectories,
in which case the sums and the union are empty. Thus,
\(\max\{N_n^{\mathrm f},|\mathcal D_n|\}\le S_n\).
Set
\(\theta_n:=c_0/(2C_*\kappa_n)\).
Since \(\kappa_n>0\) and \(\kappa_n\to0\), we have
\(\theta_n>0\) and \(\theta_n\to\infty\).
For every \(u>v_n\), the uniform estimate on \(g_n\) yields
\(
\theta_ng_n(u)
\le
\frac{c_0}{2C_*\kappa_n}
C_*\kappa_nu^q
=\frac{c_0}{2}u^q.
\)
This choice allows the exponential parameter to diverge while
keeping the exponential contribution of each far mark below
the integrable range weight.

The intensity measure of
\(\widetilde{\mathcal P}_n^{\mathrm f}\) is
\(n\,dz\,\mathbf{1}\{R(\gamma)>v_n\}\Q(d\gamma)\)
on \(\Lambda\times\CC_0\). For each fixed \(n\), the function
\(\theta_ng_n(R(\gamma))\) is bounded by
\(\theta_n\rho_{k,n}\), and this intensity measure is finite.
Thus, the Poisson exponential formula
\cite{last-penrose2018} applies and gives
\[
\begin{aligned}
\log\mathbb E e^{\theta_nS_n}
&=
n\int_\Lambda\int_{\CC_0}
\big(e^{\theta_ng_n(R(\gamma))}-1\big)
\mathbf{1}\{R(\gamma)>v_n\}
\,\Q(d\gamma)\,dz\\
&=
n\mathbb E_\Q\big[
\big(e^{\theta_ng_n(R)}-1\big)
\mathbf{1}\{R>v_n\}
\big],
\end{aligned}
\]
where, inside the expectation, \(\Gamma\) has law \(\Q\) and
\(R=R(\Gamma)\). The second equality uses
\(\lambda_\Lambda(\Lambda)=1\) and the fact that the integrand
does not depend on \(z\).
On the event \(\{R>v_n\}\), we have
\[
e^{\theta_ng_n(R)}-1
\le e^{\theta_ng_n(R)}
\le e^{c_0R^q/2}
=e^{c_0R^q}e^{-c_0R^q/2}
\le e^{-c_0v_n^q/2}e^{c_0R^q}.
\]
Recall that
\(M_0=\mathbb E_\Q e^{c_0R(\Gamma)^q}<\infty\).
Taking expectations and using
\(v_n^q=8c_0^{-1}\log m_n\), we obtain
\[
\begin{aligned}
\mathbb E_\Q\big[
\big(e^{\theta_ng_n(R)}-1\big)
\mathbf{1}\{R>v_n\}
\big]
&\le
e^{-c_0v_n^q/2}
\mathbb E_\Q\big[
e^{c_0R^q}\mathbf{1}\{R>v_n\}
\big]\\
&\le M_0e^{-c_0v_n^q/2}\\
&=M_0e^{-4\log m_n}
=M_0m_n^{-4}.
\end{aligned}
\]
Since \(n/\rho_{k,n}=m_n\), division of the Poisson
exponential formula by \(\rho_{k,n}\) now gives
\[
\begin{aligned}
\frac1{\rho_{k,n}}\log\mathbb E e^{\theta_nS_n}
=
m_n\mathbb E_\Q\big[
\big(e^{\theta_ng_n(R)}-1\big)
\mathbf{1}\{R>v_n\}
\big]
\le M_0m_n\,m_n^{-4}
=M_0m_n^{-3}
\longrightarrow0.
\end{aligned}
\]

Finally, fix \(\delta>0\).
Since \(\theta_n>0\), exponential Markov inequality gives
\[
\begin{aligned}
\mathbb P(S_n\ge\delta\rho_{k,n})
&=
\mathbb P\big(
e^{\theta_nS_n}
\ge e^{\theta_n\delta\rho_{k,n}}
\big)
\le
e^{-\theta_n\delta\rho_{k,n}}
\mathbb E e^{\theta_nS_n}.
\end{aligned}
\]
Taking logarithms, with \(\log0=-\infty\), and dividing by
\(\rho_{k,n}\), we conclude that
\[
\frac1{\rho_{k,n}}
\log\mathbb P(S_n\ge\delta\rho_{k,n})
\le
-\delta\theta_n+M_0m_n^{-3}
\longrightarrow-\infty.
\]
Indeed, \(\delta\) is fixed and positive,
\(\theta_n\to\infty\), and \(m_n^{-3}\to0\).
The pointwise bounds by \(S_n\) give the event inclusions
\(
\{N_n^{\mathrm f}\ge\delta\rho_{k,n}\}
\subseteq
\{S_n\ge\delta\rho_{k,n}\},
\) and \(
\{|\mathcal D_n|\ge\delta\rho_{k,n}\}
\subseteq
\{S_n\ge\delta\rho_{k,n}\}.
\)
Each of the two normalized logarithmic probabilities in the
statement is therefore bounded above by a quantity tending
to \(-\infty\). This proves both probabilistic conclusions
for every \(\delta>0\).
\end{proof}

\section{Poisson approximation}
\label{sec:1proof}

We prove Proposition~\ref{pr:55} by coupling the contribution of each
block with a Poisson random measure. Convergence in one block makes a
coupling error unlikely, and an exponential-moment bound controls the
size of that error. Independence across blocks then gives exponential
equivalence of the averages. We state these two inputs and prove the
proposition before establishing the local approximation and the counting
estimate. Only sparsity and \eqref{eq:curve_moment_condition} are used
in this section.

For \(i\le\rho_{k,n}\), define
\[
\eta_{k,n}^{(i)}:=
\sum_{\substack{\mathcal Y\subseteq\PP_{Q_i}\\|\mathcal Y|=k}}
s_{\mathrm{conn},n}(\mathcal Y)
s_{\mathrm{iso},n}(\mathcal Y,\PP_{Q_i})
\delta_{r_n^{-1}\bar{\mathcal Y}}.
\]
Thus, \(\eta_{k,n}=\rho_{k,n}^{-1}\sum_i\eta_{k,n}^{(i)}\).
Recall that \(\zeta_{k,n}\) is the corresponding average of independent
Poisson random measures \(\zeta_k^{(i)}\) with intensity \(\tau_k\).
For all sufficiently large \(n\), the block side length \(a_n\) is
less than \(1/2\). The starting-point differences within one block
are then their ordinary Euclidean differences, so the centering rule
commutes with translations between blocks. Together with independence
of the restrictions of the marked Poisson process, this shows that the
block measures are independent and identically distributed.

To state the counting estimate, let \(\Theta\le\Theta_n\) be a finite
intensity measure on \(\CC_{\mathbb T}\), and let \(\Pi\) be a
Poisson process with intensity \(\Theta\). Let
\(h:\CC_{\mathbb T}^k\to\{0,1\}\) be symmetric and measurable,
vanishing unless the trajectories form a connected cluster. Write
\(\Pi_{\ne}^k\) for ordered tuples of distinct trajectories, and put
\[
N_h:=\frac1{k!}\sum_{\underline\gamma'\in\Pi_{\ne}^k}
h(\underline\gamma')s_{\mathrm{iso},n}(\{\underline\gamma'\},\Pi),
\qquad
\Lambda_h:=\frac1{k!}\int h\,d\Theta^{\otimes k}.
\]
Here \(\{\underline\gamma'\}\) denotes the unordered set of entries.
For an integer \(m\ge1\), write \((z)_m=z(z-1)\cdots(z-m+1)\),
with \((z)_0=1\), and let \(A\subseteq\Lambda\) be Borel, with
Lebesgue measure \(|A|\). The first input uses the disjointness of
isolated connected clusters to bound their factorial moments. Its
constant \(C\) depends only on \(d,k,L\) and the path moments in
\eqref{eq:curve_moment_condition}.

\begin{lemma}[Factorial moments of isolated clusters]
\label{lem:isolated-factorial}
For every integer \(m\ge1\) and every \(\theta>0\),
\[
\mathbb E[(N_h)_m]\le\Lambda_h^m,
\qquad
\mathbb E e^{\theta N_h}\le
\exp\{\Lambda_h(e^\theta-1)\}.
\]
If \(h\) vanishes whenever all starting points lie outside \(A\),
then \(\Lambda_h\le C\rho_{k,n}|A|\).
\end{lemma}

The second input compares a single block with its Poisson limit in
total variation. This provides a maximal coupling whose probability
of disagreement tends to zero; the preceding moment estimate will then
control the contribution on the disagreement event.

\begin{lemma}[Approximation in one cube]
\label{lem-equival}
If \(\rho_{k,n}\to\infty\), \(nr_n^d\to0\), and
\eqref{eq:curve_moment_condition} holds, then
\[
\sup_{i\le\rho_{k,n}}
d_{\mathrm{TV}}\big(\mathcal L(\eta_{k,n}^{(i)}),
\mathcal L(\zeta_k^{(i)})\big)\longrightarrow0.
\]
\end{lemma}

\begin{proof}[Proof of Proposition~\ref{pr:55}]
For each block, choose a maximal coupling
\((\widehat\eta_{k,n}^{(i)},\widehat\zeta_k^{(i)})\), and take the
product of these coupling laws over the blocks. Write
\(\widehat{\mathbb P}\) and \(\widehat{\mathbb E}\) for probability
and expectation on this product space. Each marginal has the required
block law, and the pairs are independent. Put
\(D_{i,n}:=d_{\mathrm{TV}}(\widehat\eta_{k,n}^{(i)},
\widehat\zeta_k^{(i)})\) and
\(\epsilon_n:=\sup_i\widehat{\mathbb P}(D_{i,n}>0)\).
Maximality and Lemma~\ref{lem-equival} give
\(\epsilon_n=\sup_i d_{\mathrm{TV}}(\mathcal L(\eta_{k,n}^{(i)}),
\mathcal L(\zeta_k^{(i)}))\to0\).

Define \(\widehat\eta_{k,n}:=\rho_{k,n}^{-1}
\sum_i\widehat\eta_{k,n}^{(i)}\) and
\(\widehat\zeta_{k,n}:=\rho_{k,n}^{-1}
\sum_i\widehat\zeta_k^{(i)}\). Independence of the block marginals
shows that these averages have the laws of \(\eta_{k,n}\) and
\(\zeta_{k,n}\), respectively. Homogeneity and subadditivity of total
variation give
\[
\begin{aligned}
d_{\mathrm{TV}}(\widehat\eta_{k,n},\widehat\zeta_{k,n})
&=\Big\|\frac1{\rho_{k,n}}
\sum_i(\widehat\eta_{k,n}^{(i)}-\widehat\zeta_k^{(i)})\Big\|_{\mathrm{TV}}
\le\frac1{\rho_{k,n}}
\sum_i\|\widehat\eta_{k,n}^{(i)}-\widehat\zeta_k^{(i)}\|_{\mathrm{TV}}
=\frac1{\rho_{k,n}}\sum_iD_{i,n}.
\end{aligned}
\]

We next bound the exponential moment of each summand. In
Lemma~\ref{lem:isolated-factorial}, take
\(\Theta=\Theta_n|_{\CC_{\mathbb T,Q_i}}\) and let \(h\) be the
connection indicator. Since \(\rho_{k,n}|Q_i|=1\), this gives
\(\sup_{n,i}\mathbb E e^{t\eta_{k,n}^{(i)}(E_k)}<\infty\)
for each \(t>0\), with the supremum over sufficiently large \(n\).
Also, \(\mathbb E e^{t\zeta_k^{(i)}(E_k)}
=\exp\{\tau_k(E_k)(e^t-1)\}<\infty\).
Although the two measures within a maximal coupling need not be
independent, Cauchy--Schwarz applies to their joint law. Since
\(D_{i,n}\le\widehat\eta_{k,n}^{(i)}(E_k)
+\widehat\zeta_k^{(i)}(E_k)\), for every \(a>0\) we obtain
\[
\begin{aligned}
\widehat{\mathbb E}[(e^{aD_{i,n}}-1)
\mathbf{1}\{D_{i,n}>0\}]
\le\sqrt{\widehat{\mathbb E}e^{2aD_{i,n}}}
\sqrt{\widehat{\mathbb P}(D_{i,n}>0)}
\le\big(\mathbb E e^{4a\eta_{k,n}^{(i)}(E_k)}\big)^{1/4}
\big(\mathbb E e^{4a\zeta_k^{(i)}(E_k)}\big)^{1/4}
\epsilon_n^{1/2},
\end{aligned}
\]
which is at most  \(C_a\epsilon_n^{1/2}\).
The constant \(C_a\) is independent of \(n\) and \(i\). Thus,
\(\sup_i\widehat{\mathbb E}e^{aD_{i,n}}\le1+C_a\epsilon_n^{1/2}\to1\),
which supplies the moment convergence needed to sum the errors.

For \(\delta>0\), exponential Markov inequality and independence of
the pairs now yield
\[
\begin{aligned}
\frac1{\rho_{k,n}}\log\widehat{\mathbb P}
\big(d_{\mathrm{TV}}(\widehat\eta_{k,n},\widehat\zeta_{k,n})>\delta\big)
&\le\frac1{\rho_{k,n}}\log\widehat{\mathbb P}
\Big(\sum_iD_{i,n}>\delta\rho_{k,n}\Big)\\
&\le-a\delta+\frac1{\rho_{k,n}}
\sum_i\log\widehat{\mathbb E}e^{aD_{i,n}}\\
&\le-a\delta+\log(1+C_a\epsilon_n^{1/2}).
\end{aligned}
\]
For each fixed \(a\), taking the upper limit in \(n\) gives a bound
of \(-a\delta\). Letting \(a\to\infty\) gives \(-\infty\), for
every \(\delta>0\). The coupled averages therefore establish the
exponential equivalence asserted in Proposition~\ref{pr:55}.
\end{proof}

\subsection{Approximation in one cube}
\label{ssec:equival}

We prove Lemma~\ref{lem-equival} by first removing isolation and then
approximating the resulting point process of connected tuples. The first
error involves a cluster and one additional blocking trajectory, while
the second involves two clusters with shared trajectories. A common
geometric integral bounds both errors, and a separate intensity
calculation identifies the limiting cluster shapes. We give the precise
inputs to the Poisson approximation theorem before proving that each
error vanishes.

Fix a block index \(i\), and put
\(X_{i,n}:=\CC_{\mathbb T,Q_i}\) and
\(K_{i,n}:=\Theta_n|_{X_{i,n}}\).
Let \(D_{i,n}^{\mathrm{conn}}\subseteq X_{i,n}^k\) be the symmetric
measurable set on which the connection indicator is one, with the
diagonals of repeated trajectories omitted. On this set define
\(f_{i,n}(\underline\gamma'):=r_n^{-1}\overline{\{\underline\gamma'\}}\).
The point process obtained by omitting isolation is
\[
\widetilde\eta_{k,n}^{(i)}
:=\frac1{k!}\sum_{\underline\gamma'\in(\PP_{Q_i})_{\ne}^k\cap D_{i,n}^{\mathrm{conn}}}
\delta_{f_{i,n}(\underline\gamma')},
\qquad
\alpha_{i,n}:=\mathbb E\widetilde\eta_{k,n}^{(i)}.
\]
The summands are restricted to \(D_{i,n}^{\mathrm{conn}}\), where
\(f_{i,n}\) is defined. The factor \(1/k!\) counts each unordered
tuple once. Starting positions have a diffuse intensity, so the
diagonals have zero product intensity and do not affect the integrals
below.
For \(1\le p\le k-1\), define the overlap integral
\[
\mathfrak r_{i,n,p}:=
\int_{X_{i,n}^{p}}
\Big(\int_{X_{i,n}^{k-p}}
\mathbf{1}\{(\gamma'_1,\ldots,\gamma'_k)
\in D_{i,n}^{\mathrm{conn}}\}
\,dK_{i,n}^{\otimes(k-p)}\Big)^2
\,dK_{i,n}^{\otimes p},
\qquad
\mathfrak r_{i,n}:=\max_{1\le p\le k-1}\mathfrak r_{i,n,p}.
\]
The inner integral keeps the first \(p\) trajectories fixed. Squaring
it introduces two separate choices of the other \(k-p\) trajectories,
so the expanded integral contains \(2k-p\) trajectories in total.
Write \(d_{\mathrm{KR}}\) for the Kantorovich--Rubinstein distance
on laws of finite point measures whose test functions are measurable
functions \(F\) satisfying
\(|F(\nu)-F(\nu')|\le\sup_A|\nu(A)-\nu'(A)|\),
where \(A\) ranges over Borel subsets of \(E_k\).
The Poisson case of \cite[Theorem~3.1]{DecreusefondSchulteThaele2016}
gives, for the symmetric measurable tuple map above and finite
intensities,
\begin{equation}
d_{\mathrm{KR}}\big(\mathcal L(\widetilde\eta_{k,n}^{(i)}),
\mathcal L(\zeta_k^{(i)})\big)
\le\|\alpha_{i,n}-\tau_k\|_{\mathrm{TV}}
+\frac{2^{k+1}}{k!}\mathfrak r_{i,n}.
\label{eq:local-poisson-approximation-bound}
\end{equation}
Here we use the total variation norm fixed in
Section~\ref{sec:proof-outline}. The intensity term in the cited
bound is \(\sup_A|\alpha_{i,n}(A)-\tau_k(A)|\), which is at
most this variation norm.

For the common geometric estimate, consider  trajectories
\(\gamma'_\ell=x_\ell+r_n\gamma_\ell\) on the torus, with
\(\gamma_\ell\in\CC_0\), and put
\(R_j:=\max_{1\le\ell\le j}R(\gamma_\ell)\).
Join two indices if their trajectories approach within \(Lr_n\) at
some time, and let \(c_{j,n}^{\mathrm{gr}}\) be the indicator that
this graph is connected. This graph condition is used only as an upper
bound; the cluster condition still requires the full diameter to be
small at one common time. For a Borel set \(A\subseteq\Lambda\), put
\(I_{j,n}(A):=\int\mathbf{1}\{\gamma'_1(0)\in A\}
 c_{j,n}^{\mathrm{gr}}(\underline\gamma')\,
 d\Theta_n^{\otimes j}(\underline\gamma')\).
All graph indicators are measurable, since the infimum of the distance
between two continuous paths on the compact time interval is a
continuous function of those paths.

\begin{proof}[Proof of Lemma~\ref{lem-equival}]
We first establish the bound on \(I_{j,n}(A)\), for
\(2\le j\le2k-1\). If \(\ell\) and \(m\) are adjacent, the
triangle inequality at a time of approach gives
\(d_{\mathbb T}(x_\ell,x_m)
\le r_n(R(\gamma_\ell)+L+R(\gamma_m))
\le r_n(L+2R_j)\).
When the graph is connected, a path of at most \(j-1\) edges joins
index \(1\) to each other index. Hence, every starting point belongs
to the torus ball about \(x_1\) of radius
\((j-1)r_n(L+2R_j)\). Write \(B_{\mathbb T}(x,u)\) for that
ball with center \(x\) and radius \(u\). Its volume is at most
\(C_du^d\) for every \(u>0\): projection of the Euclidean ball
onto the torus cannot increase volume. Integrating the starting points
and then the displacement paths gives
\begin{equation}
\begin{aligned}
I_{j,n}(A)
&\le n^j\int_{\CC_0^j}\int_A
\prod_{\ell=2}^j
|B_{\mathbb T}(x_1,(j-1)r_n(L+2R_j))|
\,dx_1\,\Q^{\otimes j}(d\underline\gamma)\\
&\le C_jn^j|A|r_n^{d(j-1)}
\mathbb E_{\Q^{\otimes j}}(1+R_j)^{d(j-1)}
\le C_jn^j|A|r_n^{d(j-1)}.
\end{aligned}
\label{eq:connected-pattern-integral}
\end{equation}
The last moment is finite because
\((1+R_j)^u\le\sum_{\ell=1}^j(1+R(\gamma_\ell))^u\) for \(u\ge0\),
and \eqref{eq:curve_moment_condition} applies. In particular, for
\(A=Q_i\), the deterministic factor equals
\(n^jr_n^{d(j-1)}/\rho_{k,n}
=n^{j-k}r_n^{d(j-k)}=s_n^{j-k}\).
This identity is the source of the powers of the sparsity parameter
in both error estimates.
Define \(\chi_n(\underline\gamma',\gamma')\) to be the indicator that
\(\gamma'\) approaches at least one member of
\(\{\underline\gamma'\}\) within \(Lr_n\).
If \(\widetilde\eta_{k,n}^{(i)}\ne\eta_{k,n}^{(i)}\), some
connected \(k\)-tuple has such an additional trajectory in
\(\PP_{Q_i}\). Counting all these choices, rather than only the
first blocker, bounds the probability of a discrepancy. The
multivariate Mecke formula \cite{last-penrose2018} gives
\[
\begin{aligned}
d_{\mathrm{TV}}\big(\mathcal L(\eta_{k,n}^{(i)}),
\mathcal L(\widetilde\eta_{k,n}^{(i)})\big)
&\quad\le\mathbb P(\eta_{k,n}^{(i)}\ne\widetilde\eta_{k,n}^{(i)})\\
&\quad\le\frac1{k!}\mathbb E
\sum_{(\gamma'_1,\ldots,\gamma'_{k+1})\in(\PP_{Q_i})_{\ne}^{k+1}}
s_{\mathrm{conn},n}(\{\gamma'_1,\ldots,\gamma'_k\})
\chi_n((\gamma'_1,\ldots,\gamma'_k),\gamma'_{k+1})\\
&\quad=\frac1{k!}\int_{X_{i,n}^{k+1}}
s_{\mathrm{conn},n}(\{\gamma'_1,\ldots,\gamma'_k\})
\chi_n((\gamma'_1,\ldots,\gamma'_k),\gamma'_{k+1})
\,dK_{i,n}^{\otimes(k+1)}\\
&\quad\le\frac1{k!}I_{k+1,n}(Q_i)
\le Cs_n\longrightarrow0.
\end{aligned}
\]
For the penultimate inequality, the first \(k\) trajectories form a
complete proximity graph, and the blocker has an edge to that graph.
Their union is therefore connected. Dropping the block restrictions
on all starting points except the first only increases the integral,
so \eqref{eq:connected-pattern-integral} applies.

We next bound the overlap integral in
\eqref{eq:local-poisson-approximation-bound}. Fix
\(1\le p\le k-1\) and expand the square in its definition by
Tonelli's theorem. The two \(k\)-tuples share the first \(p\)
trajectories; each tuple has a complete proximity graph, although their
connection times need not agree. Since \(p\ge1\), the union graph
is connected. Consequently,
\[
\begin{aligned}
\mathfrak r_{i,n,p}
&=\int_{X_{i,n}^{2k-p}}
 s_{\mathrm{conn},n}(\{\gamma'_1,\ldots,\gamma'_k\})
 s_{\mathrm{conn},n}(\{\gamma'_1,\ldots,\gamma'_p,
 \gamma'_{k+1},\ldots,\gamma'_{2k-p}\})
 \,dK_{i,n}^{\otimes(2k-p)}\\
&\le I_{2k-p,n}(Q_i)
\le C\frac{n^{2k-p}r_n^{d(2k-p-1)}}{n^kr_n^{d(k-1)}}
=C(nr_n^d)^{k-p}\longrightarrow0.
\end{aligned}
\]
There are only \(k-1\) choices of \(p\), so
\(\mathfrak r_{i,n}\to0\), uniformly in the block index.
Also, Mecke's formula and \eqref{eq:connected-pattern-integral} with
\(j=k\) give \(\alpha_{i,n}(E_k)\le I_{k,n}(Q_i)/k!\le C\).
The intensity \(\tau_k\) is finite by Section~\ref{sec:model}.
Thus, both intensities in the stated Poisson approximation bound are
finite, and the symmetry and measurability hypotheses hold by the
construction of \(D_{i,n}^{\mathrm{conn}}\) and \(f_{i,n}\).

It remains to compare the two intensities. For
\(y=(y_1,\ldots,y_{k-1})\in(\mathbb R^d)^{k-1}\), set \(y_0=0\)
and define
\(\psi(\underline\gamma,y):=
\operatorname{cen}(\gamma_1,y_1+\gamma_2,\ldots,y_{k-1}+\gamma_k)\).
Let \(c(\underline\gamma,y)\) indicate that these 
trajectories have full diameter at most \(L\) at one common time.
Let \(c_n(\underline\gamma,y)\) indicate the same condition on the
torus after multiplication by \(r_n\), at threshold \(Lr_n\).
For a cube represented in \([0,1]^d\), define
\(p_{i,n}(y):=\rho_{k,n}|Q_i\cap
\bigcap_{j=1}^{k-1}(Q_i-r_ny_j)|\), using translations.
Finally, write \(\lambda:=dy\,\Q^{\otimes k}(d\underline\gamma)\)
for the product measure on offsets and displacement paths.

For a Borel set \(A\subseteq E_k\), Mecke's formula expresses
\(\alpha_{i,n}(A)\) as the integral of the connection and shape
indicators against \(n^k\,dx_1\cdots dx_k\,
\Q^{\otimes k}(d\underline\gamma)/k!\), with every \(x_j\in Q_i\).
Make the change of variables \(x_1=x\) and
\(x_{j+1}=x+r_ny_j\). Its Jacobian is \(r_n^{d(k-1)}\), and
\(n^kr_n^{d(k-1)}=\rho_{k,n}\). On this domain the rescaled
centered shape is exactly \(\psi(\underline\gamma,y)\). Integrating
first in the anchor \(x\) gives
\[
\begin{aligned}
\alpha_{i,n}(A)
&=\frac{\rho_{k,n}}{k!}\int
 c_n(\underline\gamma,y)\mathbf{1}\{\psi(\underline\gamma,y)\in A\}
 \Big(\int_{Q_i}\prod_{j=1}^{k-1}
 \mathbf{1}\{x+r_ny_j\in Q_i\}\,dx\Big)\,d\lambda\\
&=\frac1{k!}\int p_{i,n}(y)c_n(\underline\gamma,y)
 \mathbf{1}\{\psi(\underline\gamma,y)\in A\}\,d\lambda,\\
\tau_k(A)
&=\frac1{k!}\int c(\underline\gamma,y)
 \mathbf{1}\{\psi(\underline\gamma,y)\in A\}\,d\lambda.
\end{aligned}
\]
For fixed \(y\), put \(M_y:=\max_j|y_j|\). If an anchor \(x\)
lies at distance greater than \(r_nM_y\) from the boundary of \(Q_i\),
then every \(x+r_ny_j\) lies in \(Q_i\). The missing anchors are
therefore contained in an inner boundary layer of that width. The
volume of such a layer in a cube of side \(a_n\) is at most
\(2d\min\{a_n^d,r_nM_ya_n^{d-1}\}\), whence
\(1-p_{i,n}(y)\le2d\min\{1,b_nM_y\}\to0\), uniformly in
\(i\). Both bounds are independent of the location of the block.

To dominate the integrand uniformly, put \(C_L:=L+2R_k\).
If \(c=1\), comparison with the first trajectory at a connection
time gives \(|y_j|\le C_L\) for every \(j\). If
\(p_{i,n}c_n>0\), choose an anchor for which all starting points
belong to the same cube. The triangle inequality on the torus gives
\(d_{\mathbb T}(x,x+r_ny_j)\le r_nC_L\).
Because the two starting points are in one cube of side less than
\(1/2\), their difference is the principal torus difference, and
\(d_{\mathbb T}(x,x+r_ny_j)=r_n|y_j|\). Thus, the same bound on
\(y_j\) holds in both cases.

It follows that \(|p_{i,n}c_n-c|\) is bounded by the indicator of
\(\max_j|y_j|\le C_L\). Its integral is at most
\(C\mathbb E_{\Q^{\otimes k}}(1+R_k)^{d(k-1)}<\infty\).
Total variation decreases under a measurable pushforward. Applying
this to the two intensity formulas and using the uniform pointwise
convergence just established yields
\[
\begin{aligned}
\sup_{i\le\rho_{k,n}}\|\alpha_{i,n}-\tau_k\|_{\mathrm{TV}}
&\le\frac1{k!}\int
 \sup_{i\le\rho_{k,n}}
 |p_{i,n}(y)c_n(\underline\gamma,y)-c(\underline\gamma,y)|
 \,d\lambda
\longrightarrow0.
\end{aligned}
\]
The limit follows from dominated convergence with the integrable
indicator above. In particular, the supremum over measurable sets
of shapes is controlled before taking the limit.

Finally, an indicator of a measurable set of point measures is
1-Lipschitz for total variation: its values differ by at most one,
and for two unequal integer-valued point measures
\(\sup_A|\nu(A)-\nu'(A)|\ge1\). Thus, the total variation distance between
laws is at most their \(d_{\mathrm{KR}}\) distance. The triangle
inequality, the isolation bound, and
\eqref{eq:local-poisson-approximation-bound} now give
\[
\begin{aligned}
d_{\mathrm{TV}}\big(\mathcal L(\eta_{k,n}^{(i)}),
\mathcal L(\zeta_k^{(i)})\big)
&\le Cs_n+\|\alpha_{i,n}-\tau_k\|_{\mathrm{TV}}
 +\frac{2^{k+1}}{k!}\mathfrak r_{i,n}
\longrightarrow0.
\end{aligned}
\]
The isolation term vanishes by \(s_n\to0\), the intensity term by
the preceding dominated-convergence argument, and the overlap term
by the estimates for \(1\le p\le k-1\). Each bound is uniform in
\(i\), which proves the stated local approximation.
\end{proof}

\subsection{Factorial moments of isolated clusters}
\label{ssec:ui}

We now prove the counting estimate used in the coupling argument.
Isolation prevents two distinct counted clusters from sharing a
trajectory, which permits one application of the multivariate Mecke
formula to their joint factorial moments. Expanding the exponential
in factorial moments then gives the required bound. Choosing a starting
point in a prescribed set as anchor makes the same estimate available
for blocks, the whole torus, and boundary layers.

\begin{proof}[Proof of Lemma~\ref{lem:isolated-factorial}]
Suppose that two distinct counted clusters \(\mathcal Y\) and
\(\mathcal Z\) share a trajectory. Since both have size \(k\),
there is a trajectory in \(\mathcal Y\setminus\mathcal Z\).
At a time when \(\mathcal Y\) has diameter at most \(Lr_n\), this
trajectory is within \(Lr_n\) of the shared trajectory. It is an
outside trajectory for \(\mathcal Z\), contradicting isolation of
\(\mathcal Z\). Therefore, distinct counted clusters are disjoint.

The falling factorial \((N_h)_m\) counts ordered lists of \(m\)
distinct counted clusters. Every nonzero term consequently uses
\(km\) distinct trajectories. Ordering the entries within each
cluster introduces the factor \((k!)^m\). Mecke's formula and the
bound of one on every isolation indicator give
\[
\begin{aligned}
\mathbb E[(N_h)_m]
\le\frac1{(k!)^m}\int
 \prod_{j=1}^m
 h(\gamma'_{(j-1)k+1},\ldots,\gamma'_{jk})
 \,d\Theta^{\otimes km}
=\prod_{j=1}^m
 \Big(\frac1{k!}\int h\,d\Theta^{\otimes k}\Big)
=\Lambda_h^m.
\end{aligned}
\]
The equality holds because the remaining integrand factors over
\(m\) disjoint groups of integration variables. This is a bound on
joint factorial moments and does not assert independence of the
cluster counts themselves.
For every nonnegative integer \(z\), the binomial formula gives
\(e^{\theta z}=(1+(e^\theta-1))^z
=\sum_{m=0}^{\infty}(z)_m(e^\theta-1)^m/m!\); the summands vanish
for \(m>z\). Since \(\theta>0\), they are nonnegative, so monotone
convergence and the factorial-moment bound yield
\[
\mathbb E e^{\theta N_h}
=\sum_{m=0}^{\infty}\frac{(e^\theta-1)^m}{m!}
 \mathbb E[(N_h)_m]
\le\sum_{m=0}^{\infty}\frac{(\Lambda_h(e^\theta-1))^m}{m!}
=\exp\{\Lambda_h(e^\theta-1)\}.
\]

For the spatial assertion, the condition on \(h\) bounds it by the
connection indicator times
\(\sum_{\ell=1}^k\mathbf{1}\{\gamma'_\ell(0)\in A\}\).
Use \(\Theta\le\Theta_n\) and symmetry to choose the first starting
point as anchor, at the cost of a factor \(k\). A connected cluster
has a connected proximity graph, so the geometric integral already
proved in \eqref{eq:connected-pattern-integral} gives
\[
\begin{aligned}
\Lambda_h
\le\frac{k}{k!}\int
 \mathbf{1}\{\gamma'_1(0)\in A\}
 s_{\mathrm{conn},n}(\{\gamma'_1,\ldots,\gamma'_k\})
 \,d\Theta_n^{\otimes k}
\le\frac{k}{k!}I_{k,n}(A)
\le Cn^k|A|r_n^{d(k-1)}
=C\rho_{k,n}|A|.
\end{aligned}
\]
This proves the spatial bound and completes all parts of the lemma.
\end{proof}

\section{Spatial localization}
\label{sec:2proof}

We prove Proposition~\ref{pr:56} by separating the effect of far
trajectories from the localization errors of the short process.
The first comparison uses the affected blocks from
Lemma~\ref{lem:far-influence}, and the second uses the deterministic
range bound inside enlarged cubes. We state both inputs and deduce the
proposition before proving them. The short-process error estimates will
also supply the exact equality needed for the hard-core lower bound.

Throughout this section, assume the hypotheses of Theorem~\ref{thm:fine}.
Let \(\xi_{k,n}^{\mathrm s}\) and \(\eta_{k,n}^{\mathrm s}\) be
defined as \(\xi_{k,n}\) and \(\eta_{k,n}\), with \(\PP_n\)
replaced by \(\PP_n^{\mathrm s}\). For each block \(Q_i\), let
\(Z_{i,n}^{\mathrm{iso},+,\mathrm s}\) count connected short clusters
whose starting points all lie in \(Q_i\), which are isolated relative
to \(\PP_{Q_i}^{\mathrm s}\), and which are approached within
\(Lr_n\) by a short trajectory starting in
\(Q_i^{\mathrm s,+}\setminus Q_i\). Let
\(\widetilde Z_{i,n}^{\mathrm s}\) count connected clusters that are
isolated relative to \(\PP_{Q_i^{\mathrm s,+}}^{\mathrm s}\), have
all starting points in \(Q_i^{\mathrm s,+}\), and have starting
points in both \(Q_i\) and \(Q_i^{\mathrm s,+}\setminus Q_i\).
Put \(E_{i,n}:=Z_{i,n}^{\mathrm{iso},+,\mathrm s}
+\widetilde Z_{i,n}^{\mathrm s}\). Each \(E_{i,n}\) is determined by
the short trajectories starting in \(Q_i^{\mathrm s,+}\).

The first input controls the changes caused by deleting far trajectories.
It applies to the global and blockwise measures on their original common
probability space.

\begin{lemma}[Removal of trajectories with large range]
\label{lem:remove-far}
For every \(\delta>0\),
\[
\lim_{n\to\infty}\frac1{\rho_{k,n}}
\log\mathbb P\big(
d_{\mathrm{TV}}(\xi_{k,n},\xi_{k,n}^{\mathrm s})
+d_{\mathrm{TV}}(\eta_{k,n},\eta_{k,n}^{\mathrm s})>\delta
\big)=-\infty.
\]
\end{lemma}

The second input bounds the remaining discrepancy by local counts whose
exponential moments converge to one. Their common zero event will be
used again in Section~\ref{sec:interact}.

\begin{lemma}[Localization errors of the short process]
\label{lem:short-localization}
For all sufficiently large \(n\),
\(\rho_{k,n}d_{\mathrm{TV}}(\xi_{k,n}^{\mathrm s},
\eta_{k,n}^{\mathrm s})\le\sum_iE_{i,n}\).
For every \(\theta>0\),  we have
\(
\sup_{i\le\rho_{k,n}}\mathbb E e^{\theta E_{i,n}}
\longrightarrow1.
\)
\end{lemma}

\begin{proof}[Proof of Proposition~\ref{pr:56}]
Because \(h_n^{\mathrm s}=o(a_n)\), the intersection graph of the
enlarged cubes has degree bounded independently of \(n\). Color this
graph with a fixed number \(M\) of colors. Within each color class, the
variables \(E_{i,n}\) depend on disjoint restrictions of the Poisson
process and are independent. H\"older's inequality and independence
within colors give
\(
\mathbb E e^{\theta\sum_iE_{i,n}}
\le\prod_i(\mathbb E e^{M\theta E_{i,n}})^{1/M}
\le\big(\sup_i\mathbb E e^{M\theta E_{i,n}}\big)^{\rho_{k,n}/M}.
\)
Lemma~\ref{lem:short-localization} and exponential Markov inequality
therefore imply, for every \(\delta,\theta>0\),
\[
\limsup_{n\to\infty}\frac1{\rho_{k,n}}
\log\mathbb P\big(
d_{\mathrm{TV}}(\xi_{k,n}^{\mathrm s},\eta_{k,n}^{\mathrm s})
>\delta\big)\le-\theta\delta.
\]
Letting \(\theta\to\infty\) proves exponential equivalence of the two
short-process measures. By the triangle inequality,
\(d_{\mathrm{TV}}(\xi_{k,n},\eta_{k,n})\) is at most the sum of this
distance and the two distances in Lemma~\ref{lem:remove-far}.
If that sum exceeds \(\delta\), either the short-process distance or
the sum in Lemma~\ref{lem:remove-far} exceeds \(\delta/2\).
Both probabilities are superexponentially small, which proves the
proposition.
\end{proof}

\subsection{Removal of trajectories with large range}

Deleting far trajectories removes clusters that contain them and can
create short clusters that they previously prevented from being isolated.
The first effect is bounded by the number of far trajectories, because
isolated connected clusters are disjoint. The second is localized in the
affected blocks, which are independent of the short process.

\begin{proof}[Proof of Lemma~\ref{lem:remove-far}]
Let \(U_{i,n}\) count connected short clusters with all starting
points in \(Q_i\) that are isolated relative to
\(\PP_{Q_i}^{\mathrm s}\). Define \(V_{i,n}\) in the same way with
\(Q_i^{\mathrm s,+}\) in place of \(Q_i\).
The enlarged cube is contained in the torus projection of a Euclidean
cube of side \(a_n+2kh_n^{\mathrm s}\). Since
\(h_n^{\mathrm s}/a_n\to0\) and \(a_n^d=\rho_{k,n}^{-1}\),
both regions have volume at most \(C/\rho_{k,n}\), uniformly in
\(i\) for sufficiently large \(n\).
Their short-process intensities are bounded by \(\Theta_n\).
Applying Lemma~\ref{lem:isolated-factorial} with the relevant region
as anchor set therefore gives
\(\sup_{n,i}\mathbb E e^{tU_{i,n}}<\infty\) and
\(\sup_{n,i}\mathbb E e^{tV_{i,n}}<\infty\) for every \(t>0\).
Here and below, these suprema are over sufficiently large \(n\).
Cauchy--Schwarz gives
\(\mathbb E e^{\theta(U_{i,n}+V_{i,n})}
\le\sqrt{\mathbb E e^{2\theta U_{i,n}}}
\sqrt{\mathbb E e^{2\theta V_{i,n}}}\), so
\(K_\theta:=\sup_{n,i}\mathbb E e^{\theta(U_{i,n}+V_{i,n})}<\infty\).

We next compare the actual clusters contributing to the measures.
Deleting trajectories preserves connectedness of every remaining tuple
and can only improve its isolation. Thus, a cluster counted in the full
process disappears under deletion only if it contains a far trajectory.
Distinct isolated connected clusters are disjoint: if they shared a
trajectory, a member belonging to only one cluster would approach a
member of the other, contradicting isolation.
Assigning one far trajectory to each disappearing global cluster
therefore bounds their number by \(N_n^{\mathrm f}\).

Conversely, suppose a cluster is globally isolated in the short process
and was not isolated in the full process. An approaching trajectory
that was deleted must be far. Anchor the cluster at the block containing
the starting point of its first trajectory in the fixed Borel order.
The range bound places every other starting point within
\(h_n^{\mathrm s}\) of that anchor, so all of them lie in
\(Q_i^{\mathrm s,+}\). Global isolation in the short process
implies isolation relative to \(\PP_{Q_i^{\mathrm s,+}}^{\mathrm s}\),
and the cluster is counted by \(V_{i,n}\).
Lemma~\ref{lem:far-influence} places its anchor block in
\(\mathcal D_n\). Choosing a unique anchor assigns each such
cluster to one of these counts.

For the blockwise measures, every disappearing cluster again contains
a far trajectory. Clusters isolated within one block are disjoint,
and clusters from different blocks have disjoint starting points, so
their total number is at most \(N_n^{\mathrm f}\).
A cluster gained in block \(Q_i\) is counted by \(U_{i,n}\) and
must previously have been blocked by a far trajectory starting in
\(Q_i\). For that trajectory \((z,\gamma)\), we have
\(d_{\mathbb T}(z,Q_i)=0\), hence \(i\in\mathcal D_n\).
Each cluster retained in both measures contributes exactly the same
shape atom, since the centering rule depends only on that cluster.
Cancelling these common contributions and using the triangle inequality
for the remaining atoms gives
\[
\rho_{k,n}d_{\mathrm{TV}}(\xi_{k,n},\xi_{k,n}^{\mathrm s})
\le N_n^{\mathrm f}+\sum_{i\in\mathcal D_n}V_{i,n},\quad\text{ and }\quad 
\rho_{k,n}d_{\mathrm{TV}}(\eta_{k,n},\eta_{k,n}^{\mathrm s})
\le N_n^{\mathrm f}+\sum_{i\in\mathcal D_n}U_{i,n}.\]
For sufficiently large \(n\), \(kh_n^{\mathrm s}<a_n/2\).
Two enlarged cubes can then intersect only when their block indices
differ by at most one in every coordinate, with indices read
periodically on the torus. Their intersection graph has degree at
most \(3^d-1\), and greedy coloring uses at most \(M:=3^d\) colors.
Fix such a coloring, with classes \(\mathcal I_1,\ldots,\mathcal I_M\),
allowing empty classes. Within each class the enlarged cubes are
disjoint, so the corresponding variables \(U_{i,n}+V_{i,n}\) are
independent. Set
\(T_n:=\sum_{i\in\mathcal D_n}(U_{i,n}+V_{i,n})\).
For \(\theta,\varepsilon>0\) and a deterministic set \(I\) of
block indices, H\"older's inequality over the color classes and
independence within each class yield
\[
\begin{aligned}
\mathbb E e^{\theta\sum_{i\in I}(U_{i,n}+V_{i,n})}
\le\prod_{\ell=1}^M
\Big(\mathbb E e^{M\theta
\sum_{i\in I\cap\mathcal I_\ell}(U_{i,n}+V_{i,n})}\Big)^{1/M}
&=\prod_{i\in I}
\big(\mathbb E e^{M\theta(U_{i,n}+V_{i,n})}\big)^{1/M}
\le K_{M\theta}^{|I|/M},\\
\mathbb E\big[e^{\theta T_n}
\mathbf{1}\{|\mathcal D_n|\le\varepsilon\rho_{k,n}\}\big]
&\le\sum_{I:\,|I|\le\varepsilon\rho_{k,n}}
\mathbb P(\mathcal D_n=I)K_{M\theta}^{|I|/M}
\le K_{M\theta}^{\varepsilon\rho_{k,n}/M}.
\end{aligned}
\]
The random set \(\mathcal D_n\) depends only on the far process and
is independent of the entire collection of short-process counts.
Consequently, conditioning on \(\mathcal D_n=I\) leaves their joint
law unchanged, which justifies the last line. Also, \(K_{M\theta}\ge1\),
and the probabilities in that sum add to at most one.
Fix \(\delta>0\) and a desired exponential rate \(A>0\).
Choose \(\theta>0\) so that \(\theta\delta/2>2A\), and then
choose \(\varepsilon>0\) so that
\(\varepsilon M^{-1}\log K_{M\theta}<A\).
Such a choice is possible because \(K_{M\theta}<\infty\) for
this fixed \(\theta\); if its logarithm is zero, any
\(\varepsilon>0\) works. The deterministic discrepancy bound and
exponential Markov inequality on
\(\{|\mathcal D_n|\le\varepsilon\rho_{k,n}\}\) now give
\[
\begin{aligned}
\mathbb P\big(
d_{\mathrm{TV}}(\xi_{k,n},\xi_{k,n}^{\mathrm s})
+d_{\mathrm{TV}}(\eta_{k,n},\eta_{k,n}^{\mathrm s})>\delta\big)
&\le\mathbb P(N_n^{\mathrm f}>\delta\rho_{k,n}/4)
+\mathbb P(T_n>\delta\rho_{k,n}/2)\\
&\quad\le\mathbb P(N_n^{\mathrm f}>\delta\rho_{k,n}/4)
+\mathbb P(|\mathcal D_n|>\varepsilon\rho_{k,n})
+e^{-\rho_{k,n}(\theta\delta/2-\varepsilon M^{-1}\log K_{M\theta})}.
\end{aligned}
\]
Lemma~\ref{lem:far-influence} makes the first two terms in the last
line superexponentially small, while the last term is at most
\(e^{-A\rho_{k,n}}\). Bounding their sum by three times its largest
term gives an upper limit of at most \(-A\) for the normalized
logarithm of the probability in the statement. Since \(A>0\) was
arbitrary and \(\rho_{k,n}^{-1}\log3\to0\), this upper limit is
\(-\infty\), proving the lemma.
\end{proof}

\subsection{Localization of the short process}

The deterministic range bound confines every interaction involving a
short cluster to an enlarged starting block. We distinguish a cluster
that loses isolation because of a neighboring block from a cluster whose
starting points lie on both sides of a block boundary. Their counts have
small exponential moments by the spatial integration and factorial-moment
estimates already established.

\begin{figure}[!htbp]
    \centering
    \begin{minipage}[t]{0.4\textwidth}
        \centering
        \includegraphics[width=0.95\textwidth]{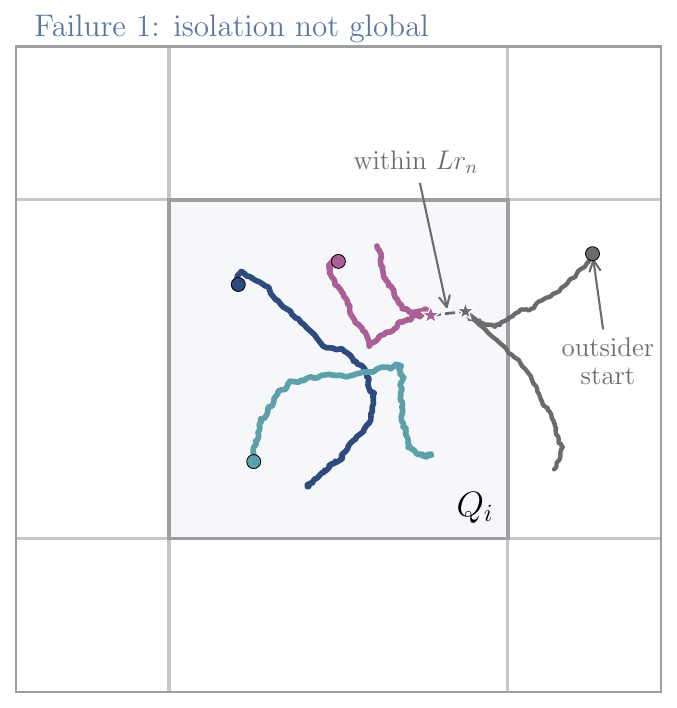}
        \vspace{0.05cm}
        {\small (a) Locally isolated, but not globally isolated.}
    \end{minipage}
    \qquad\qquad
    \begin{minipage}[t]{0.4\textwidth}
        \centering
        \includegraphics[width=0.95\textwidth]{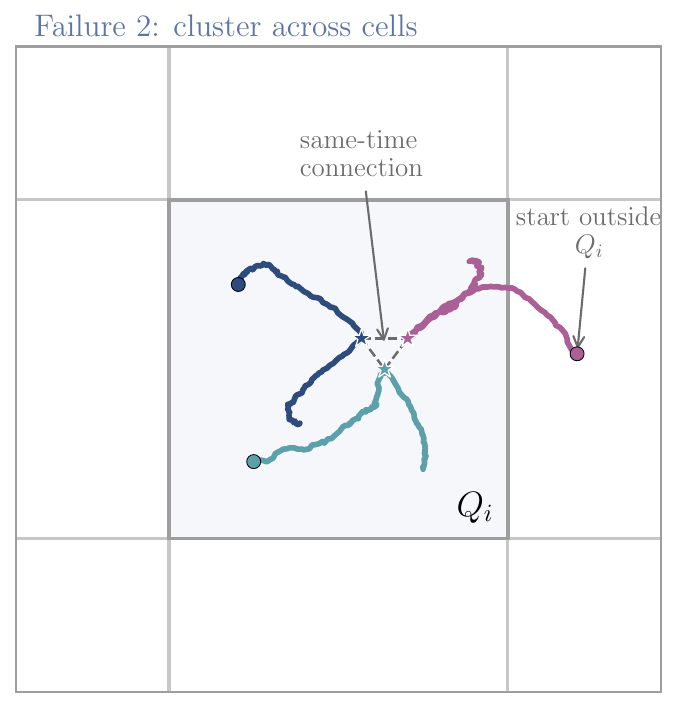}
        \vspace{0.05cm}
        {\small (b) Cluster with starting points on both sides of the boundary of \(Q_i\).}
    \end{minipage}
    \caption{The two localization errors in \(Q_i\).}
    \label{fig:localization-errors}
\end{figure}

\begin{proof}[Proof of Lemma~\ref{lem:short-localization}]
Suppose that two short trajectories starting at \(x\) and \(y\)
approach within \(Lr_n\) at time \(t\). The triangle inequality
gives \(d_{\mathbb T}(x,y)\le r_nv_n+Lr_n+r_nv_n=h_n^{\mathrm s}\).
In a connected cluster, every member approaches the anchor trajectory
at the common connection time. Thus, if the anchor starts in \(Q_i\),
every member starts within \(h_n^{\mathrm s}\) of \(Q_i\).
If another short trajectory approaches one of these members, its
starting point lies within \(2h_n^{\mathrm s}\) of \(Q_i\).
Since \(k\ge2\) and the enlargement has width \(kh_n^{\mathrm s}\),
both the entire cluster and every possible short blocker start in
\(Q_i^{\mathrm s,+}\).

A cluster counted locally in \(Q_i\) and failing global isolation
has an approaching short trajectory outside \(Q_i\). The preceding
geometry places that blocker in \(Q_i^{\mathrm s,+}\setminus Q_i\),
so the cluster is counted by \(Z_{i,n}^{\mathrm{iso},+,\mathrm s}\).
A globally isolated cluster whose starting points all lie in one
block is counted locally in that block as well. Every other globally
isolated cluster has starting points in several blocks. Choosing one
of those blocks as \(Q_i\) places all its starting points in
\(Q_i^{\mathrm s,+}\), with at least one in its outer layer.
Global isolation implies isolation in this enlarged block, so the
cluster is counted by \(\widetilde Z_{i,n}^{\mathrm s}\).
Hence, every cluster present in only one of the two measures is covered
by at least one of the error counts. The clusters present in both have
identical centered shape atoms and cancel. Each remaining atom has
total variation norm \(\rho_{k,n}^{-1}\), giving
\(\rho_{k,n}d_{\mathrm{TV}}(\xi_{k,n}^{\mathrm s},
\eta_{k,n}^{\mathrm s})\le\sum_iE_{i,n}\).

For the isolation error, let \(\Theta_n^{\mathrm s}\)
be the intensity of the short process on path space. Define
\(\chi_{i,n}(\gamma'_1,\ldots,\gamma'_{k+1})\) to be the
indicator that the first \(k\) trajectories are connected and start
in \(Q_i\), while the last starts in
\(Q_i^{\mathrm s,+}\setminus Q_i\) and approaches at least one
of them within \(Lr_n\). This indicator omits isolation and is
symmetric in its first \(k\) entries. Each cluster counted by
\(Z_{i,n}^{\mathrm{iso},+,\mathrm s}\) has at least one such
blocker. Counting all blockers and all \(k!\) orderings of its
members therefore gives an upper bound after division by \(k!\).
Recall the connected proximity-graph integral \(I_{j,n}(A)\) from
the proof of Lemma~\ref{lem-equival}. The multivariate Mecke formula
\cite{last-penrose2018} gives
\[
\begin{aligned}
\mathbb E Z_{i,n}^{\mathrm{iso},+,\mathrm s}
&\le\frac1{k!}\mathbb E
\sum_{(\gamma'_1,\ldots,\gamma'_{k+1})\in(\PP_n^{\mathrm s})_{\ne}^{k+1}}
\chi_{i,n}(\gamma'_1,\ldots,\gamma'_{k+1})
=\frac1{k!}\int\chi_{i,n}\,d(\Theta_n^{\mathrm s})^{\otimes(k+1)}
\le\frac1{k!}I_{k+1,n}(Q_i)\\
&\le Cn^{k+1}|Q_i|r_n^{dk}
=C\frac{n^{k+1}r_n^{dk}}{n^kr_n^{d(k-1)}}
=Cnr_n^d=Cs_n\longrightarrow0.
\end{aligned}
\]
For the bound by \(I_{k+1,n}(Q_i)\), the first \(k\) trajectories
form a complete proximity graph and the blocker has an edge to it,
so their union is connected. We keep the restriction that the first
starting point lies in \(Q_i\), discard the remaining spatial
restrictions, and replace \(\Theta_n^{\mathrm s}\) by
\(\Theta_n\). The final estimate is
\eqref{eq:connected-pattern-integral}; its constant is uniform in
\(i\), and \(|Q_i|=\rho_{k,n}^{-1}\).

Let \(U_{i,n}\) again denote the number of connected short clusters
isolated within \(Q_i\). Then
\(Z_{i,n}^{\mathrm{iso},+,\mathrm s}\le U_{i,n}\), and
Lemma~\ref{lem:isolated-factorial}, with anchor set \(Q_i\), gives
\(\sup_{n,i}\mathbb E e^{tU_{i,n}}<\infty\) for every \(t>0\).
The difference \(e^{\theta Z_{i,n}^{\mathrm{iso},+,\mathrm s}}-1\)
vanishes unless the error count is positive. Cauchy--Schwarz on this
event, followed by the inequality
\(\mathbb P(Z_{i,n}^{\mathrm{iso},+,\mathrm s}>0)
\le\mathbb E Z_{i,n}^{\mathrm{iso},+,\mathrm s}\), yields
\[
\begin{aligned}
\mathbb E\big(e^{\theta Z_{i,n}^{\mathrm{iso},+,\mathrm s}}-1\big)
&=\mathbb E\big[
(e^{\theta Z_{i,n}^{\mathrm{iso},+,\mathrm s}}-1)
\mathbf{1}\{Z_{i,n}^{\mathrm{iso},+,\mathrm s}>0\}\big]
\le\sqrt{\mathbb E e^{2\theta U_{i,n}}}
\sqrt{\mathbb P(Z_{i,n}^{\mathrm{iso},+,\mathrm s}>0)}
\le C_\theta\sqrt{s_n}\longrightarrow0.
\end{aligned}
\]
This convergence is uniform in \(i\) for each fixed \(\theta>0\).
In particular, the exponential moments of the first error count tend
to one, which is stronger than convergence of its mean alone.

For the boundary error, put
\(A_{i,n}:=Q_i^{\mathrm s,+}\setminus Q_i\), and let
\(\Theta_{i,n}^{\mathrm s,+}\) be the restriction of
\(\Theta_n^{\mathrm s}\) to trajectories starting in
\(Q_i^{\mathrm s,+}\).
Let \(h_{i,n}\) be the symmetric indicator that a tuple is connected
and has starting points in both \(Q_i\) and \(A_{i,n}\).
With this intensity and indicator, the variable \(N_h\) in
Lemma~\ref{lem:isolated-factorial} is exactly
\(\widetilde Z_{i,n}^{\mathrm s}\). Define its one-cluster
integral by
\(\Lambda_{i,n}:=\frac1{k!}\int h_{i,n}\,
d(\Theta_{i,n}^{\mathrm s,+})^{\otimes k}\).
To bound it, choose one of the at most \(k\) starting points in
\(A_{i,n}\) as anchor and drop the other block restrictions.
Symmetry then bounds \(\Lambda_{i,n}\) by
\(I_{k,n}(A_{i,n})/(k-1)!\).
The spatial integration in \eqref{eq:connected-pattern-integral}
places the other \(k-1\) starting points in balls about that anchor
and gives \(I_{k,n}(A_{i,n})\le C\rho_{k,n}|A_{i,n}|\).

The enlarged block is contained  a cube of side \(a_n+2kh_n^{\mathrm s}\).
Its outer layer therefore has volume at most
\((a_n+2kh_n^{\mathrm s})^d-a_n^d\le Ca_n^{d-1}h_n^{\mathrm s}\),
because \(h_n^{\mathrm s}/a_n\to0\).
Since \(\rho_{k,n}a_n^d=1\), Lemma~\ref{lem:isolated-factorial}
and its factorial-moment expansion give, for every integer \(m\ge1\)
and \(\theta>0\), that\(\Lambda_{i,n}
\le C\rho_{k,n}|A_{i,n}|
\le C\rho_{k,n}a_n^{d-1}h_n^{\mathrm s}
=C h_n^{\mathrm s}/a_n\longrightarrow0,\) and  \(\mathbb E[(\widetilde Z_{i,n}^{\mathrm s})_m]
\le\Lambda_{i,n}^m,\) and therefore
 \[\mathbb E e^{\theta\widetilde Z_{i,n}^{\mathrm s}}
=\sum_{m=0}^\infty\frac{(e^\theta-1)^m}{m!}
\mathbb E[(\widetilde Z_{i,n}^{\mathrm s})_m]
\le e^{\Lambda_{i,n}(e^\theta-1)}
\le e^{C(h_n^{\mathrm s}/a_n)(e^\theta-1)}\longrightarrow1.\]
The series has nonnegative terms, so its interchange with expectation
is justified by monotone convergence. All bounds are uniform in \(i\).
Finally, the two error counts need not be independent; Cauchy--Schwarz
gives
\(\mathbb E e^{\theta E_{i,n}}
\le\sqrt{\mathbb E e^{2\theta Z_{i,n}^{\mathrm{iso},+,\mathrm s}}}
\sqrt{\mathbb E e^{2\theta\widetilde Z_{i,n}^{\mathrm s}}}\).
Both factors tend to one uniformly in \(i\) by the preceding
estimates with \(2\theta\) in place of \(\theta\). This proves
the exponential-moment assertion and completes the lemma.
\end{proof}

\section{Free energy and hard-core interactions}
\label{sec:interact}

We first deduce Corollary~\ref{cor:free} from the LDP and the common
factorial-moment estimate. We then prove Proposition~\ref{pr:hard-core-lower}
by conditioning on the absence of far trajectories and tilting the
independent short blocks. The localization errors from
Section~\ref{sec:2proof} can all be excluded at subexponential cost under
this tilt. We use the product-domination theorem of Liggett, Schonmann,
and Stacey to pass from small individual error probabilities to the
simultaneous absence of all errors. Their absence gives exact equality
of the cluster measures, so the argument applies directly to the
forbidden set \(B\).

\begin{proof}[Proof of Corollary~\ref{cor:free}]
Set \(K_n:=\rho_{k,n}\xi_{k,n}(E_k)\),
\(\Psi(\mu):=-\beta\int W\,d\mu\), and \(V:=\|W\|_\infty\).
Here \(K_n\) is the number of globally isolated connected clusters.
Apply Lemma~\ref{lem:isolated-factorial} to the full Poisson process,
with the connection indicator as \(h\) and \(A=\Lambda\).
Since \(|\Lambda|=1\), it gives
\(\mathbb E e^{\theta K_n}\le
\exp\{C\rho_{k,n}(e^\theta-1)\}\) for every \(\theta>0\),
where \(C\) is independent of \(n\).
If \(V=0\), the Gibbs weight equals one, while the entropy is
nonnegative and vanishes at \(\tau_k\); both sides of the claimed
identity are zero. We therefore assume \(V>0\).

By the definition of the \(\tau\)-topology, integration against
the bounded measurable function \(W\) is continuous. Thus, \(\Psi\)
is continuous, although its values can be unbounded as the total mass
of its argument varies. To apply the unbounded form of Varadhan's
integral lemma, we verify its upper-tail condition using
\(\rho_{k,n}\Psi(\xi_{k,n})\le\beta VK_n\).
Fix \(\theta>\beta V\). On the event
\(\{\Psi(\xi_{k,n})\ge M\}\), with \(M>0\), this inequality
implies \(K_n\ge\rho_{k,n}M/(\beta V)\). Consequently,
\[
\begin{aligned}
e^{\rho_{k,n}\Psi(\xi_{k,n})}
\mathbf{1}\{\Psi(\xi_{k,n})\ge M\}
&\le e^{\beta VK_n}
\mathbf{1}\{K_n\ge\rho_{k,n}M/(\beta V)\}\\
&=e^{\theta K_n}e^{-(\theta-\beta V)K_n}
\mathbf{1}\{K_n\ge\rho_{k,n}M/(\beta V)\}\\
&\le e^{\theta K_n}
e^{-\rho_{k,n}(\theta-\beta V)M/(\beta V)}.
\end{aligned}
\]
Taking expectations and using the cluster-count bound yields
\[
\begin{aligned}
\limsup_{n\to\infty}\frac1{\rho_{k,n}}
\log\mathbb E\big[e^{\rho_{k,n}\Psi(\xi_{k,n})}
\mathbf{1}\{\Psi(\xi_{k,n})\ge M\}\big]
&\le-\frac{\theta-\beta V}{\beta V}M+C(e^\theta-1),\\
\lim_{M\to\infty}\limsup_{n\to\infty}\frac1{\rho_{k,n}}
\log\mathbb E\big[e^{\rho_{k,n}\Psi(\xi_{k,n})}
\mathbf{1}\{\Psi(\xi_{k,n})\ge M\}\big]
&=-\infty.
\end{aligned}
\]
In the second line, \(\theta>\beta V\) remains fixed while
\(M\to\infty\), so the coefficient of \(M\) is strictly negative.
This is the upper-tail condition in
\cite[Theorem~4.3.1]{dz98}. Together with the continuity of \(\Psi\)
and the good LDP from Theorem~\ref{thm:fine}, it verifies the
hypotheses of that integral lemma and gives
\[
\lim_{n\to\infty}\frac1{\rho_{k,n}}
\log\mathbb E e^{\rho_{k,n}\Psi(\xi_{k,n})}
=\sup_{\mu\in M_f(E_k)}\{\Psi(\mu)-h(\mu\mid\tau_k)\}.
\]
Substitution of \(\Psi\) proves the claimed variational formula.
\end{proof}

\subsection{Hard-core lower bound}
\label{subsec:hard-core}
Here, we proceed similarly as in \cite{jkm15}.
The lower bound requires exact absence of forbidden clusters, which a
comparison at a fixed positive total variation distance does not ensure.
We obtain this event by first requiring the far process to be empty and
then excluding the short-process localization errors under a product
tilt. Two inputs suffice: convergence of the partition function in one
short block and a subexponential bound for simultaneous absence of its
localization errors. The second input follows by showing that the
indicators of error-free blocks dominate independent Bernoulli variables
whose success probabilities tend to one.

Write \(\mathbb P_n^{\mathrm s}\) for the law of the short marked
process and \(\mathbb E_n^{\mathrm s}\) for its expectation.
Let \(\eta_{k,n}^{\mathrm s,(i)}\) be the unnormalized block
contribution obtained from \(\eta_{k,n}^{(i)}\) by replacing the
full process with the short process. Then
\(\eta_{k,n}^{\mathrm s}=\rho_{k,n}^{-1}
\sum_i\eta_{k,n}^{\mathrm s,(i)}\).
For \(\nu\in\mathcal M_p(E_k)\), put
\(G_B(\nu):=\exp\{-\beta\int w\,d\nu\}\mathbf{1}\{\nu(B)=0\}\),
and define
\[
a_n^{\mathrm s}:=\mathbb E_n^{\mathrm s}
G_B(\eta_{k,n}^{\mathrm s,(1)}),
\qquad
a_B:=\exp\Big\{-\tau_k(B)+
\int_{B^c}(e^{-\beta w}-1)\,d\tau_k\Big\}.
\]
Translation invariance makes the block normalizers identical for all
sufficiently large \(n\). Lemma~\ref{lem:isolated-factorial} makes
\(a_n^{\mathrm s}\) finite, and the event that a block is empty makes
it strictly positive. The first input identifies its limit.

\begin{lemma}[Partition function in one short block]
\label{lem:local-hard-core-laplace}
Under the assumptions of Theorem~\ref{thm:fine},
\(a_n^{\mathrm s}\to a_B>0\).
\end{lemma}

Recall the error counts \(E_{i,n}\) from Section~\ref{sec:2proof}, and
set \(\mathcal A_{i,n}:=\{E_{i,n}>0\}\) and
\(H_n:=\bigcap_{i=1}^{\rho_{k,n}}\mathcal A_{i,n}^c\).
Define the tilted short-process law by
\[
\frac{d\widehat{\mathbb P}_n^{\mathrm s}}
{d\mathbb P_n^{\mathrm s}}
:=(a_n^{\mathrm s})^{-\rho_{k,n}}
\prod_{i=1}^{\rho_{k,n}}G_B(\eta_{k,n}^{\mathrm s,(i)}).
\]
Each factor depends on one original block, and the blocks are independent
under \(\mathbb P_n^{\mathrm s}\). Their expectations are
\(a_n^{\mathrm s}\), so this density integrates to one and preserves
independence of the blocks. The second input controls the event on which
the short-process measures agree exactly.

\begin{lemma}[Simultaneous absence of boundary errors]
\label{lem:hard-core-boundary-void}
Under the assumptions of Theorem~\ref{thm:fine},
\[
\lim_{n\to\infty}\frac1{\rho_{k,n}}
\log\widehat{\mathbb P}_n^{\mathrm s}(H_n)=0.
\]
For all sufficiently large \(n\),
\(\xi_{k,n}^{\mathrm s}=\eta_{k,n}^{\mathrm s}\) on \(H_n\).
\end{lemma}

\begin{proof}[Proof of Proposition~\ref{pr:hard-core-lower}]
Restrict the expectation defining \(Z_n^{\mathrm{hc}}\) to
\(F_n\cap H_n\). On \(F_n\), the full process equals the short
process, and \(F_n\) is independent of the short process by Poisson
thinning. On \(H_n\), Lemma~\ref{lem:hard-core-boundary-void}
identifies the global and blockwise short measures. Their common Gibbs
weight is the product of the block weights \(G_B\). Therefore,
\(
Z_n^{\mathrm{hc}}
\ge\mathbb P(F_n)\,
\mathbb E_n^{\mathrm s}\Big[
\prod_iG_B(\eta_{k,n}^{\mathrm s,(i)})\mathbf{1}\{H_n\}\Big]
=\mathbb P(F_n)(a_n^{\mathrm s})^{\rho_{k,n}}
\widehat{\mathbb P}_n^{\mathrm s}(H_n).
\)
The normalized logarithm of \(\mathbb P(F_n)\) tends to zero by
\eqref{eq:cutoff-void-probability}. Lemma~\ref{lem:local-hard-core-laplace}
gives \(\log a_n^{\mathrm s}\to\log a_B\), and
Lemma~\ref{lem:hard-core-boundary-void} gives zero for the remaining
normalized logarithm. Taking the lower limit proves
\(\liminf_n\rho_{k,n}^{-1}\log Z_n^{\mathrm{hc}}
\ge\log a_B\).
\end{proof}

\begin{proof}[Proof of Lemma~\ref{lem:local-hard-core-laplace}]
Couple the full and short processes in \(Q_1\) by deleting the far
trajectories from the full process. Their number is Poisson with mean
\(n|Q_1|q_n=m_nq_n\). If this number is zero, the two restrictions
coincide, and their block cluster measures agree, including their
isolation indicators and centered shapes. Writing
\(\epsilon_n:=d_{\mathrm{TV}}(
\mathcal L(\eta_{k,n}^{\mathrm s,(1)}),
\mathcal L(\zeta_k^{(1)}))\), the coupling inequality, the triangle
inequality, and Lemma~\ref{lem-equival} give
\(d_{\mathrm{TV}}\big(
\mathcal L(\eta_{k,n}^{\mathrm s,(1)}),
\mathcal L(\eta_{k,n}^{(1)})\big)
\le1-e^{-m_nq_n}\le m_nq_n,\) and thus
\(
\epsilon_n
\le m_nq_n+
d_{\mathrm{TV}}\big(\mathcal L(\eta_{k,n}^{(1)}),
\mathcal L(\zeta_k^{(1)})\big)\longrightarrow0.
\)
The first term tends to zero by \eqref{eq:range-cutoff}; the second
uses the full-block Poisson approximation. We now control the
possibly unbounded weight \(G_B\) under these converging laws.

Put \(V:=\|w\|_\infty\). For every finite point measure \(\nu\),
\(0\le G_B(\nu)^2\le e^{2\beta V\nu(E_k)}\).
Lemma~\ref{lem:isolated-factorial}, applied to the short process in
\(Q_1\), with \(\rho_{k,n}|Q_1|=1\), bounds the expectations of
this exponential uniformly in \(n\). One may use the positive
exponent \(2\beta V+1\), so the same argument includes \(V=0\).
For the limiting Poisson random measure,
\(\mathbb E e^{2\beta V\zeta_k^{(1)}(E_k)}
=\exp\{\tau_k(E_k)(e^{2\beta V}-1)\}<\infty\).
Choose \(C<\infty\) that bounds both
\(\mathbb E_n^{\mathrm s}G_B(\eta_{k,n}^{\mathrm s,(1)})^2\)
for all sufficiently large \(n\) and
\(\mathbb E G_B(\zeta_k^{(1)})^2\).
For \(M>0\), define \(G_{B,M}:=\min\{G_B,M\}\).
This is an evaluation-measurable function bounded by \(M\), since
\(w\) is bounded measurable and \(B\) is Borel.
The elementary inequality
\(0\le x-\min\{x,M\}\le x^2/M\), for \(x\ge0\), controls
the two truncation errors. The total variation bound controls the
expectations of the bounded truncation, giving
\begin{enumerate}
    \item
\(\mathbb E_n^{\mathrm s}
(G_B-G_{B,M})(\eta_{k,n}^{\mathrm s,(1)}) +\mathbb E(G_B-G_{B,M})(\zeta_k^{(1)})
\le C/M,\)
\item \(\big|a_n^{\mathrm s}-\mathbb E G_B(\zeta_k^{(1)})\big|
\le\big|\mathbb E_n^{\mathrm s}G_{B,M}(\eta_{k,n}^{\mathrm s,(1)})
-\mathbb E G_{B,M}(\zeta_k^{(1)})\big|+2C/M
\le2M\epsilon_n+2C/M.\)
\end{enumerate}
For fixed \(M\), the upper limit of the last expression as
\(n\to\infty\) is \(2C/M\). Letting \(M\to\infty\) proves
\(a_n^{\mathrm s}\to\mathbb E G_B(\zeta_k^{(1)})\).

To identify the limiting expectation, define the bounded nonnegative
function \(g:=e^{-\beta w}\mathbf{1}_{B^c}\) on \(E_k\).
For a finite point measure \(\nu\), the product of \(g\) over its
atoms, counted with multiplicity, equals \(G_B(\nu)\): a point in
\(B\) makes the product zero, and otherwise the product is
\(e^{-\beta\int w\,d\nu}\). The empty product is one.
The Poisson product formula \cite{last-penrose2018} therefore gives
\[
\begin{aligned}
\mathbb E G_B(\zeta_k^{(1)})
&=e^{-\tau_k(E_k)}
\sum_{j=0}^\infty\frac1{j!}
\Big(\int_{E_k}g\,d\tau_k\Big)^j
=e^{\int_{E_k}(g-1)\,d\tau_k}
&=e^{-\tau_k(B)
+\int_{B^c}(e^{-\beta w}-1)\,d\tau_k}=a_B.
\end{aligned}
\]
The exponent is finite because \(\tau_k\) is finite and \(w\)
is bounded, so \(a_B>0\). If \(\tau_k(E_k)=0\), the Poisson
measure is empty almost surely and the series equals one, in agreement
with the same formula. This proves the stated convergence and positivity.
\end{proof}

\begin{proof}[Proof of Lemma~\ref{lem:hard-core-boundary-void}]
We first locate the variables on which each error depends, then bound
its probability under the tilted law. This will allow us to apply
product domination to the indicators of error-free blocks. Finally,
the all-one event gives the required lower bound, while the earlier
localization estimate gives equality of the measures.
For each \(i\), let
\(\mathcal J_i:=\{j:Q_j\cap Q_i^{\mathrm s,+}\ne\varnothing\}\),
and let \(\mathcal F_{j,n}\) be the \(\sigma\)-field generated by
the short trajectories starting in \(Q_j\).
The event \(\mathcal A_{i,n}\) is measurable with respect to
\(\sigma(\mathcal F_{j,n}:j\in\mathcal J_i)\).
Since \(kh_n^{\mathrm s}<a_n\) for all sufficiently large \(n\),
the enlarged cube reaches at most one neighboring block in either
direction of each coordinate. Hence, \(|\mathcal J_i|\le J:=3^d\).
The same observation applies at the torus boundary, with block indices
read periodically.

The tilted density is a product of one factor for each original block.
Thus, the \(\sigma\)-fields \(\mathcal F_{j,n}\) remain independent
under \(\widehat{\mathbb P}_n^{\mathrm s}\), although the law within
each block changes. Join distinct indices \(i\) and \(i'\) when
\(\mathcal J_i\cap\mathcal J_{i'}\ne\varnothing\), and denote
this graph by \(\mathcal G_n\). Its degree is at most
\(D:=5^d-1\), because intersecting neighborhoods require the block
indices to differ by at most two in each coordinate. Moreover,
\(\mathcal A_{i,n}\) is independent of the entire collection of
error events indexed by nonneighbors: all their determining blocks
lie outside \(\mathcal J_i\). 

Lemma~\ref{lem:short-localization} gives
\(\sup_i\mathbb E_n^{\mathrm s}e^{E_{i,n}}\to1\).
Since \(E_{i,n}\) is a nonnegative integer,
\(\mathbf{1}\{\mathcal A_{i,n}\}
\le(e^{E_{i,n}}-1)/(e-1)\). Consequently,
\(\sup_i\mathbb P_n^{\mathrm s}(\mathcal A_{i,n})
\le(\sup_i\mathbb E_n^{\mathrm s}e^{E_{i,n}}-1)/(e-1)\to0\).
We must transfer this estimate to the tilted law, since an estimate
under the original law alone would not suffice.

Put \(V:=\|w\|_\infty\) and
\(M:=\sup_{n,j}\mathbb E_n^{\mathrm s}
e^{2\beta V\eta_{k,n}^{\mathrm s,(j)}(E_k)}\), with the supremum
over sufficiently large \(n\). The short-process intensity is bounded
by \(\Theta_n\), so Lemma~\ref{lem:isolated-factorial}, with anchor
set \(Q_j\), gives \(M<\infty\). This also covers \(V=0\), when
\(M=1\). Let \(c_B:=\min\{1,a_B/2\}>0\).
Lemma~\ref{lem:local-hard-core-laplace} ensures
\(a_n^{\mathrm s}\ge c_B\) for all sufficiently large \(n\).
In the tilted probability of \(\mathcal A_{i,n}\), each density
factor outside \(\mathcal J_i\) integrates to \(a_n^{\mathrm s}\)
and cancels against the normalizer. Using
\(G_B(\nu)\le e^{\beta V\nu(E_k)}\), Cauchy--Schwarz, and then
independence of the original blocks, we obtain
\[
\begin{aligned}
\widehat{\mathbb P}_n^{\mathrm s}(\mathcal A_{i,n})
&=(a_n^{\mathrm s})^{-|\mathcal J_i|}
\mathbb E_n^{\mathrm s}\Big[
\mathbf{1}\{\mathcal A_{i,n}\}
\prod_{j\in\mathcal J_i}G_B(\eta_{k,n}^{\mathrm s,(j)})\Big]
\le(a_n^{\mathrm s})^{-|\mathcal J_i|}
\mathbb P_n^{\mathrm s}(\mathcal A_{i,n})^{1/2}
\Big(\mathbb E_n^{\mathrm s}
 e^{2\beta V\sum_{j\in\mathcal J_i}
 \eta_{k,n}^{\mathrm s,(j)}(E_k)}\Big)^{1/2}\\
&=(a_n^{\mathrm s})^{-|\mathcal J_i|}
\sqrt{\mathbb P_n^{\mathrm s}(\mathcal A_{i,n})}
\prod_{j\in\mathcal J_i}
\sqrt{\mathbb E_n^{\mathrm s}
 e^{2\beta V\eta_{k,n}^{\mathrm s,(j)}(E_k)}}
\le c_B^{-J}M^{J/2}
\sqrt{\mathbb P_n^{\mathrm s}(\mathcal A_{i,n})}.
\end{aligned}
\]
The multiplier is independent of \(n\) and \(i\). Therefore,
\(p_n:=\sup_i\widehat{\mathbb P}_n^{\mathrm s}(\mathcal A_{i,n})\to0\).
In particular, the tilt preserves the vanishing of each local error
probability uniformly over the growing collection of blocks.

Set \(Y_{i,n}:=\mathbf{1}\{\mathcal A_{i,n}^c\}\).
Let \(N_n[i]\) be the closed neighborhood of \(i\) in
\(\mathcal G_n\), including \(i\) itself, and put
\(\mathcal H_{i,n}:=\sigma(Y_{j,n}:j\notin N_n[i])\).
By the dependency property proved above, \(Y_{i,n}\) is independent
of \(\mathcal H_{i,n}\). Hence, a.s.,
\[
\widehat{\mathbb P}_n^{\mathrm s}
 (Y_{i,n}=1\mid\mathcal H_{i,n})
=\widehat{\mathbb P}_n^{\mathrm s}(Y_{i,n}=1)
\ge1-p_n
\]
For a graph with vertex set \(S\), write
\(\nu_\vartheta^S\) for the product law of independent Bernoulli
variables with success probability \(\vartheta\).
We use \cite[Corollary~1.4]{liggettschonmannstacey1997}: on a graph of degree at most the fixed constant \(D\),
a \(\{0,1\}\)-valued field whose conditional success probabilities,
given the variables outside each closed neighborhood, are at least
\(1-p\) dominates \(\nu_{\vartheta_D(p)}^S\) for all sufficiently
small \(p\), where \(\vartheta_D(p)\to1\) as \(p\downarrow0\).

The graphs \(\mathcal G_n\) have the common degree bound \(D\),
and the preceding conditional estimate verifies the other hypothesis
with \(p=p_n\). Thus, when \(p_n>0\) and \(n\) is sufficiently
large, the law of \((Y_{i,n})_i\) dominates the Bernoulli product law
with success probability \(\vartheta_n:=\vartheta_D(p_n)\).
If \(p_n=0\), every error event has probability zero; since there
are finitely many blocks, \(H_n\) has probability one, and we set
\(\vartheta_n=1\). In both cases \(\vartheta_n\to1\).
The event \(H_n=\{Y_{i,n}=1\text{ for every }i\}\) is increasing,
so
\(\widehat{\mathbb P}_n^{\mathrm s}(H_n)
\ge\nu_{\vartheta_n}^{\{1,\ldots,\rho_{k,n}\}}
 (Y_i=1\text{ for every }i)
=\vartheta_n^{\rho_{k,n}},\)
and therefore
\(\frac1{\rho_{k,n}}
\log\widehat{\mathbb P}_n^{\mathrm s}(H_n)
\ge\log\vartheta_n\longrightarrow0.
\)
This proves the probability assertion. On \(H_n\), all
\(E_{i,n}\) vanish, so Lemma~\ref{lem:short-localization} gives
\(\rho_{k,n}d_{\mathrm{TV}}(\xi_{k,n}^{\mathrm s},
\eta_{k,n}^{\mathrm s})\le\sum_iE_{i,n}=0\).
The two finite measures are therefore equal, which proves the remaining
assertion.
\end{proof}

\section{Numerical examples and simulations}
\label{sec:ex}

This section connects the variational formula in Corollary~\ref{cor:free}
with a numerical sampling procedure. We first identify the finite intensity
measure selected by the free-energy problem and its normalized cluster-shape
law. We then describe a Metropolis--Hastings implementation and illustrate the
resulting model for OU-Brownian trajectories with an anisotropic interaction.

\subsection{The Gibbs optimizer and its normalized shape law}
\label{ss:opt}

The free-energy variational problem is posed over finite measures on the
cluster state space \(E_k\), so its optimizer need not have mass one. Its
normalization nevertheless gives the probability law needed to sample the
shape of a typical cluster. The following calculation records both objects and
fixes the notation used in the remainder of the section.

Represent a cluster by
\[
z=(\gamma_1,x_2,\gamma_2,\ldots,x_k,\gamma_k),
\qquad x_1=0,
\]
where \(\gamma_1,\ldots,\gamma_k\) are displacement trajectories and
\(x_2,\ldots,x_k\) are relative spatial offsets. If the microscopic
interaction is \(v\), the induced cluster energy is
\[
W(z):=\int_0^1
v\big(\gamma_1(t),x_2+\gamma_2(t),\ldots,x_k+\gamma_k(t)\big)
\,\mathrm dt.
\]

\begin{lemma}[Gibbs optimizer]
\label{lem:gibbs-opt}
Let \(W:E_k\to\mathbb R\) be bounded and measurable, as in
Corollary~\ref{cor:free}. The unique minimizer in that corollary is the finite
measure
\[
\mu_\beta^{\mathrm{gc}}(\mathrm dz)
=e^{-\beta W(z)}\tau_k(\mathrm dz).
\]
Its mass and normalized cluster-shape law are
\[
Z_k(\beta):=\int_{E_k}e^{-\beta W(z)}\tau_k(\mathrm dz),
\qquad
\pi_\beta(\mathrm dz):=Z_k(\beta)^{-1}e^{-\beta W(z)}\tau_k(\mathrm dz).
\]
Moreover,
\[
\lim_{n\to\infty}\frac{1}{\rho_{k,n}}
\log\mathbb E\exp\big\{-\rho_{k,n}\beta H_W(\PP_n)\big\}
=\int_{E_k}\big(e^{-\beta W}-1\big)\,\mathrm d\tau_k.
\]
\end{lemma}

\begin{proof}
Write \(\mu=f\tau_k\). The functional in Corollary~\ref{cor:free} is
\(
\int_{E_k}\big(\beta Wf+f\log f-f+1\big)\,\mathrm d\tau_k.
\)
For each fixed \(z\), the integrand is strictly convex as a function of
\(f(z)>0\), and its derivative is \(\beta W(z)+\log f(z)\). Hence, its unique
minimum is attained at \(f(z)=e^{-\beta W(z)}\). Substitution gives the stated
free-energy identity, and normalization gives \(\pi_\beta\).
\end{proof}

The measure \(\mu_\beta^{\mathrm{gc}}\) specifies both the total cluster
intensity \(Z_k(\beta)\) and the normalized shape law \(\pi_\beta\). A natural
effective grand-canonical sampler first draws \(N\sim\operatorname{Poisson}
(Z_k(\beta))\) and then draws \(N\) independent shapes from \(\pi_\beta\). A
fixed-count sampler instead prescribes \(N\) and draws the same cluster-shape
law. This Poisson construction is the sampling model associated with the
variational optimizer; Corollary~\ref{cor:free} alone is not asserted to prove
convergence of the entire tilted point process.

\subsection{Metropolis--Hastings sampling}
\label{ss:mcmc}

The normalizing constant \(Z_k(\beta)\) is not needed to sample from
\(\pi_\beta\). It is enough to construct a proposal kernel on configurations
that satisfy the cluster diameter condition and compare the unnormalized
weights \(e^{-\beta W}\). In the ordered anchored coordinates used below, the
proposal is reversible with respect to the Lebesgue--Gaussian reference
measure whose restriction and pushforward define \(\tau_k\). The constant
symmetry factor therefore cancels together with the Hastings correction.
More generally, a nonreversible proposal must include the corresponding
proposal-density or Radon--Nikodym ratio. For a proposal kernel that is
reversible with respect to \(\tau_k\), the acceptance probability reduces to
\[
\alpha(z,z')
=\min\big\{1,\exp[-\beta(W(z')-W(z))]\big\}.
\]
If the proposed configuration fails the cluster diameter condition, it is
rejected before this acceptance test.

\subsection{OU--Brownian cluster model}
\label{ss:ou-brownian-simulation}

We now specialize the sampler to planar stochastic trajectories with an
Ornstein--Uhlenbeck velocity process and an additional Brownian positional
component. The numerical model combines persistent motion along a preferred
field direction, an anisotropic pair interaction, and Brownian positional
fluctuations. The inverse temperature is absorbed into dimensionless
coefficients so that it is applied exactly once in the Metropolis ratio.

The velocity dynamics can be motivated by Newton's law with a constant
driving force in the field direction, linear drag, and stochastic forcing.
Let \(V_i(t)\in\mathbb R^d\) denote the velocity of particle \(i\), let
\(\widehat b\) be the unit field direction, and write
\[
m\,dV_i(t)
=
\bigl(F\widehat b-\gamma V_i(t)\bigr)\,dt
+\eta\,dB_i^V(t),
\]
where \(m\) is the particle mass, \(F>0\) is the magnitude of the effective
driving force experienced by an individual particle along the field, \(\gamma>0\) is a drag coefficient, and
\(B_i^V\) is a standard \(d\)-dimensional Brownian motion. The deterministic
terms represent driving along the field and linear drag opposing the particle
velocity, respectively. Rearranging the deterministic part gives
\[
m\,dV_i(t)
=
-\gamma
\left(
V_i(t)-\frac{F}{\gamma}\widehat b
\right)dt
+\eta\,dB_i^V(t).
\]
Introducing the effective parameters
\[
\kappa:=\frac{\gamma}{m},
\qquad
v_0:=\frac{F}{\gamma},
\qquad
\sigma:=\frac{\eta}{m},
\]
the velocity process becomes
\begin{align}
dV_i(t)
&=
\kappa\bigl(v_0\widehat b-V_i(t)\bigr)\,dt
+\sigma\,dB_i^V(t).
\label{eq:ou-velocity}
\end{align}
Thus, \(v_0\widehat b\) is the stationary mean velocity,
\(\kappa^{-1}\) is the velocity relaxation time, and \(\sigma\) controls
the fluctuations around the mean. The process is initialized in its
stationary distribution,
\[
V_i(0)
\sim
N\left(
v_0\widehat b,
\frac{\sigma^2}{2\kappa}I_d
\right).
\]

The particle position additionally contains an independent Brownian
component. Writing \(B_i^X\) for another standard \(d\)-dimensional Brownian
motion, independent of \(B_i^V\), the displacement satisfies
\begin{align}
dX_i(t)
&=
V_i(t)\,dt
+\sqrt{2D}\,dB_i^X(t),
\qquad
X_i(0)=0,
\label{eq:ou-brownian-position}
\end{align}
where \(D>0\) is the positional diffusion coefficient. For fixed
\(v_0,\sigma,\kappa\), and \(D\), these dynamics define the reference path
law \(\mathbb Q\).

On a grid
\[
0=t_0<t_1<\cdots<t_M=T,
\qquad
\Delta t_\ell:=t_{\ell+1}-t_\ell,
\]
the Ornstein--Uhlenbeck transition can be sampled exactly. Setting
\[
a_\ell:=e^{-\kappa\Delta t_\ell},
\]
gives
\[
V_i^{\ell+1}
=
v_0\widehat b
+
a_\ell\bigl(V_i^\ell-v_0\widehat b\bigr)
+
\frac{\sigma}{\sqrt{2\kappa}}
\sqrt{1-a_\ell^2}\,Z_{i,\ell}^V,
\qquad
Z_{i,\ell}^V\sim N(0,I_d).
\]
We sample the velocity transition exactly and approximate the integrated velocity over each time step by the trapezoidal rule.
The resulting position update is
\[
X_i^{\ell+1}-X_i^\ell
=
\frac12\bigl(V_i^\ell+V_i^{\ell+1}\bigr)\Delta t_\ell
+
\sqrt{2D\Delta t_\ell}\,Z_{i,\ell}^X,
\qquad
Z_{i,\ell}^X\sim N(0,I_d),
\]
where the Gaussian variables \(Z_{i,\ell}^V\) and
\(Z_{i,\ell}^X\) are independent.
For a cluster
\[
z=(X_1,x_2,X_2,\ldots,x_k,X_k),
\qquad
x_1=0,
\]
define the local frame velocity and its direction by
\[
U_i^\ell
:=
\frac{X_i^{\ell+1}-X_i^\ell}{\Delta t_\ell},
\qquad
\widehat u_i^\ell
:=
\frac{U_i^\ell}{\lVert U_i^\ell\rVert}
\]
whenever \(\lVert U_i^\ell\rVert>0\), and set
\(\widehat u_i^\ell=0\) otherwise. Thus, \(\widehat u_i^\ell\) is computed
directly from the realized displacement over one frame.

\subsubsection{Dimensionless energy}

The numerical implementation uses the dimensionless energy
\(\mathcal E_\beta(z):=\beta W(z)\). This convention separates the physical
energy from the global inverse temperature while allowing the code to work
with effective coefficients. We decompose
\(\mathcal E_\beta=\mathcal E_{\mathrm{field}}+\mathcal E_{\mathrm{pair}}\),
and write \(\beta_{\mathrm{field}}=\beta h\) and
\(\beta_{\mathrm{pair}}=\beta\epsilon_{\mathrm{pair}}\). The field term is
\[
\mathcal E_{\mathrm{field}}(z)
=
-\sum_{\ell=0}^{M-1}\Delta t_\ell
\sum_{i=1}^k
\beta_{\mathrm{field}}
\widehat u_i^\ell\cdot\widehat b.
\]
For the pair term, define the midpoint position
\[
\overline x_i^\ell
:=
x_i+\frac12\bigl(X_i^\ell+X_i^{\ell+1}\bigr),
\qquad 1\le i\le k,
\]
and, for \(i<j\), let
\[
r_{ij}^\ell
:=
\lVert\overline x_j^\ell-\overline x_i^\ell\rVert,
\qquad
\widehat r_{ij}^\ell
:=
\frac{\overline x_j^\ell-\overline x_i^\ell}{r_{ij}^\ell}
\]
when \(r_{ij}^\ell>0\). The radial cutoff and directional cone are given by
\(f_{\mathrm{rad}}(r):=\mathbf 1\{0\le r\le r_{\max}\}\) and
\(w_{\mathrm{cone}}(\widehat r)
:=\mathbf 1\{|\widehat r\cdot\widehat b|\ge\cos\theta_0\}\),
respectively. The pair energy is then
\begin{align}
\mathcal E_{\mathrm{pair}}(z)
&=
-\sum_{\ell=0}^{M-1}\Delta t_\ell
\sum_{1\le i<j\le k}
\beta_{\mathrm{pair}}
f_{\mathrm{rad}}(r_{ij}^\ell)
w_{\mathrm{cone}}(\widehat r_{ij}^\ell)
\left[
2(\widehat u_i^\ell\cdot\widehat b)
(\widehat u_j^\ell\cdot\widehat b)
-
\widehat u_i^\ell\cdot\widehat u_j^\ell
\right].
\label{eq:pair-energy}
\end{align}
The interaction therefore favors motion along the field and head-to-tail
alignment of nearby particles whose separation is close to the field axis.
\subsubsection{Initialization and proposal moves}

The numerical implementation represents each trajectory through the Gaussian
random variables used to generate the Ornstein--Uhlenbeck velocity and
Brownian positional components. The initial state is obtained by setting the
first offset to \(x_1=0\), drawing the remaining offsets close to the origin,
and independently sampling these Gaussian variables from their reference
distributions. The corresponding velocities and positions are then
constructed from the dynamics above. The initial state need not be an exact
draw from the target law, and its influence is removed through burn-in.

Figure~\ref{fig:initial_realisation} shows one initial state. The figure is
included only to illustrate the geometry of the initialization.

\begin{figure}[ht]
\centering
\begin{subfigure}[t]{0.4\textwidth}
    \centering
    \includegraphics[width=0.9\textwidth]{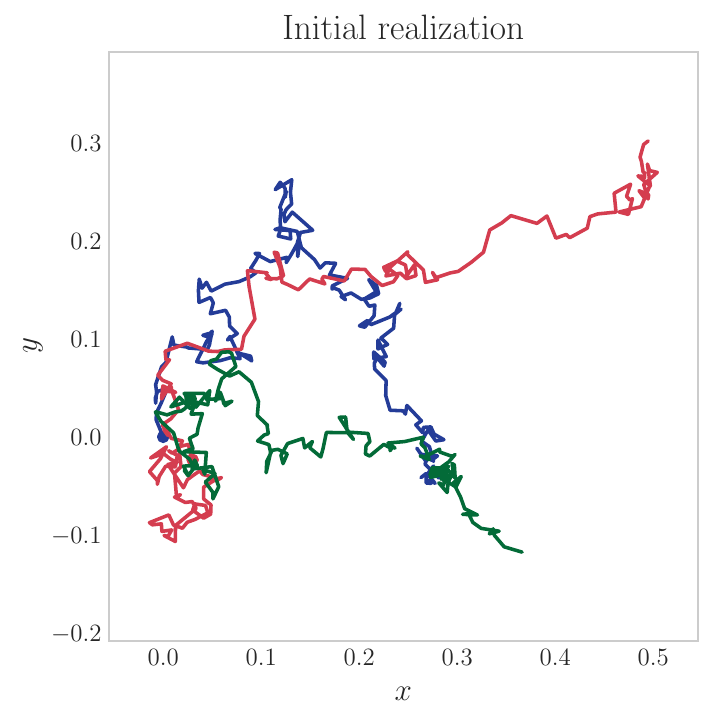}
    \caption{One initial realization of the connected cluster trajectories.}
    \label{fig:initial_realisation}
\end{subfigure}
\begin{subfigure}[t]{0.4\textwidth}
    \centering
    \includegraphics[width=0.9\textwidth]{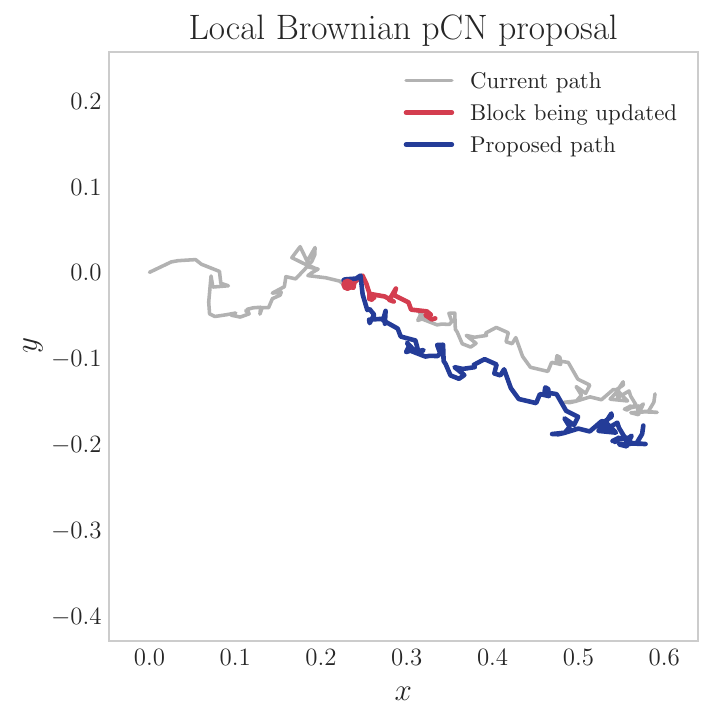}
    \caption{Local Brownian pCN proposal updating a consecutive block of Brownian increments.}
    \label{fig:brownian_bridge_proposal}
\end{subfigure}
\caption{Illustration of the initialization and path proposal used in the Metropolis--Hastings sampler.}
\label{fig:mcmc_initialization_proposal}
\end{figure}

At each Metropolis step, one of five proposal types is selected. Let
\(Z_i^V\) and \(Z_i^X\) collect the standard-Gaussian random variables used
to generate, respectively, the Ornstein--Uhlenbeck velocity process and the
Brownian positional increments of particle \(i\).

The Gaussian variables are updated using a preconditioned Crank--Nicolson
(pCN) proposal. For a current variable \(Z\), the proposal has the form
\[
Z'
=
\rho_{\mathrm{pCN}} Z
+
\sqrt{1-\rho_{\mathrm{pCN}}^2}\,\Xi,
\qquad
\Xi\sim N(0,I),
\]
where \(\Xi\) is independent of \(Z\) and
\(\rho_{\mathrm{pCN}}\in(0,1)\) controls how strongly the proposal is
correlated with the current value. In particular, if \(Z\sim N(0,I)\), then
\(Z'\sim N(0,I)\), since \(Z'\) is Gaussian with covariance
\[
\rho_{\mathrm{pCN}}^2 I
+
\bigl(1-\rho_{\mathrm{pCN}}^2\bigr)I
=
I.
\]
Thus, the pCN update preserves the standard-Gaussian reference distribution.

\begin{itemize}[label=\textbullet]
    \item \textbf{Whole OU move.}
    A particle \(i\) is selected and all Gaussian variables in \(Z_i^V\)
    are updated simultaneously using the pCN rule. The velocity and position
    trajectories of that particle are then reconstructed.

    \item \textbf{Whole Brownian move.}
    A particle \(i\) is selected and all Gaussian variables in \(Z_i^X\),
    which determine its Brownian positional increments, are updated using the
    pCN rule.

    \item \textbf{Local OU move.}
    A particle and a consecutive block of entries in \(Z_i^V\) are selected.
    Only the variables in this block are updated using the pCN rule, while
    all remaining entries are left unchanged.

    \item \textbf{Local Brownian move.}
    A particle and a consecutive block of Brownian increments are selected.
    The corresponding entries of \(Z_i^X\) are updated using the pCN rule,
    while all remaining entries are left unchanged.

    \item \textbf{Offset move.}
    A particle \(i\in\{2,\ldots,k\}\) is selected and its relative offset is
    proposed according to
    \[
    x_i'=x_i+\sigma_xG,
    \qquad
    G\sim N(0,I_d).
    \]
\end{itemize}

The pCN proposals are reversible with respect to the Gaussian reference
distribution, while the offset proposal is symmetric with respect to
Lebesgue measure. Their mixture is therefore reversible with respect to the
Lebesgue--Gaussian reference measure defining the trajectory law. After each
proposal, the velocities and positions are reconstructed from the updated
Gaussian variables. A proposal that fails the cluster diameter condition is
rejected. Otherwise, the Metropolis acceptance probability is
\[
\alpha(z,z')
=
\min\bigl\{
1,
\exp[-(\mathcal E_\beta(z')-\mathcal E_\beta(z))]
\bigr\}.
\]

\subsubsection{Parameters and simulation output}

Unless stated otherwise, the simulations use the parameters listed in
Table~\ref{tab:simulation_parameters} in the appendix. In all figures below,
\(k=3\).

Figure~\ref{fig:simulation_varying_D} compares four diffusion constants, $D$. The
trajectory insets show that larger \(D\) permits larger excursions, while the
empirical local-speed distributions shift towards larger speeds and become
broader with longer upper tails. In particular, changing \(D\) also changes
the shape of the local-speed distribution, from a more concentrated and
approximately symmetric distribution to a more right-skewed distribution.

\begin{figure}
\centering
\begin{subfigure}[t]{\textwidth}
    \centering
    \includegraphics[width=\textwidth]{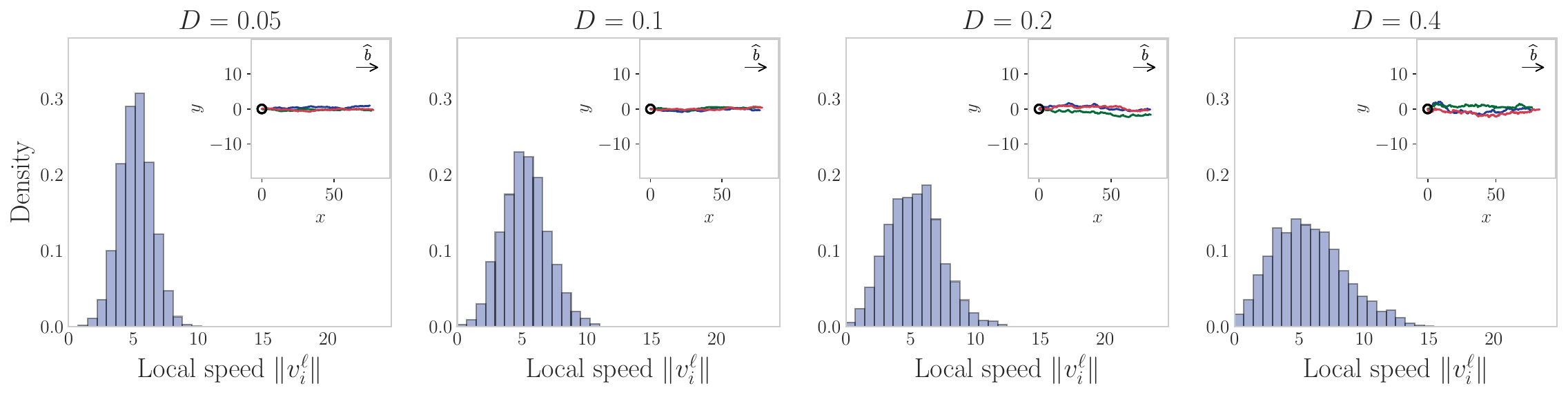}
    \caption{Different diffusion constants \(D\).}
    \label{fig:simulation_varying_D}
\end{subfigure}

\vspace{0.8em}

\begin{subfigure}[t]{\textwidth}
    \centering
    \includegraphics[width=\textwidth]{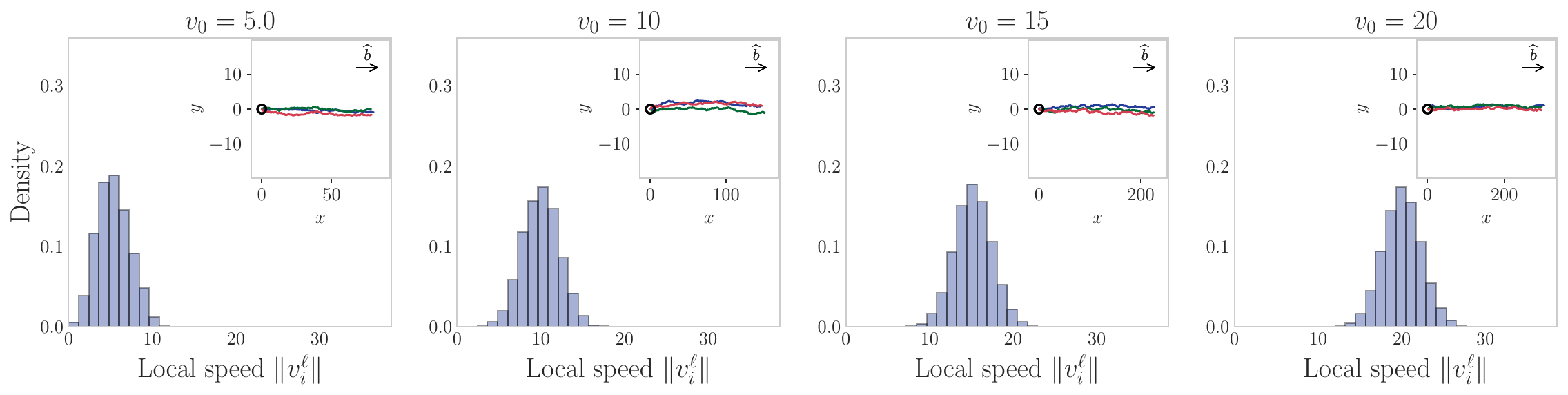}
    \caption{Different characteristic velocities \(v_0\).}
    \label{fig:simulation_varying_v0}
\end{subfigure}

\caption{
Simulated local-speed distributions for clusters of size \(k=3\). Panel
\textup{(a)} varies the diffusion constant \(D\), using a baseline
characteristic velocity of \(v_0=5\). Panel \textup{(b)} varies the
characteristic velocity \(v_0\), with the diffusion constant fixed at
\(D=0.15\). The insets show the corresponding simulated trajectories.
Within each inset, different colors correspond to different particles,
the open circle marks the starting point of one trajectory, and the arrow
indicates the magnetic-field direction.
}
\label{fig:simulation_parameter_comparison}
\end{figure}

Figure~\ref{fig:simulation_varying_v0} varies the characteristic velocity
\(v_0\), which is the stationary mean velocity in the field direction.
Increasing \(v_0\) shifts the empirical local-speed distribution towards
larger speeds. Together with Figure~\ref{fig:simulation_varying_D}, this shows
that both \(v_0\) and \(D\) affect the velocity distribution, while \(D\) also
changes its shape. The associated trajectory insets show the corresponding
geometric variability under the connectivity constraint. Compared with the initial realization in Figure~\ref{fig:initial_realisation},
the retained trajectories are more strongly aligned with the field direction
and therefore follow straighter paths.

\section{Application to experimental trajectory data}
\label{sec:real-data}

The preceding sections develop a large deviation framework for stochastic
clusters and a simulation model derived from its variational structure. We
now turn to a first application of this framework. Experiments with micro-
and nanoparticles can be costly and time consuming, particularly when several
experimental parameters are varied simultaneously. If the objective is to
understand how particle properties and external forcing influence the
resulting collective motion, an exhaustive experimental exploration of the
parameter space can therefore be impractical.

A natural long-term objective is to use the stochastic model as a tool for
virtual materials design. Rather than testing every parameter configuration
experimentally, one may first explore a large collection of configurations
numerically and then restrict laboratory experiments to those conditions that
appear most informative or promising. Related inverse-design approaches based
on statistical-physics models have been developed for self-assembling
materials \cite{miskin:etal:2016}, including approaches that use large
deviation theory to design nonequilibrium colloidal assembly
\cite{das:limmer:2021}. Here we investigate whether the cluster model
developed above can provide a useful starting point for such a strategy in the
setting of magnetically driven micromotors.

Our purpose in this section is deliberately modest. We perform a pilot study
in which the stochastic cluster model is calibrated against experimental
trajectory data and examine how the resulting effective parameters vary with
the physical experimental conditions. Already at this level, the fitted
characteristic velocity follows a clear systematic pattern across particle
sizes and forcing conditions, while the remaining calibrated parameters
exhibit less regular behavior. Thus, the analysis should be viewed as a proof
of concept for using simulation to explore experimentally relevant parameter
regimes rather than as a complete virtual materials design procedure.

The connection with the large deviation theory is through the variational
problem of Corollary~\ref{cor:free}. For a bounded cluster energy \(W\), its
minimizer is \(\mu_\beta^{\mathrm{gc}}=e^{-\beta W}\tau_k\), whose
normalization \(\pi_\beta\) gives the Gibbs law of a typical cluster shape.
In Section~\ref{sec:ex}, this normalized law provides the target distribution
for Metropolis--Hastings sampling and is specialized to Ornstein--Uhlenbeck
velocity dynamics with Brownian positional fluctuations, together with field
and pair interaction terms.

The experimental data consist of microscopy movies of polystyrene-based
particles coated with magnetic nanoparticles and driven by external magnetic
fields. Particle tracking yields, for each movie, a finite collection of
planar trajectories observed over a common time interval. The particles have
diameters \(4,10,20\), and \(40\,\mu\mathrm m\), and the experimental
conditions vary both the particle concentration and the strength of the
external magnetic forcing. For each condition, movies consisting of \(300\)
frames were recorded at \(16.67\) frames per second using a \(40\times\)
objective. Particle trajectories were reconstructed with the TrackMate plugin
in Fiji
\cite{ershov:phan:pylvanainen:etal:2022,schindelin:argandacarreras:etal:2012}.

The comparison uses two experimental inputs. The observed chain-length
frequencies determine the mixture weights over cluster sizes, while the
trajectory data determine the empirical velocity distributions and the
reference scales used in the calibration. For each experimental condition, we
sample the stochastic cluster law for the relevant cluster sizes, combine the
resulting velocity distributions according to the observed chain-length
frequencies, and calibrate the remaining effective parameters by comparison
with the experimental velocity distribution. Thus, the data analysis examines
whether the stochastic model motivated by the large deviation variational
principle can reproduce the observed experimental velocity distributions; it
does not constitute an empirical test of the large
deviation principle itself.

The experimental protocol, including particle assembly, magnetic
characterization, imaging, and tracking details, is collected in
Appendix~\ref{app:experimental-protocol}. Background on magnetically driven
micro- and nanorobots can be found in
\cite{zhou:mayorgamartinez:pane:zhang:pumera:2021}; related micromotor
locomotion experiments are described in
\cite{dediosandres:ramosdocampo:qian:stingaciu:stadler:2021}.

\subsection{Preprocessing and mixture model}

For the comparison with the stochastic cluster model, we focus on the
low-concentration, or \(\times1\), measurements. Each experimental condition
is specified by the particle diameter \(s\) and the applied magnetic forcing
\(F\). The calibration uses the observed velocity measurements together with
the empirical frequencies of the different chain lengths.

For a fixed condition \((s,F)\), let
\(
\{v^{\mathrm{exp}}_1,\ldots,v^{\mathrm{exp}}_{N_{s,F}}\}
\)
denote the corresponding collection of experimentally measured velocities.
These data are used in the model fitting. Since higher-order statistics, and
in particular skewness, can be strongly affected by isolated extreme
observations, an upper-tail cleaning step was applied separately within each
experimental condition. More precisely, for each condition we defined
\[
c_{s,F}
=
\min\big\{
m_{s,F}+8\times1.4826\,\mathrm{MAD}_{s,F},
q_{0.995,s,F}
\big\},
\]
where \(m_{s,F}\) is the sample median,
\(\mathrm{MAD}_{s,F}\) is the median absolute deviation, and
\(q_{0.995,s,F}\) is the empirical \(99.5\%\) quantile. The median absolute
deviation is a standard robust measure of scale
\cite{rousseeuw:croux:1993}. The factor \(1.4826\) is the usual
normal-consistency correction: for a Gaussian distribution,
\(\mathrm{MAD}=\Phi^{-1}(0.75)\sigma\approx0.6745\sigma\), so
\(1.4826\,\mathrm{MAD}\) estimates the standard deviation. This normalization
has also been used in motility analysis \cite{prummer:etal:2013}, while
conservative thresholds of eight MADs have been used in applied data
preprocessing \cite{tan:etal:2024}. Thus, the MAD component of the threshold
can be viewed as a robust analogue of an eight-standard-deviation upper-tail
cutoff.

Observations exceeding \(c_{s,F}\) were excluded from the calibration. As a
safeguard, no observations were removed if the resulting sample for a given
condition would contain fewer than ten velocity measurements. The same cutoff
was applied to the simulated velocity samples before comparison with the
experimental data. Since the \(99.5\%\) quantile may be more restrictive than
the MAD component, the cleaning rule can affect precisely the extreme upper
tail that also influences skewness. The comparisons involving skewness and
tail behavior should therefore be interpreted with this preprocessing choice
in mind.

The experimental samples contain chains with different numbers of particles.
This creates an additional issue because the large deviation model is
formulated for a fixed cluster size \(k\). We therefore represent the observed
velocity distribution as a mixture of the simulated distributions for the
different experimentally observed chain lengths. More precisely, we set
\[
P^{\mathrm{sim}}_{s,F}
=
\sum_{K\in\mathcal K_{s,F}}
p_{K,s,F}P^{\mathrm{sim}}_{K,s,F},
\]
where \(\mathcal K_{s,F}\) denotes the set of chain lengths observed under
condition \((s,F)\), \(P^{\mathrm{sim}}_{K,s,F}\) denotes the velocity
distribution generated by the stochastic cluster model for a cluster
containing \(K\) particles, and \(p_{K,s,F}\) denotes the experimentally
observed frequency of chains of length \(K\) under condition \((s,F)\). As
shown in Figure~\ref{fig:supp-micromotor-low}, the vast majority of chains in
the dilute setting have length at most four. All experimentally observed chain
lengths are nevertheless included in the numerical mixture.

\subsection{Calibration of the stochastic cluster model}

We next describe the calibration of the stochastic cluster model. Rather than
calibrating a small collection of summary statistics, the present procedure
compares the complete experimental and simulated velocity distributions. This
allows differences in the overall distributional shape to enter directly into
the fitting criterion.

For each experimentally observed chain length, cluster configurations are
sampled using the MCMC procedure described in
Section~\ref{ss:ou-brownian-simulation}. The Markov chain is restricted to
connected cluster configurations, and samples retained after burn-in and
thinning are used to construct the simulated velocity distributions. These
distributions are combined using the empirical chain-length weights described
above.

For a candidate parameter vector \(\vartheta\), let
\(\widehat F^{\mathrm{sim}}_{s,F,\vartheta}\) denote the empirical cumulative
distribution function of the resulting simulated mixture, and let
\(\widehat F^{\mathrm{exp}}_{s,F}\) denote the corresponding empirical
cumulative distribution function of the cleaned experimental velocities. The
calibration criterion is the two-sample Kolmogorov--Smirnov distance
\[
D_{\mathrm{KS},s,F}(\vartheta)
=
\sup_{u}
\left|
\widehat F^{\mathrm{exp}}_{s,F}(u)
-
\widehat F^{\mathrm{sim}}_{s,F,\vartheta}(u)
\right|.
\]
The fitted parameter vector is selected by minimizing
\(D_{\mathrm{KS},s,F}(\vartheta)\). Thus, the calibration compares the
experimental and simulated distributions directly rather than matching a
prespecified set of moments. The empirical mean, standard deviation, and
skewness are retained as descriptive diagnostics but do not enter the fitting
criterion.

The flexibility of the stochastic cluster model creates a potential
overparametrization problem because several model parameters can influence the
velocity distribution in similar ways. We therefore restrict the parameter
ranges before performing the numerical search. The global modeling parameters
are fixed at
\[
\beta=10,
\qquad
h_{\mathrm{scale}}=4,
\qquad
L=4s,
\qquad
r_{\max}=1.5s,
\qquad
\theta_0=\frac{\pi}{6},
\]
where \(s\) denotes the particle diameter. The field parameter is consequently
set to \(h=h_{\mathrm{scale}}F=4F\). The Ornstein--Uhlenbeck relaxation rate
is fixed at
\[
\kappa=1.8\,\mathrm{s}^{-1}.
\]
These quantities are held fixed across all experimental conditions and are
not interpreted as condition-specific estimates of microscopic physical
constants.

As before, the remaining parameters are shared by all cluster sizes within a
fixed experimental condition. The dependence on the cluster size nevertheless
enters through the stochastic cluster dynamics and through the empirical
mixture weights \(p_{K,s,F}\). This restriction substantially reduces the
number of free parameters and avoids fitting a separate parameter vector to
each chain length.

The parameters varied in the calibration are the characteristic OU velocity
\(v_0\), the OU noise amplitude \(\sigma\), the pair-interaction strength
\(\epsilon_{\mathrm{pair}}\), and the positional diffusion coefficient \(D\).
For each experimental condition, \(v_0\) and \(\sigma\) are scaled relative
to the empirical mean and standard deviation,
\[
v_0=c_v\widehat\mu_{s,F},
\qquad
\sigma=c_\sigma\widehat\sigma_{s,F}.
\]
The coarse candidate set contains \(11\) values of \(c_v\), \(8\) values of
\(c_\sigma\), \(11\) values of \(\epsilon_{\mathrm{pair}}\), and \(5\) values
of \(D\), with \(\kappa=1.8\,\mathrm{s}^{-1}\) fixed. The complete Cartesian
grid therefore contains
\[
11\times8\times11\times5=4840
\]
candidate parameter combinations. Rather than evaluating the complete grid,
\(100\) deterministic approximately space-filling points are selected from
this set for the initial search.

The coarse search is followed by a local refinement. The three coarse
parameter combinations with the smallest Kolmogorov--Smirnov distances are
used as starting points for bounded Nelder--Mead optimization. The refinement
is performed within the bounds
\[
0\le v_0\le1.70\,\widehat\mu_{s,F},
\qquad
0.20\,\widehat\sigma_{s,F}
\le\sigma\le
2.30\,\widehat\sigma_{s,F},
\]
\[
2\le\epsilon_{\mathrm{pair}}\le40,
\qquad
0.08\le D\le0.80,
\qquad
\kappa=1.8\,\mathrm{s}^{-1}.
\]
The four best candidates from the coarse search together with the three
locally refined candidates are then evaluated using longer MCMC simulations.
After removing duplicate parameter combinations, the parameter set with the
smallest Kolmogorov--Smirnov distance in this final reranking is retained.

The numerical settings used for the MCMC simulations, including burn-in,
thinning, proposal parameters, and chain lengths, are listed in
Table~\ref{tab:simulation_parameters} in the appendix. For each candidate
parameter vector, separate chains are simulated for the experimentally
observed cluster sizes and combined using their empirical chain-length
frequencies.

\subsection{Comparison between experimental and simulated distributions}

We now compare the fitted stochastic cluster model with the experimental
velocity distributions. Since the same experimental distributions are used to
select the fitted parameters, these comparisons are in-sample goodness-of-fit
diagnostics rather than out-of-sample validation. Their purpose is to
visualize the agreement measured by the Kolmogorov--Smirnov criterion and to
identify any remaining differences in the shapes of the distributions.

Figures~\ref{fig:velocity-overlays-size-4} and
\ref{fig:velocity-overlays-size-10} compare the cleaned experimental velocity
distributions with the fitted simulated mixtures for \(4\,\mu\mathrm m\)
particles (\({}^{4}M_{\mathrm{PEG}}\)) and \(10\,\mu\mathrm m\) particles
(\({}^{10}M_{\mathrm{PEG}}\)), respectively, for the \(\times1\) condition. 
The
corresponding comparisons for \(20\,\mu\mathrm m\) particles
(\({}^{20}M_{\mathrm{PEG}}\))  is provided in the 
supplementary material. No cluster-model calibration is performed for  \(40\,\mu\mathrm m\) particles
(\({}^{40}M_{\mathrm{PEG}}\)) because no chains were observed.
Each figure contains the four magnetic forcing conditions. The upper row shows
the experimental and fitted simulated velocity distributions, while the lower
row shows the corresponding empirical cumulative distribution functions. For
each condition, the simulated distribution is the empirical mixture over the
experimentally observed chain lengths.

\begin{figure}[!htbp]
    \centering
    \includegraphics[width=\textwidth]
    {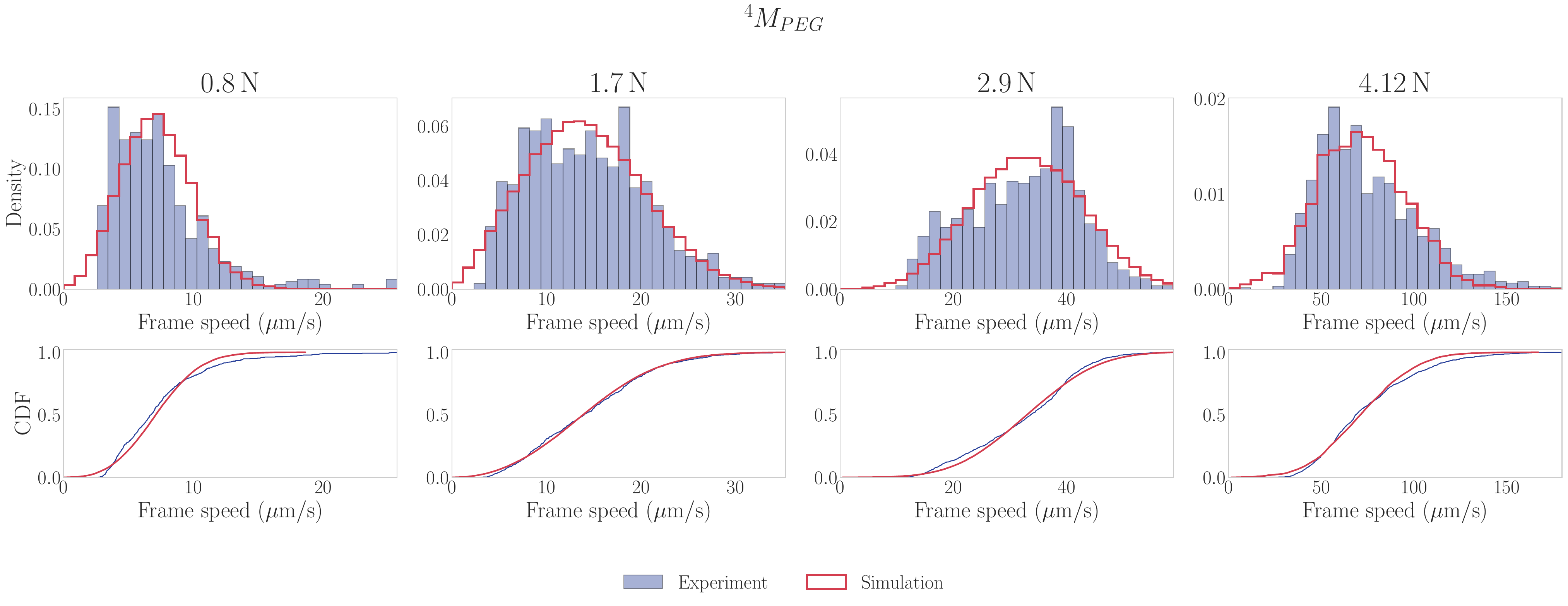}
    \caption{
    Experimental and fitted simulated velocity distributions for
    \({}^{4}M_{\mathrm{PEG}}\) particles for the \(\times1\) condition at nominal magnetic pulling force values  \(0.8\), \(1.7\), \(2.9\), and \(4.12\) N.
    }
    \label{fig:velocity-overlays-size-4}
\end{figure}

\begin{figure}[!htbp]
    \centering
    \includegraphics[width=\textwidth]
    {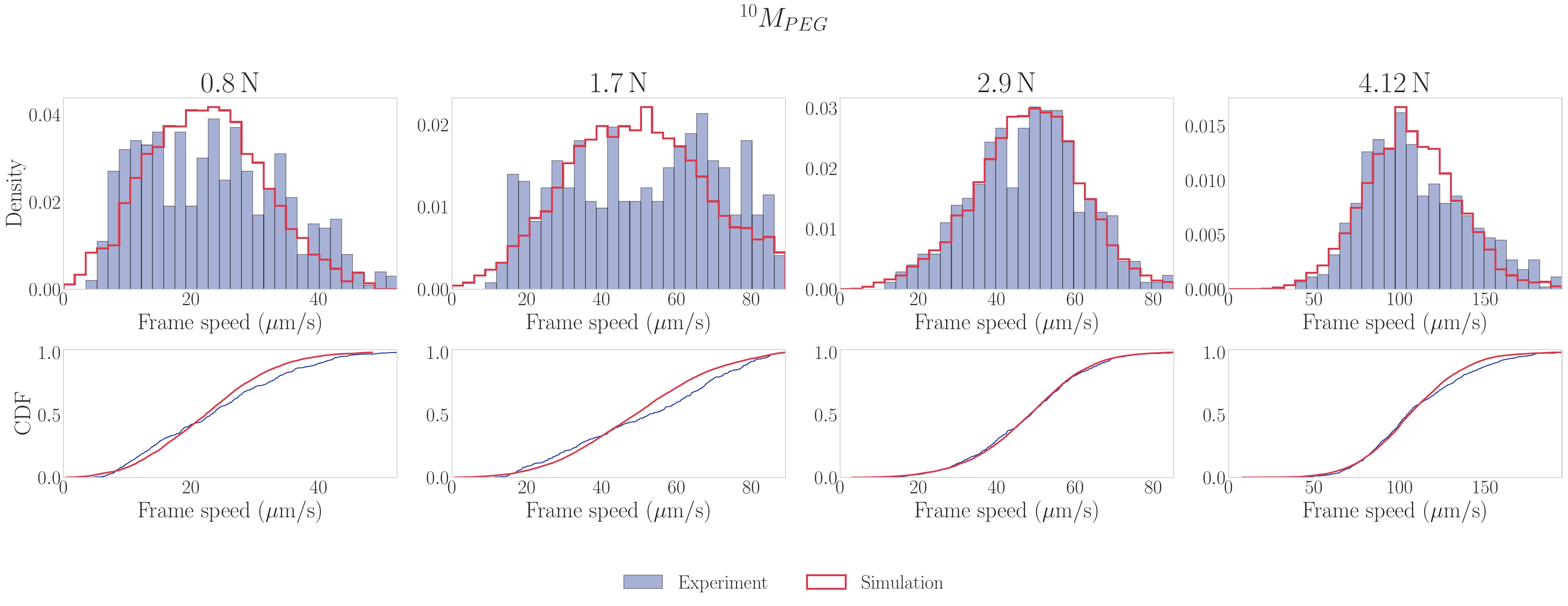}
    \caption{
    Experimental and fitted simulated velocity distributions for
    \({}^{10}M_{\mathrm{PEG}}\) particles for the \(\times1\) condition at nominal magnetic pulling force values  \(0.8\), \(1.7\), \(2.9\), and \(4.12\) N.
    }
    \label{fig:velocity-overlays-size-10}
\end{figure}

These distributional overlays provide a visual counterpart to the
Kolmogorov--Smirnov calibration criterion. The Kolmogorov--Smirnov distance
measures the largest difference between the empirical cumulative distribution
functions of the experimental and simulated samples, so smaller values
indicate closer distributional agreement. The corresponding
Kolmogorov--Smirnov distances are reported in
Table~\ref{tab:ks-x1} in the appendix.

Overall, the fitted stochastic cluster model reproduces the experimental
velocity distributions with varying degrees of agreement across the
experimental conditions. Some experimental distributions exhibit shoulders
or multimodal features that are not fully reproduced by the fitted
simulations. Although these discrepancies contribute to the
Kolmogorov--Smirnov criterion, the criterion records only the largest
difference between the two empirical cumulative distribution functions and
does not describe the location or nature of the remaining mismatch. A natural
extension is therefore to examine velocity distributions conditional on the
observed chain length. This would make it possible to distinguish
multimodality arising from the mixture over different chain lengths from
multimodality already present within a fixed chain length.

\subsection{Trends in the calibrated model parameters}
\label{ss:parameter-trends}

We next examine how the fitted characteristic velocity varies across particle
sizes and magnetic forcing values. Figure~\ref{fig:parameter-trends} shows the
fitted Ornstein--Uhlenbeck mean velocity \(v_0\) as a function of magnetic
forcing, with separate curves corresponding to the different particle sizes.
Recall that \(v_0\widehat b\) is the stationary mean velocity of the
Ornstein--Uhlenbeck process, so \(v_0\) represents the characteristic velocity
in the field direction. It arises from the force-balance interpretation
\[
m\,dV_t=(F\widehat b-\gamma V_t)\,dt+\eta\,dB_t^V,
\]
where \(F\) is the effective driving force experienced by an individual
particle along the field, \(\gamma\) is the linear drag coefficient, and
\(\eta\) is the stochastic-force amplitude, giving \(v_0=F/\gamma\). The force \(F\) should not be confused with the manufacturer-specified
nominal magnetic pull force \(F_{\mathrm{mag}}\), which is used to
label the experimental magnet conditions.

\begin{figure}[!htbp]
    \centering
    \includegraphics[width=0.65\textwidth]
    {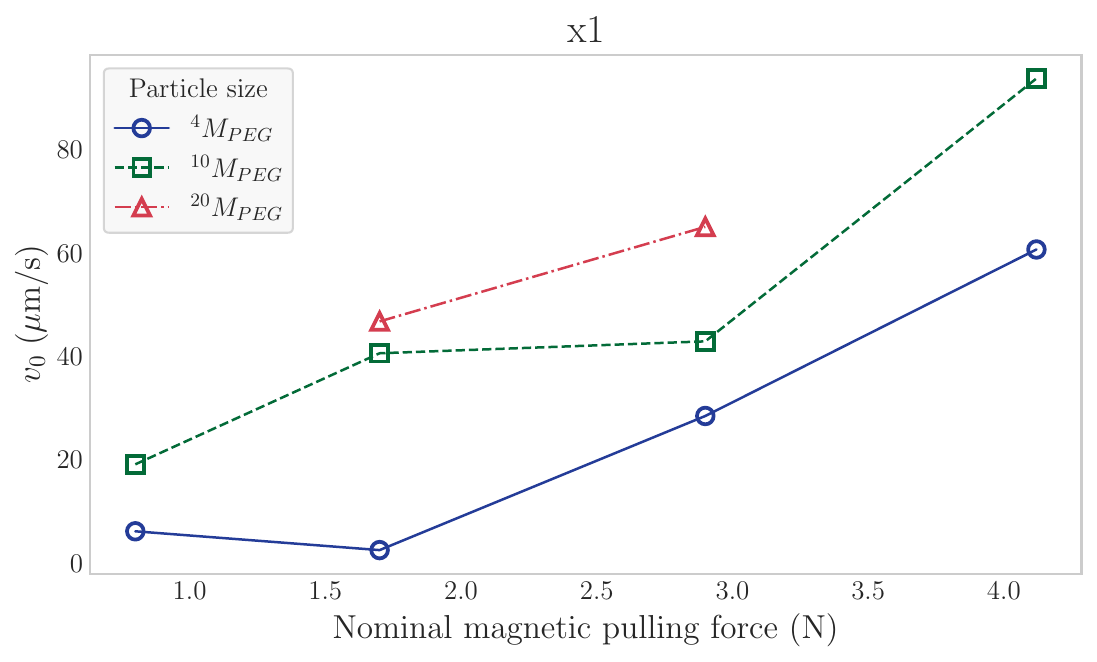}
    \caption{
    Fitted OU mean velocity \(v_0\) for the \(\times1\)
    condition as a function of 
    nominal magnetic pulling force, with separate curves
    corresponding to the different particle sizes.
    }
    \label{fig:parameter-trends}
\end{figure}

The fitted values show an overall increase in the characteristic velocity with
stronger magnetic forcing. The dependence is not strictly monotone for every
particle size, but the highest fitted velocities occur under the stronger
forcing conditions. The magnitude of the response also differs between
particle sizes, indicating that the effective velocity scale depends on both
the particle size and the applied forcing.

\subsection{Interpretation and limitations}

The experimental analysis is intended as a first test of whether the
stochastic cluster law can serve as a useful computational description of the
observed micromotor trajectories. It is not intended to identify a unique
microscopic interaction mechanism from velocity measurements alone. The
fitted quantities should therefore be interpreted as effective parameters of
the present stochastic model rather than as direct measurements of microscopic
physical constants.

The comparison incorporates several quantities obtained directly from the
experiment, including the observed chain-length composition and the applied
forcing. Within each experimental condition, the fitted parameters are shared
across the experimentally observed cluster sizes, while the contribution of
each cluster size to the final simulated distribution differs through the
stochastic cluster dynamics and the empirical mixture weights.

The fitted simulations reproduce the overall experimental velocity
distributions with varying degrees of agreement across the conditions.
Remaining discrepancies include localized differences in distributional shape,
such as shoulders or possible multimodality. Possible physical sources include
heterogeneity in the magnetic coating, intermittent particle interactions,
hydrodynamic or substrate-mediated effects, tracking uncertainty, and
variability between particles within the same nominal experimental condition.

Three limitations of the present calibration are particularly relevant.
First, several model parameters can influence the velocity distribution in
similar ways. The fitted values may therefore not be uniquely identifiable
from the velocity distribution alone. In particular, changes in the
Ornstein--Uhlenbeck velocity parameters, the positional diffusion coefficient,
and the interaction strength can produce partially overlapping effects on the
resulting speed distribution. The relaxation rate \(\kappa\) is consequently
held fixed in the present calibration, reducing the dimension of the
parameter search.

Second, the calibration is based on stochastic MCMC simulations. Consequently,
the Kolmogorov--Smirnov distance obtained for a fixed parameter vector contains
Monte Carlo variability. The numerical search mitigates this by combining a
space-filling coarse search with local refinement and a final reranking using
longer simulations, but the fitted parameter values remain subject to
simulation uncertainty. Moreover, although the Kolmogorov--Smirnov criterion
compares the complete empirical distributions, it summarizes their discrepancy
through the largest difference between their cumulative distribution
functions and does not identify the origin of a remaining mismatch.

Third, the available experimental data limit how finely the dependence on the
physical control parameters can be resolved. Measurements are available only
at a small number of magnetic forcing values, so the present analysis cannot
reliably determine a continuous response of the effective model parameters to
the forcing strength. In addition, the velocity distributions used for
calibration are pooled across chain lengths. If sufficiently many velocity
measurements were available separately for each cluster size \(K\), the
simulated distributions \(P^{\mathrm{sim}}_{K,s,F}\) could be compared
directly with their experimental counterparts rather than only through the
mixture over the observed chain lengths. Such data would provide more direct
information about chain-length effects and help reduce ambiguity in the fitted
parameters.

The present distributional overlays are also in-sample diagnostics because the
same experimental data are used for calibration and assessment. A stronger
future test of the virtual-design perspective would estimate the relation
between experimental controls and effective model parameters using only a
subset of the experimental conditions and then predict velocity distributions
at held-out particle sizes or forcing values. Such a design would test whether
the model can interpolate or extrapolate to experimental conditions that were not
used in its calibration.
\bigskip

{\bf Acknowledgements}
The authors used OpenAI's ChatGPT during the preparation of this manuscript. All mathematical statements, proofs, computations, references, and conclusions were independently checked and verified by the authors, who take full responsibility for the content of the manuscript.
This work was supported by a research grant from VILLUM FONDEN (C.H., B.S. grant number VIL69126).

\section*{Declarations}

{\bf Competing interests} The authors have no competing interests to declare that are relevant to the content of this article.

{\bf Data availability} The code used for the numerical experiments and the calibration analysis is provided as supplementary material. The experimental data analyzed in this study are available from the corresponding authors upon reasonable request.

\bibliographystyle{abbrv}
\bibliography{lit}

\providecommand{\MPEG}[1]{{}^{#1}\!M_{\mathrm{PEG}}}
\newpage

\appendix
\section{Experimental micromotor system and trajectory data}
\label{sec:supp-micromotor-experiment}
\label{app:experimental-protocol}

This section describes the fabrication, imaging, trajectory extraction, and
empirical characterization of the magnetic micromotors. Magnetically driven
micro- and nanorobots form a standard class of externally actuated particle
systems; see \cite{zhou:mayorgamartinez:pane:zhang:pumera:2021}
for a review. 
The present experiment provides planar trajectories of moving
particles under controlled changes of particle size, relative particle
concentration, and magnetic condition. In particular, the manufacturer-specified magnet values are used only to label the nominal magnetic pulling-force conditions and are not interpreted
as forces acting on individual micromotors.

\subsection{Materials}
\label{sec:supp-micromotor-materials}

Sodium chloride (NaCl, 99\%),
poly(diallyldimethylammonium chloride) (PDDA, 100-200 kDa, 20 wt.\% in
$\mathrm{H}_2\mathrm{O}$),
poly(sodium 4-styrenesulfonate) (PSS, molecular weight
$70\,\mathrm{kDa}$), and polystyrene particles (PS, diameters
$4$, $10$, $20$, and $40\,\mu\mathrm{m}$, 10 wt.\%) were purchased from
Sigma-Aldrich. Untreated $\mu$-Slides VI 0.4 were purchased from ibidi GmbH,
and NdFeB permanent magnets were purchased from Supermagnete. Ultrapure water
(resistivity $18.2\,\mathrm{M}\Omega\,\mathrm{cm}$)
was supplied by a Synergy UV system from Millipore. Magnetic nanoparticles
(MNPs) and poly(L-lysine)-g-polyethylene glycol (PLL-g-PEG) were synthesized as
previously described in
\cite{dediosandres:ramosdocampo:qian:stingaciu:stadler:2021}.

\subsection{Micromotor assembly and surface characterization}
\label{sec:supp-micromotor-assembly}

The micromotors were assembled by a layer-by-layer procedure. All polymers were
dissolved in $0.5\,\mathrm{M}$ NaCl. A volume of $2\,\mathrm{mL}$ of polymer
solution was used in every deposition step, and all washing steps were carried
out in ultrapure water. First, $100\,\mu\mathrm{L}$ of PS stock suspension was
mixed with PDDA at $2\,\mathrm{mg}\,\mathrm{mL}^{-1}$ for $15\,\mathrm{min}$.
The particles were then subjected to three washing cycles at
$6000\,\mathrm{rpm}$ for $5\,\mathrm{min}$ in an Eppendorf MiniSpin
centrifuge. The particles were incubated with PSS at 20 vol.\% for
$15\,\mathrm{min}$ and washed three times. MNPs were subsequently deposited at
$20\,\mu\mathrm{g}\,\mathrm{mL}^{-1}$ for $15\,\mathrm{min}$, followed by
three washing cycles and a second PSS deposition step. Finally, PLL-g-PEG was
deposited at $2\,\mathrm{mg}\,\mathrm{mL}^{-1}$ for $15\,\mathrm{min}$ as the
terminating layer, followed by three washing cycles.

The resulting micromotors are denoted by $\MPEG{x}$, where
$x\in\{4,10,20,40\}$ is the nominal PS-core diameter in micrometres. The
$\zeta$-potential was measured in ultrapure water after each deposition step
using a Malvern Zetasizer 4. Transmission electron microscopy (TEM) and
scanning electron microscopy (SEM) were used to inspect the particle
morphology after MNP deposition.
Figure~\ref{fig:supp-micromotor-characterisation} summarises the assembly and
the corresponding surface characterisation.

\begin{figure}[htbp]
\centering
\includegraphics[width=\textwidth]{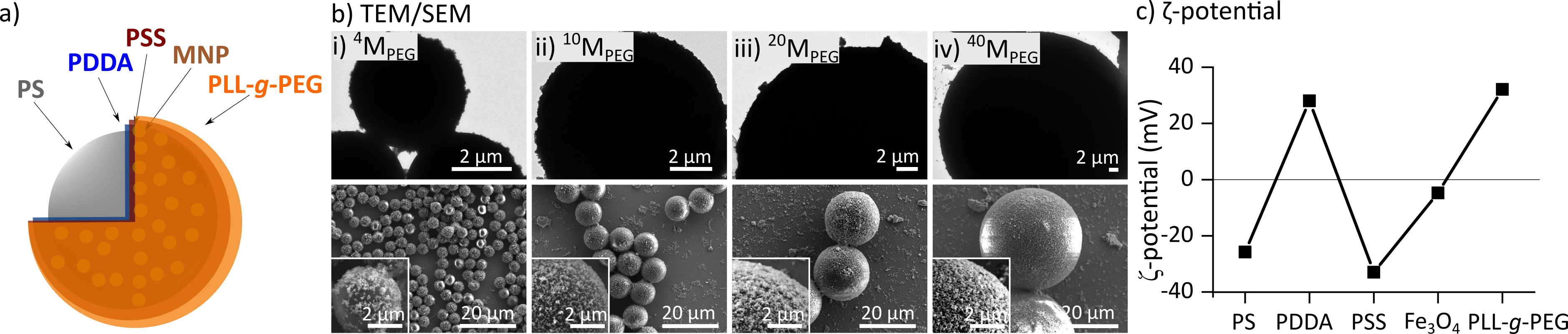}
\caption{Micromotor assembly and characterisation. (a) Schematic illustration
of the layers in the assembled $\MPEG{x}$ micromotors: a PS core, PDDA and PSS
precursor layers, MNPs, and a PLL-g-PEG terminating layer. (b) Representative
TEM images (top) and SEM images (bottom) of $\MPEG{4}$, $\MPEG{10}$,
$\MPEG{20}$, and $\MPEG{40}$. Scale bars: $2\,\mu\mathrm{m}$ (top) and  $20\,\mu\mathrm{m}$ (bottom). (c)
$\zeta$-potential after the successive deposition steps, showing the expected
charge alternation during layer-by-layer assembly.}
\label{fig:supp-micromotor-characterisation}
\end{figure}

\subsection{Locomotion experiments}
\label{sec:supp-micromotor-locomotion}

For each experiment, $10\,\mu\mathrm{L}$ of micromotor suspension was dispersed
in $90\,\mu\mathrm{L}$ of ultrapure water and transferred to an untreated
$\mu$-Slide VI 0.4. Two relative particle-concentration conditions were used:
the $\times1$ condition and the $\times10$ condition, prepared with
approximately ten times the particle amount. Absolute final concentrations
cannot be reported reliably because the particle concentration changes
substantially during micromotor assembly. The labels $\times1$ and $\times10$
should therefore be understood as operational experimental conditions rather
than absolute concentration measurements.

The particles were allowed to settle for $2\,\mathrm{min}$ in the
microfluidic channel. A permanent magnet was then placed in the middle of the
channel. Four magnets were used, with manufacturer-specified nominal pull
forces of $0.8$, $1.7$, $2.9$, and $4.12\,\mathrm{N}$. These values describe
the maximum force with which the magnet can hold a magnetic object under the
manufacturer's test conditions. They do not quantify the force exerted on an
individual micromotor in the microfluidic channel. In the figures and
discussion below, the nominal pull force is denoted by $F_{\mathrm{mag}}$ and
serves only as a label for the four magnet conditions. Control movies in the
absence of a magnet were also recorded; the comparisons reported here concern
the four magnet conditions.

For every combination of particle size, relative concentration, and magnet,
three videos were recorded. Each video contained 300 frames acquired at
$16.67$ frames per second with a $40\times$ objective, corresponding to an
observation time of approximately $18\,\mathrm{s}$. Particle trajectories were
extracted with the TrackMate plugin
\cite{ershov:phan:pylvanainen:etal:2022} in Fiji
\cite{schindelin:argandacarreras:etal:2012}. Approximately 150 particles were
tracked per video, corresponding to approximately 450 tracked particle
trajectories for each experimental condition.

Velocity data were pooled over the three videos recorded under the same
experimental condition.
{Unless explicitly stated otherwise, reported velocity values
are mean velocities.}
In the boxplots, the box spans the first to the third quartile, the horizontal
line marks the median, and the superimposed symbol marks the mean; observations
identified as outliers are displayed individually. Particle chains were
identified and counted manually from the videos. The chain-size distributions
were likewise pooled over the three videos for each condition. Since the chain
counts depend strongly on the experimental condition, the number of chains
represented in the histograms varies between panels.

\subsection{The $\times1$ condition}
\label{sec:supp-micromotor-low}

Figure~\ref{fig:supp-micromotor-low} shows representative brightfield frames,
velocity boxplots, and pooled chain-size distributions under the $\times1$
condition. The four micromotor sizes displayed distinct patterns of chain
formation, whereas the mean velocity generally increased as the nominal
magnet pull force increased.

For $\MPEG{4}$, particle chains were visible under all four magnet conditions.
The mean velocity increased from approximately
$10\,\mu\mathrm{m}\,\mathrm{s}^{-1}$ at $0.8\,\mathrm{N}$ to approximately
$100\,\mu\mathrm{m}\,\mathrm{s}^{-1}$ at $4.12\,\mathrm{N}$. The mean and
median were similar in the displayed boxplots, although individual
high-velocity observations occurred. Chains containing two or three particles
were most frequent. Chains of up to six particles were observed at the largest
nominal pull force.
The $\MPEG{10}$ micromotors showed a comparable increase in velocity. At
$0.8\,\mathrm{N}$, the observed chains were dimers. At the larger nominal pull
forces, dimers remained the most common configuration, while chains containing
three to five particles also occurred.
For $\MPEG{20}$, chain formation was much less frequent. No chains were
observed in the pooled counts at $0.8$ or $4.12\,\mathrm{N}$. At $1.7$ and
$2.9\,\mathrm{N}$, the observed chains were predominantly dimers, with a small
number of trimers. The total number of observed chains was substantially lower
than for $\MPEG{4}$ and $\MPEG{10}$. The mean velocity increased with nominal
pull force and was of the same order as for $\MPEG{10}$.
For $\MPEG{40}$, no chain formation was detected under any of the four magnet
conditions. Nevertheless, the micromotors moved faster as the nominal pull
force increased, with mean velocities rising from approximately
$80\,\mu\mathrm{m}\,\mathrm{s}^{-1}$ to approximately
$200\,\mu\mathrm{m}\,\mathrm{s}^{-1}$ across the displayed conditions.

Taken together, chain formation in the $\times1$ condition decreased markedly
with particle size. Robust chain formation was observed for $\MPEG{4}$ and
$\MPEG{10}$, whereas $\MPEG{20}$ formed only a limited number of predominantly
short chains and no chains were observed for $\MPEG{40}$. Across all particle
sizes, stronger nominal magnet conditions were associated with higher
velocities. By contrast, no clear monotone relation between nominal pull force
and chain length was apparent in the pooled $\times1$ data.

A possible explanation for the reduction in chain formation
with increasing particle size is the lower particle number density associated
with larger particles when the stock suspensions are specified at the same
weight fraction. This would reduce the frequency of particle encounters and
hence the opportunity for chain formation. The observation that $\MPEG{40}$
forms chains in the $\times10$ condition, whereas no chains are observed in the
$\times1$ condition, is consistent with this interpretation. Since the
absolute particle concentrations after assembly are not known, this should be
regarded only as a possible explanation. The reduction in chain formation may also be related to the individual
magnetic moment of motors of different sizes. For smaller motors, the relative
amount of deposited magnetic nanoparticles per particle core is much higher
than for larger motors. This means that the total magnetic moment of a smaller
motor is larger than that of a larger motor, which may make it easier to
establish dipolar interactions, i.e., magnetic particle--particle interactions,
between neighboring motors.

\begin{figure}[!htpb]
\centering
\includegraphics[width=\textwidth,height=.88\textheight,keepaspectratio]{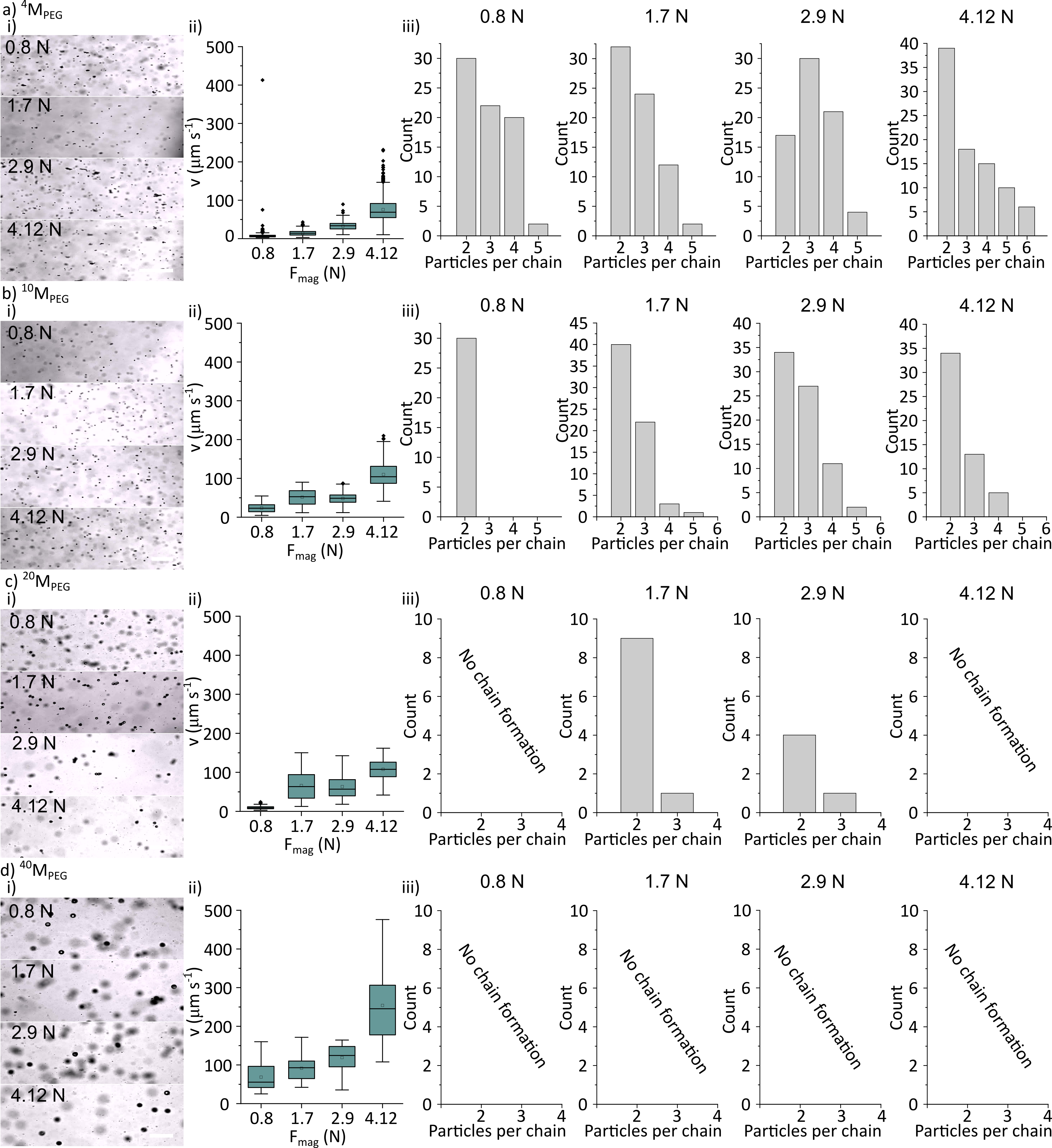}
\caption{Micromotor mobility under the $\times1$ condition. For $\MPEG{4}$ (a),
$\MPEG{10}$ (b), $\MPEG{20}$ (c), and $\MPEG{40}$ (d), the panels show
representative brightfield images (i), pooled velocity boxplots (ii), and
pooled distributions of the number of particles per chain (iii) under magnets
with manufacturer-specified nominal pull forces of $0.8$, $1.7$, $2.9$, and
$4.12\,\mathrm{N}$. The nominal pull-force values are labels for the magnet
conditions and are not forces acting directly on individual particles. Scale
bars: $200\,\mu\mathrm{m}$.}
\label{fig:supp-micromotor-low}
\end{figure}

\subsection{Additional results for the \(\times1\) condition}
\label{app:x1-results}

Figure~\ref{fig:x1-overlays-20um-app} shows the distributional
comparisons for the \(20\,\mu\mathrm m\) particles, complementing the
corresponding results for the \(4\,\mu\mathrm m\) and
\(10\,\mu\mathrm m\) particles shown in Section~\ref{sec:real-data}. As in
the main text, the simulated distribution for each experimental condition is
the empirical mixture over the experimentally observed chain lengths, with
mixture weights determined by the observed chain-length frequencies. The
upper row shows the experimental and fitted simulated velocity distributions,
while the lower row shows the corresponding empirical cumulative distribution
functions.

\begin{figure}[!htbp]
    \centering
    \includegraphics[width=.51\textwidth]
    {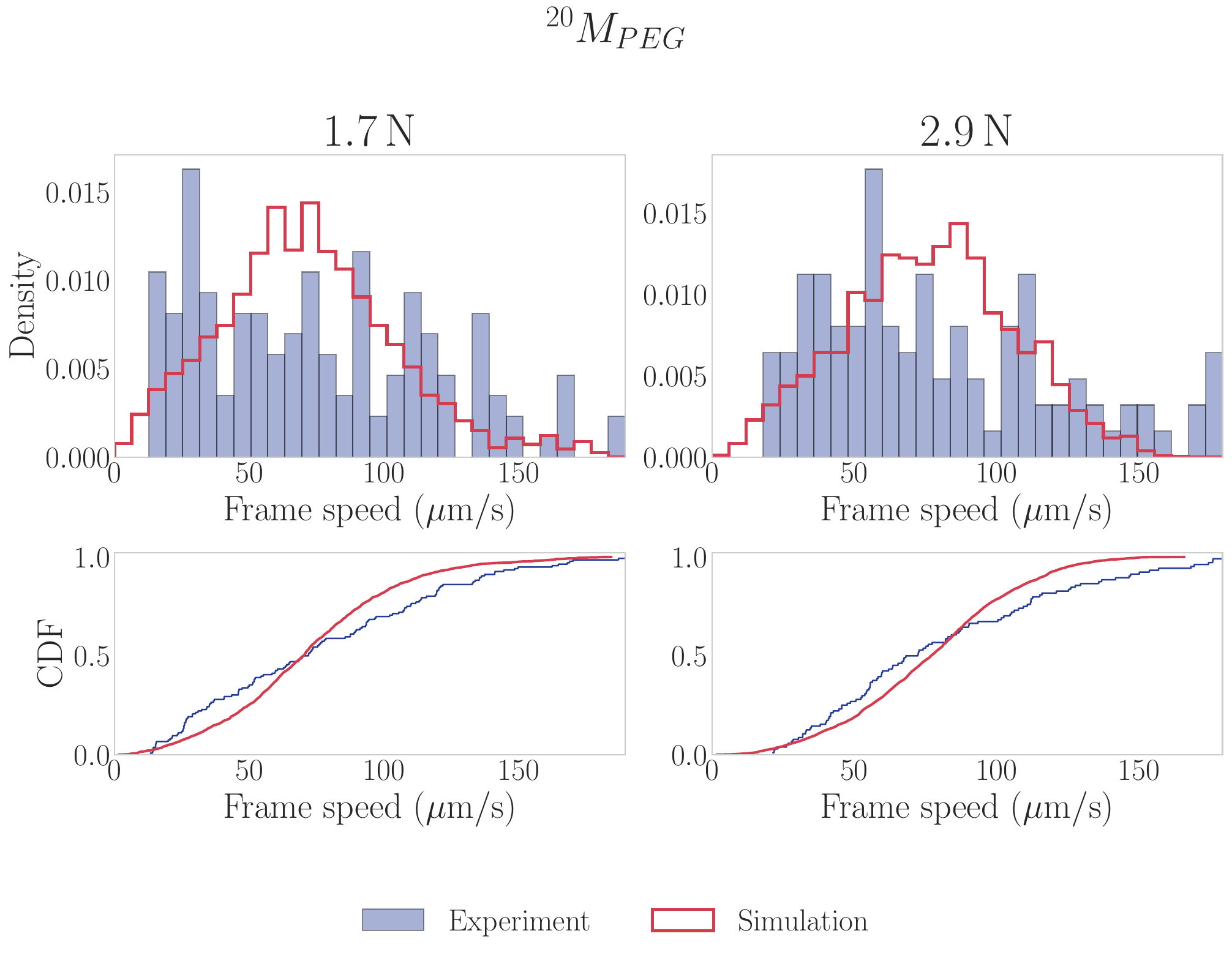}
    \caption{
    Experimental and fitted simulated velocity distributions for the
    \(\times1\) condition for \({}^{20}M_{\mathrm{PEG}}\) particles at nominal
magnetic pulling force values \(1.7\) N and \(2.9\) N.
    }
    \label{fig:x1-overlays-20um-app}
\end{figure}

\subsection{The $\times10$ condition}
\label{sec:supp-micromotor-high}

In the $\times10$ condition, chains were observed for all particle sizes; see
Figure~\ref{fig:supp-micromotor-high}. The effect of particle size remained
pronounced: $\MPEG{4}$ formed the longest chains, whereas $\MPEG{40}$ formed
almost exclusively dimers. Mean velocities again tended to increase with
nominal magnet pull force.

For $\MPEG{4}$, chains were visible under all four magnet conditions. The mean
velocity increased from  about
$20\,\mu\mathrm{m}\,\mathrm{s}^{-1}$ at $0.8\,\mathrm{N}$ to  about
$100\,\mu\mathrm{m}\,\mathrm{s}^{-1}$ at $4.12\,\mathrm{N}$. The pooled
chain-size distributions were broad. Chains containing up to 25 particles were
observed, with many chain sizes between two and fifteen particles. At
$4.12\,\mathrm{N}$, a substantial part of the distribution was concentrated
between approximately seven and twelve particles.
The $\MPEG{10}$ micromotors also formed chains under every magnet condition.
Their mean velocity increased with nominal pull force and reached approximately
$200\,\mu\mathrm{m}\,\mathrm{s}^{-1}$ at $4.12\,\mathrm{N}$. Their chains
were shorter than those formed by $\MPEG{4}$: most observed chains contained
between two and six particles, although a small number of longer chains
occurred under some conditions.
The $\MPEG{20}$ micromotors formed chains under all four magnet conditions, but
the chain-size distributions were strongly concentrated on dimers. Trimers and
occasional longer chains were also observed. The mean velocities increased
with nominal pull force and were comparable in scale to those of $\MPEG{10}$.
In the $\times10$ condition, the $\MPEG{40}$ micromotors also formed chains,
unlike in the $\times1$ condition. These chains were almost exclusively
dimers, with a smaller number of trimers. The mean velocity increased from
approximately $100\,\mu\mathrm{m}\,\mathrm{s}^{-1}$ at $0.8\,\mathrm{N}$ to
approximately $400\,\mu\mathrm{m}\,\mathrm{s}^{-1}$ at $4.12\,\mathrm{N}$.

Thus, in the $\times10$ condition, all four particle sizes exhibited chain
formation, but chain length decreased markedly with particle size.
$\MPEG{4}$ produced broad chain-size distributions and occasional very long
chains, whereas $\MPEG{40}$ predominantly formed dimers. Relative to the
$\times1$ condition, the displayed data show more frequent and longer chains,
particularly for the larger particle sizes. This is likely related to the fact that more particles mean more possible
interactions between neighbouring particles. The velocity distributions also
shifted towards larger values for several particle sizes.

\begin{figure}[!htpb]
\centering
\includegraphics[width=\textwidth,height=.88\textheight,keepaspectratio]{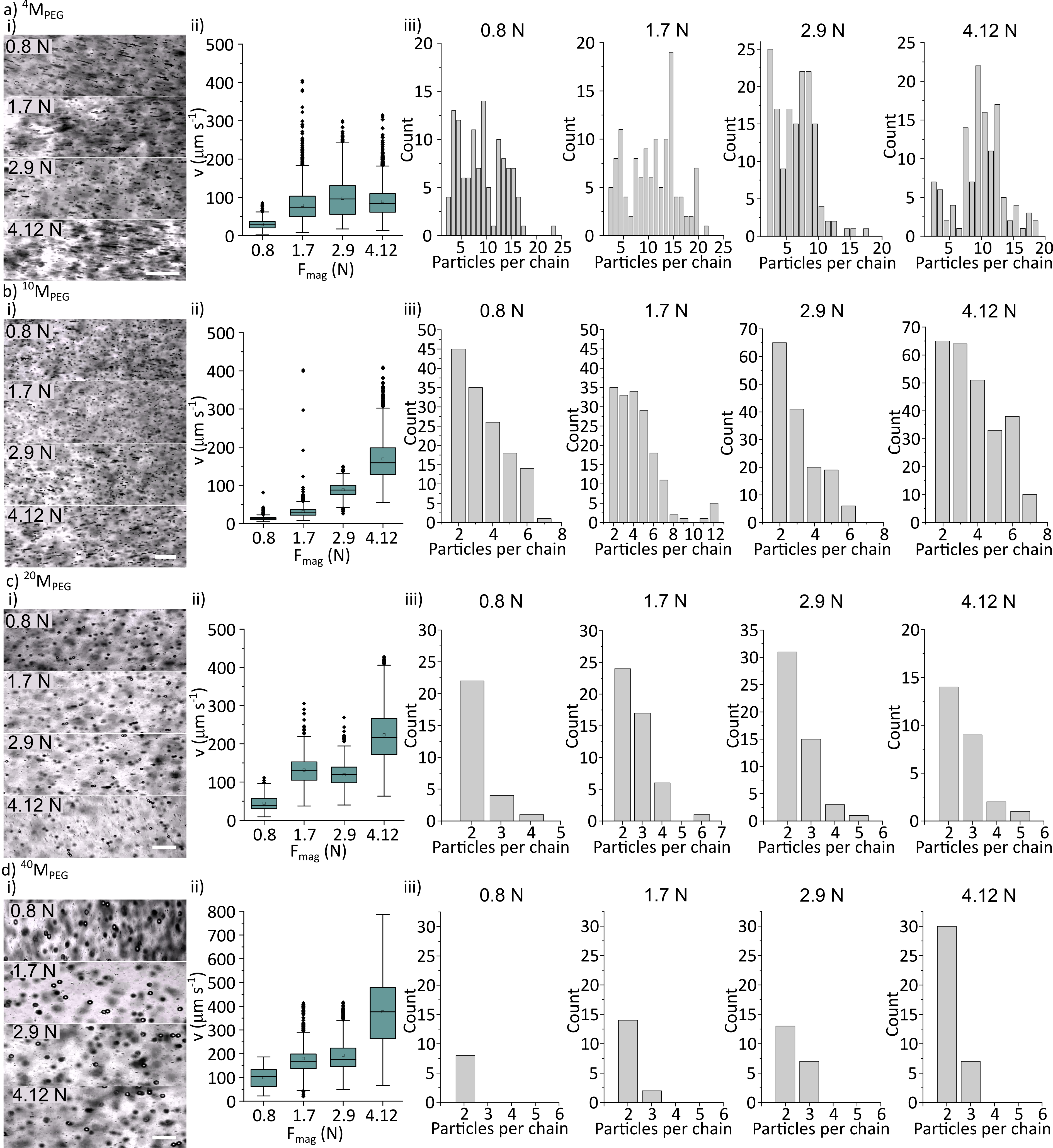}
\caption{Micromotor mobility under the $\times10$ condition. For $\MPEG{4}$
(a), $\MPEG{10}$ (b), $\MPEG{20}$ (c), and $\MPEG{40}$ (d), the panels show
representative brightfield images (i), pooled velocity boxplots (ii), and
pooled distributions of the number of particles per chain (iii) under magnets
with manufacturer-specified nominal pull forces of $0.8$, $1.7$, $2.9$, and
$4.12\,\mathrm{N}$. The nominal pull-force values are labels for the magnet
conditions and are not forces acting directly on individual particles. Scale
bars: $200\,\mu\mathrm{m}$.}
\label{fig:supp-micromotor-high}
\end{figure}

Figure~\ref{fig:10parameter-trends} shows the fitted characteristic velocity
\(v_0\) for the \(\times10\) condition across the different magnetic conditions
and particle sizes.

\begin{figure}[!htbp]
    \centering
    \includegraphics[width=0.65\textwidth]
    {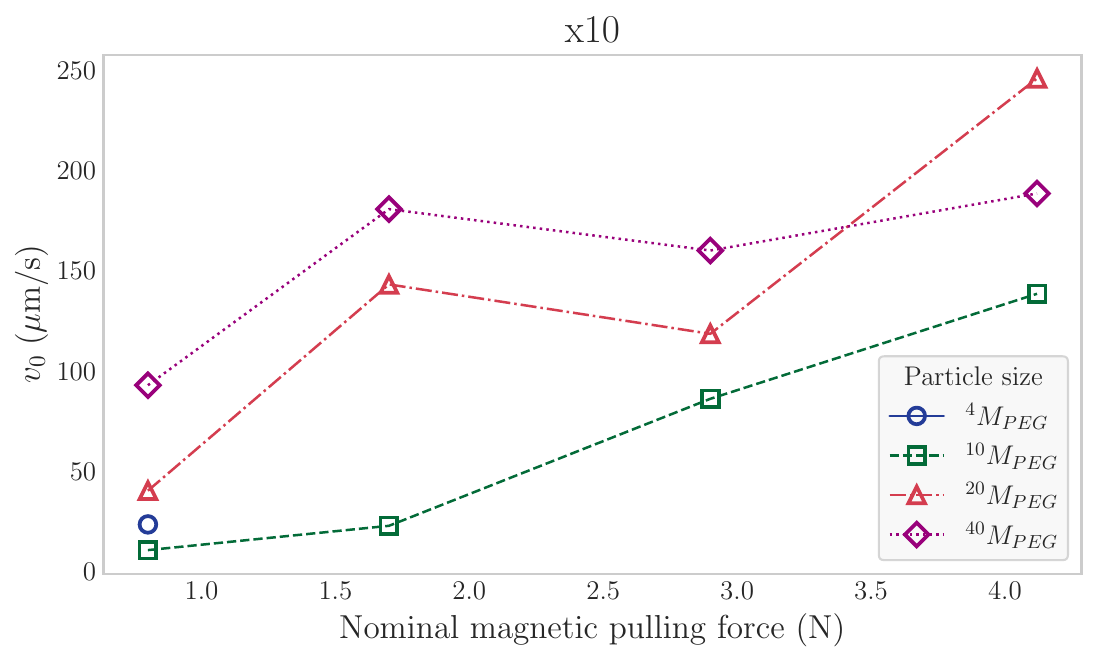}
    \caption{
    Fitted OU mean velocity \(v_0\) for the \(\times10\)
    condition as a function of    nominal magnetic pulling force, with separate curves
    corresponding to the different particle sizes.
    }
    \label{fig:10parameter-trends}
\end{figure}

\subsection{Additional results for the ×10 condition}
\label{app:x10-results}

We next report the corresponding analysis for the higher-concentration,
or \(\times10\), measurements. The same general procedure as in
Section~\ref{sec:real-data} was applied: the experimental velocity data were
processed condition by condition, the observed chain-length frequencies were
used to construct the simulated mixtures, and the stochastic-cluster model
was calibrated by minimizing the Kolmogorov--Smirnov distance between the
experimental and simulated velocity distributions. The calibration was
performed separately for each particle size and magnetic forcing value.
Accordingly, the fitted quantities should again be interpreted as
condition-specific effective parameters rather than as a single universal
microscopic parameter set.

Figures~\ref{fig:x10-overlays-size-4}--\ref{fig:x10-overlays-size-40} compare
the cleaned experimental velocity distributions with the fitted simulated
mixtures for all four particle sizes. For each condition, the simulated
distribution is the empirical mixture over the experimentally observed chain
lengths, with mixture weights determined by the observed chain-length
frequencies. The upper row shows the experimental and fitted simulated
velocity distributions, while the lower row shows the corresponding empirical
cumulative distribution functions.

\begin{figure}[!htbp]
    \centering
    \includegraphics[width=0.26\textwidth]
    {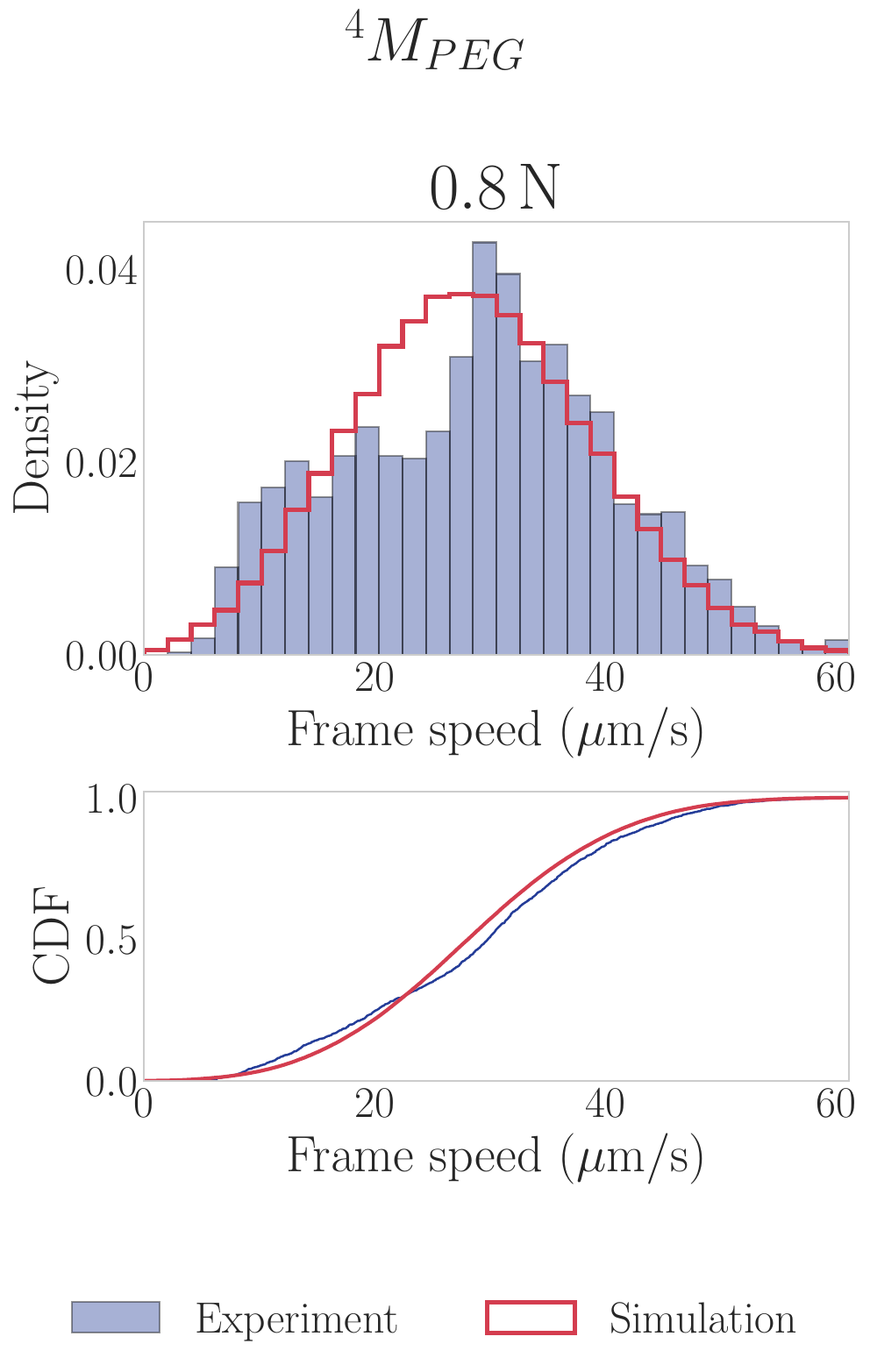}
    \caption{
    Experimental and fitted simulated velocity distributions for the
    \(\times10\) condition for \({}^{4}M_{\mathrm{PEG}}\) particles at nominal
magnetic pulling force values \(0.8\), \(1.7\), \(2.9\), and \(4.12\) N.
    }
    \label{fig:x10-overlays-size-4}
\end{figure}

\begin{figure}[!htbp]
    \centering
    \includegraphics[width=\textwidth]
    {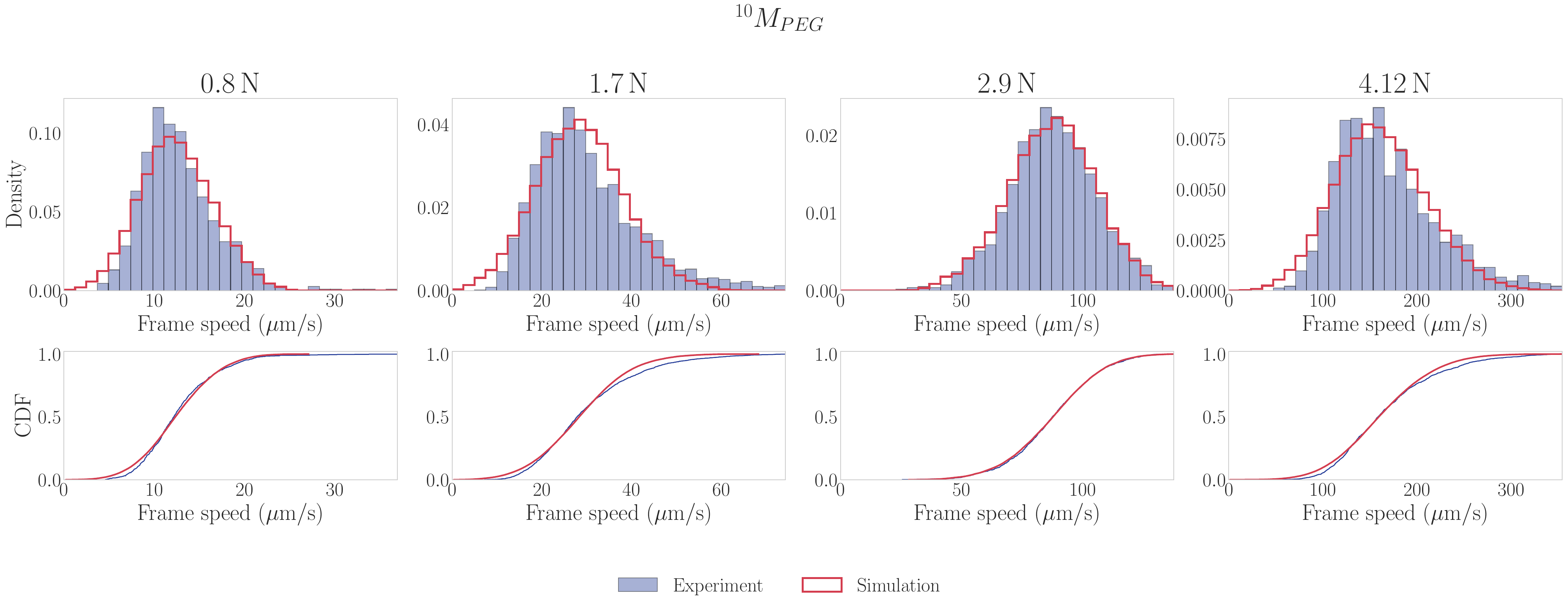}
    \caption{
    Experimental and fitted simulated velocity distributions for the
    \(\times10\) condition for \({}^{10}M_{\mathrm{PEG}}\) particles at nominal
magnetic pulling force values \(0.8\), \(1.7\), \(2.9\), and \(4.12\) N.
    }
    \label{fig:x10-overlays-size-10}
\end{figure}

\begin{figure}[!htbp]
    \centering
    \includegraphics[width=\textwidth]
    {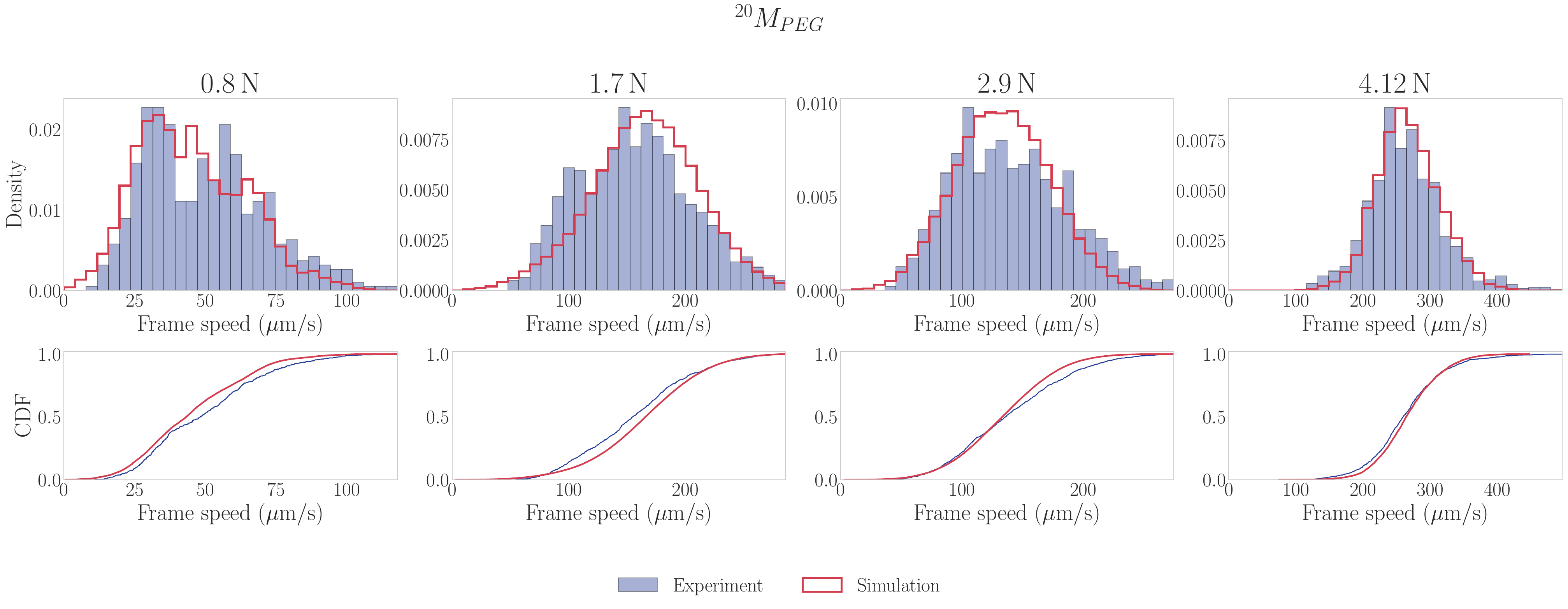}
    \caption{
    Experimental and fitted simulated velocity distributions for the
    \(\times10\) condition for \({}^{20}M_{\mathrm{PEG}}\) particles at nominal
magnetic pulling force values \(0.8\), \(1.7\), \(2.9\), and \(4.12\) N.
    }
    \label{fig:x10-overlays-size-20}
\end{figure}

\begin{figure}[!htbp]
    \centering
    \includegraphics[width=\textwidth]
    {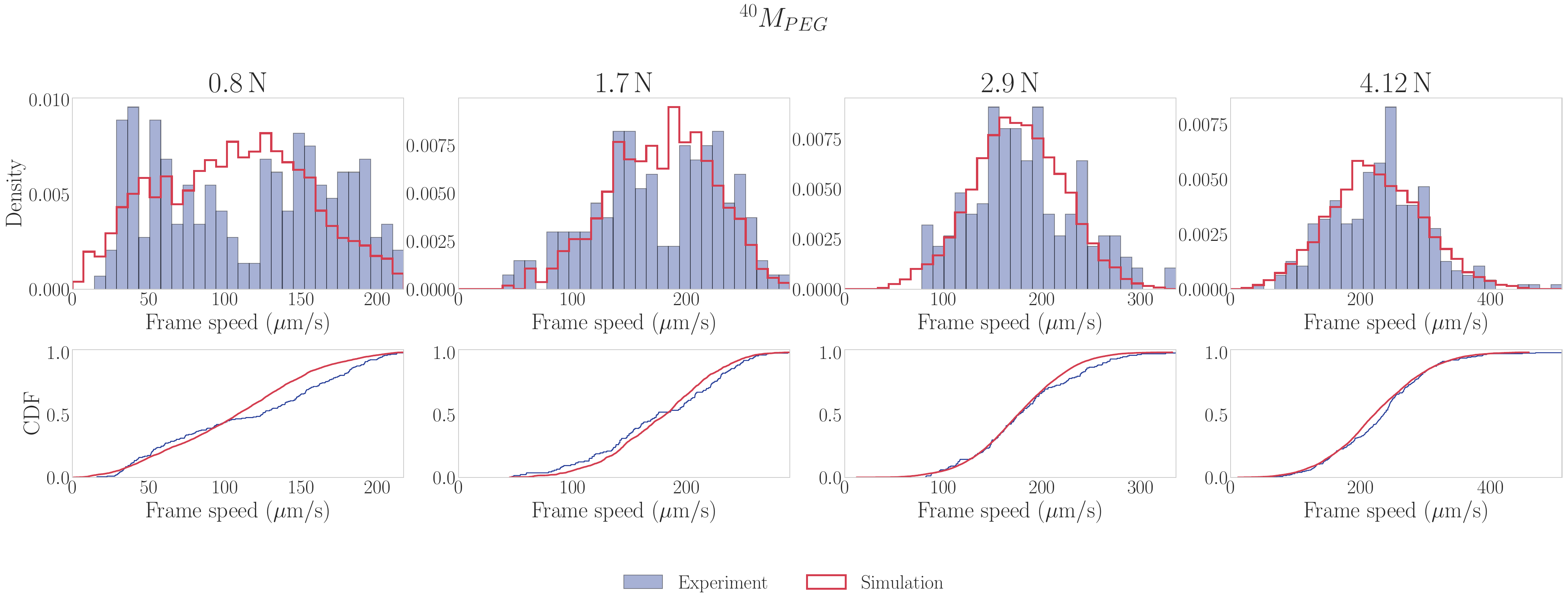}
    \caption{
    Experimental and fitted simulated velocity distributions for the
    \(\times10\) condition for \({}^{40}M_{\mathrm{PEG}}\) particles at nominal
magnetic pulling force values \(0.8\), \(1.7\), \(2.9\), and \(4.12\) N.
    }
    \label{fig:x10-overlays-size-40}
\end{figure}

\FloatBarrier

\section{Additional numerical details for the synthetic examples}
\label{app:synthetic-numerics}

This section records the numerical parameters used for the synthetic
experiments of Section~\ref{sec:ex}. These values specify the illustrative
Metropolis--Hastings simulations and are held fixed within those experiments.
They are distinct from the condition-specific calibration parameters used for
the experimental analysis in Section~\ref{sec:real-data}.

\begin{table}[htb!]
\caption{Fixed parameters used in the synthetic numerical experiments of
Section~\ref{sec:ex}.}
\label{tab:simulation_parameters}
\centering
\small
\begin{tabularx}{\textwidth}{@{} l c X @{}}
\toprule
Parameter & Value & Meaning \\
\midrule
\(k\) (code: \texttt{K}) & \(3\) & Number of particles in each cluster. \\
\(N_t\) & \(250\) & Number of temporal discretization points. \\
\(\Delta t\) & \(1/16.67\) & Time step between consecutive points. \\
\(\kappa\) (code: \texttt{kappa\_flow}) & \(5\) &
Mean-reversion rate of the Ornstein--Uhlenbeck velocity. \\
\(\sigma\) (code: \texttt{sigma\_flow}) & \(1\) &
Noise amplitude of the Ornstein--Uhlenbeck velocity. \\
\(\phi_{\mathrm{flow}}\) & \(0\) &
Direction of the mean Ornstein--Uhlenbeck velocity. \\
\(\beta\) & \(20\) & Inverse-temperature parameter. \\
\(h\) & \(2.4\) & Field-energy coefficient. \\
\(\epsilon_{\mathrm{pair}}\) & \(15\) & Pair-interaction coefficient. \\
\(\widehat b\) & \((1,0)\) & Direction of the external field. \\
\(L\) & \(16\) & Maximum separation in the cluster condition. \\
\(r_{\max}\) & \(6\) & Radial cutoff for the pair interaction. \\
\(\theta_0\) & \(\pi/6\) & Half-angle of the directional interaction cone. \\
\(\sigma_x\) & \(1.2\) & Standard deviation of the offset proposal. \\
\(\ell^X_{\min},\ell^X_{\max}\) & \(3,20\) &
Minimum and maximum block lengths for local Brownian pCN proposals. \\
\(\ell^V_{\min},\ell^V_{\max}\) & \(8,40\) &
Minimum and maximum block lengths for local Ornstein--Uhlenbeck pCN proposals. \\
\(\rho_{\mathrm{pCN}}\) & \(0.985\) &
Correlation parameter of the pCN proposals. \\
\(N_{\mathrm{steps}}\) & \(1.5\times10^5\) &
Total Metropolis--Hastings iterations. \\
\(N_{\mathrm{burn}}\) & \(6.5\times10^4\) &
Iterations discarded as burn-in. \\
\(q\) & \(10\) & Thinning interval between retained samples. \\
\bottomrule
\end{tabularx}
\end{table}

Table~\ref{tab:ks-x1}  reports the Kolmogorov--Smirnov
distances for the fitted velocity distributions under the $\times 1$ and
$\times 10$ conditions, respectively. The KS distances are generally small,
indicating that the fitted model captures the experimental velocity
distributions reasonably well. The largest discrepancies are observed for
the $20M_{\mathrm{PEG}}$  particles under the $\times 1$ condition and the $40M_{\mathrm{PEG}}$  particles
under the $\times 10$ condition.

\begin{table}[ht]
    \centering
    \caption{KS distances for the fitted velocity distributions
    under the $\times 1$ condition (left) and under the $\times 10$ condition (right).}
    \label{tab:ks-x1}
    \begin{tabular}{c|cccc}
        \hline
        & \multicolumn{4}{c}{Nominal magnetic pull force (N)} \\
        $M_{\mathrm{PEG}}$ 
        & 0.8 & 1.7 & 2.9 & 4.12 \\
        \hline
        $4M_{\mathrm{PEG}}$  
        & 0.081 & 0.051 & 0.059 & 0.055 \\
        $10M_{\mathrm{PEG}}$ 
        & 0.128 & 0.101 & 0.033 & 0.049 \\
        $20M_{\mathrm{PEG}}$ 
        & --    & 0.184 & 0.153 & --    \\
        \hline
    \end{tabular}\qquad
        \begin{tabular}{c|cccc}
        \hline
        & \multicolumn{4}{c}{Nominal magnetic pull force (N)} \\
        $M_{\mathrm{PEG}}$ 
        & 0.8 & 1.7 & 2.9 & 4.12 \\
        \hline
        $4M_{\mathrm{PEG}}$  
        & 0.052 & --    & --    & --    \\
        $10M_{\mathrm{PEG}}$ 
        & 0.044 & 0.052 & 0.025 & 0.052 \\
        $20M_{\mathrm{PEG}}$ 
        & 0.101 & 0.046 & 0.071 & 0.039 \\
        $40M_{\mathrm{PEG}}$ 
        & 0.157 & 0.140 & 0.058 & 0.063 \\
        \hline
    \end{tabular}
\end{table}



\end{document}